\documentclass[11pt]{amsart}
\usepackage{times}

\usepackage{bbm}

\usepackage[hang]{footmisc} 

\usepackage{varioref}
\usepackage[bookmarks=false]{hyperref}
\usepackage{cleveref}

\usepackage{amsthm, amssymb, amsmath, a4wide}
\usepackage[inline]{enumitem}
\usepackage{graphicx} 

\usepackage{tikz}
\usetikzlibrary{calc} 
\usepackage{subfig}

\usepackage{footnote} 

\usepackage{bm} 
\usepackage{dsfont} 

\usepackage{mathrsfs}

\usepackage{amsfonts}

\newcommand{\overbar}[1]{\mkern 1.5mu\overline{\mkern-1.5mu#1\mkern-1.5mu}\mkern 0.5mu}

\usepackage{xcolor}
\hypersetup{
    colorlinks,
    linkcolor={red!50!black},
    citecolor={blue!50!black},
    urlcolor={blue!80!black}
}

\author{Pinaki Mondal}
\address{University of the Bahamas}
\email{pinakio@gmail.com}

\title{Intersection multiplicity, Milnor number and Bernstein's theorem}

\makeatletter

\newcommand{\Rmnum}[1]{\expandafter\@slowromancap\romannumeral #1@}
\makeatother

\DeclareMathOperator\conv{conv}

\DeclareMathOperator\ord{ord}

\DeclareMathOperator\supp{Supp}
\DeclareMathOperator\vol{Vol} 

\newcommand{\scrA}{\ensuremath{\mathcal{A}}}
\newcommand{\scrB}{\ensuremath{\mathcal{B}}}

\newcommand{\scrD}{\ensuremath{\mathcal{D}}}
\newcommand{\scrE}{\ensuremath{\mathcal{E}}}

\newcommand{\scrG}{\ensuremath{\mathcal{G}}}

\newcommand{\scrI}{\ensuremath{\mathcal{I}}}

\newcommand{\scrL}{\ensuremath{\mathcal{L}}}

\newcommand{\scrN}{\ensuremath{\mathcal{N}}}

\newcommand{\scrP}{\ensuremath{\mathcal{P}}}

\newcommand{\scrT}{\ensuremath{\mathcal{T}}}
\newcommand{\scrU}{\ensuremath{\mathcal{U}}}
\newcommand{\scrV}{\ensuremath{\mathcal{V}}}

\newcommand{\scrZ}{\ensuremath{\mathcal{Z}}}

\newcommand{\kk}{\ensuremath{\mathbb{K}}}

\newcommand{\pp}{\ensuremath{\mathbb{P}}}

\newcommand{\rr}{\ensuremath{\mathbb{R}}}

\newcommand{\zz}{\ensuremath{\mathbb{Z}}}

\newcommand{\mmm}{\ensuremath{\mathfrak{m}}}
\newcommand{\nnn}{\ensuremath{\mathfrak{n}}}

\newcommand{\jjj}{\ensuremath{\mathfrak{j}}}

\newcommand{\qqq}{\ensuremath{\mathfrak{q}}}

\newcommand{\nktorus}{(\kk^*)^n}
\newcommand{\ktorus}{\kk^*}

\newcommand{\im}{\ensuremath{\Rightarrow}}

\newtheorem{thm}{Theorem}[section]
\newtheorem*{thm*}{Theorem}
\newtheorem{lemma}[thm]{Lemma}
\newtheorem*{lemma*}{Lemma}

\newtheorem{prop}[thm]{Proposition}
\newtheorem*{prop*}{Proposition}
\newtheorem{cor}[thm]{Corollary}
\newtheorem{claim}[thm]{Claim}
\newtheorem*{claim*}{Claim}
\newtheorem{proclaim}{Claim}[thm]
\newtheorem*{conjecture*}{Conjecture}

\theoremstyle{definition}

\newtheorem*{constrinition*}{Construction-Definition}
\newtheorem{convention}[thm]{Convention}
\newtheorem*{convention*}{Convention}
\newtheorem{defn}[thm]{Definition}
\newtheorem*{defn*}{Definition}

\newtheorem*{definotation*}{Definition-Notation}
\newtheorem{example}[thm]{Example}
\newtheorem*{example*}{Example}

\newtheorem*{fact*}{Fact}

\newtheorem*{facts*}{Facts}

\newtheorem{notation}[thm]{Notation}

\newtheorem*{bold-note*}{Note}
\newtheorem{problem}[thm]{Problem}
\newtheorem*{problem*}{Problem}
\newtheorem{bold-question}[thm]{Question}
\newtheorem*{bold-question*}{Question}
\newtheorem{rem}[thm]{Remark}

\newtheorem*{reminition*}{Remark-Definition}
\newtheorem{remexample}[thm]{Remark-Example}
\newtheorem{remexample*}{Remark-Example}

\newtheorem*{remtation*}{Remark-Notation}

\newtheorem*{remuestion*}{Remark-Question}

\newtheorem*{remvention*}{Remark-Convention}

\theoremstyle{remark}
\newtheorem*{rem*}{Remark}
\newtheorem*{note*}{Note}
\newtheorem*{notation*}{Notation}
\newtheorem*{question*}{Question}

\newtheorem*{questions*}{Questions}

\newcounter{UnorderedProofTempCtr}
\newcommand{\tempcommand}{}

\renewcommand{\kk}{\mathbbm{k}}

\setitemize[2]{label=$\triangleright$}

\newcommand{\nocontentsline}[3]{}
\newcommand{\tocless}[3]{\bgroup\let\addcontentsline=\nocontentsline#1{#2}\label{#3}\egroup}

\newcommand{\mscrG}{\mathscr{G}}
\newcommand{\mscrP}{\mathscr{P}}
\newcommand{\mscrS}{\mathscr{S}}

\newcommand{\ISP}{\scrI_{\mscrS,\mscrP}}
\newcommand{\ISPone}{\scrI_{\mscrS,\mscrP,1}}
\newcommand{\nG}{\scrN_{\mscrG}}
\newcommand{\Ni}{\scrN^{I}}

\newcommand{\NiG}{\scrN^{I}_{\mscrG}}

\newcommand{\NiP}{\scrN^{I}_{\mscrP}}
\newcommand{\NiiP}[1]{\scrN^{#1}_{\mscrP}}
\newcommand{\NSP}{\scrN_{\mscrS,\mscrP}}

\newcommand{\TSP}{\scrT_{\mscrS,\mscrP}}

\newcommand{\TSPrime}{\scrT'_{\mscrS,\mscrP}}

\newcommand{\TSPrimeI}{{\scrT'}^I_{\mscrS,\mscrP}}

\newcommand{\TSPstar}{\scrT^*_{\mscrS,\mscrP}}
\newcommand{\TI}{\scrT^{I}}

\newcommand{\KnS}{\kk^n_\mscrS}
\newcommand{\KnbarS}{\kk^n_{\overbar \mscrS}}
\newcommand{\KnSprime}{\kk^n_{\mscrS'}}
\newcommand{\KnSS}[1]{\kk^n_{#1}}

\newcommand{\SP}{\scrE_\mscrP}

\newcommand{\tildeTP}{{\tilde \scrT}_\mscrP}
\newcommand{\tildeTprimeP}{{\tilde \scrT'}_\mscrP}

\newcommand{\dfd}[1]{\partial f/\partial x_{#1}}
\newcommand{\dgd}[1]{\partial g/\partial x_{#1}}
\newcommand{\dGammad}[1]{\partial \Gamma/\partial x_{#1}}
\newcommand{\dx}[2]{\partial {#1}/\partial {#2}}

\newcommand{\mult}[2]{[#1, \ldots, #2]}
\newcommand{\multf}{\mult{f_1}{f_n}}
\newcommand{\multP}{\mult{\scrP_1}{\scrP_n}}

\newcommand{\multiso}[2]{[#1, \ldots, #2]^{iso}}
\newcommand{\multisof}{\multiso{f_1}{f_n}}
\newcommand{\multisoP}{\multiso{\scrP_1}{\scrP_n}}
\newcommand{\multzero}[2]{[#1, \ldots, #2]_0}
\newcommand{\multzerof}{\multzero{f_1}{f_n}}
\newcommand{\multzeroGamma}{\multzero{\Gamma_1}{\Gamma_n}}

\newcommand{\multstar}[2]{[#1, \ldots, #2]^*_{0}}
\newcommand{\multstarinftySS}[3]{[#1, \ldots, #2]^*_{\infty,#3}}
\newcommand{\multstarinftyS}{\multstarinftySS{\scrP_1}{\scrP_n}{\mscrS}}
\newcommand{\multstarinftyy}[2]{[#1, \ldots, #2]^*_{\infty}}
\newcommand{\multstarSS}[2]{[#1, \ldots, #2]^*_{\mscrS}}
\newcommand{\multstarS}{\multstarSS{\scrP_1}{\scrP_n}}
\newcommand{\multstarinfty}{\multstarinftyy{\scrP_1}{\scrP_n}}

\newcommand{\Xnu}{X_\nu}

\newcommand{\Xprimenuprime}{X'_{\nu'}}

\newcommand{\Enu}{E_\nu}

\newcommand{\Enuprime}{E_{\nu'}}

\newcommand{\Enustar}{E^*_\nu}

\newcommand{\nuB}{\nu_B}
\newcommand{\nuBi}{\nu^I_B}
\newcommand{\nuBii}[1]{\nu^{#1}_B}

\newcommand{\Bzeroi}{\scrB^I_0}
\newcommand{\Bzeroinu}{\scrB^I_{0,\nu}}
\newcommand{\Bi}{\scrB^I}
\newcommand{\Binu}{\scrB^I_{\nu}}
\newcommand{\Biomega}{\scrB^I_{\omega}}
\newcommand{\Zzeroi}{\scrZ_0^I}
\newcommand{\Zi}{\scrZ^I}

\newcommand{\Ai}{A^I}
\newcommand{\Aii}[1]{A^{#1}}

\newcommand{\Kii}[1]{\kk^{#1}}
\newcommand{\ki}{k^I}
\newcommand{\Ki}{\kk^I}

\newcommand{\Kstari}{(\kk^*)^I}
\newcommand{\Kstarii}[1]{(\kk^*)^{#1}}

\newcommand{\ri}{\rr^I}

\newcommand{\V}{\scrV}
\newcommand{\Vi}{\scrV^I}
\newcommand{\Vii}[1]{\scrV^{#1}}
\newcommand{\Vzero}{\scrV_0}
\newcommand{\Vzeroi}{\scrV^I_0}
\newcommand{\Vzeroii}[1]{\scrV^{#1}_0}
\newcommand{\VS}{\scrV_\mscrS}
\newcommand{\VSi}{\scrV^I_\mscrS}
\newcommand{\VSprimei}{\scrV^I_{\mscrS'}}
\newcommand{\VSii}[1]{\scrV^{#1}_\mscrS}
\newcommand{\VSiii}[2]{\scrV^{#1}_{#2}}

\newcommand{\wts}{\Omega_\infty}
\newcommand{\wtsi}{\Omega^{I}_\infty}
\newcommand{\wtsii}[1]{\Omega^{#1}_\infty}

\DeclareMathOperator\charac{char}

\DeclareMathOperator\In{In}
\DeclareMathOperator\lcm{lcm}
\DeclareMathOperator\ld{\scrL}

\DeclareMathOperator\mv{MV}
\DeclareMathOperator{\nd}{\scrN\scrD}
\DeclareMathOperator\np{\scrN\scrP}

\newlist{observations}{enumerate}{2}
\setlist[observations,1]{label=(O\arabic*), ref=O\arabic*}
\setlist[observations,2]{label=(\emph{\alph*}), ref=\theobservationsi(\emph{\alph*})}
\crefalias{observationsi}{observation} 
\crefalias{observationsii}{observation} 

\newlist{assumptions}{enumerate}{2}
\setlist[assumptions,1]{label=(A\arabic*), ref=A\arabic*}
\setlist[assumptions,2]{label=(\emph{\alph*}), ref=\theassumptionsi(\emph{\alph*})}
\crefalias{assumptionsi}{assumption} 
\crefalias{assumptionsii}{assumption} 

\newlist{defnlist}{enumerate}{3}
\setlist[defnlist,1]{label=(\alph*)}
\setlist[defnlist,2]{label=(\arabic*), ref=(\alph{defnlisti}.\arabic*)}
\setlist[defnlist,3]{label=(\roman*), ref=(\alph{defnlisti}.\arabic{defnlistii}.\roman*)}

\newlist{prooflist}{enumerate}{3}
\setlist[prooflist,1]{label=(\roman*)}
\setlist[prooflist,2]{label=(\arabic*)}
\setlist[prooflist,3]{label=(\alph*)}

\newcounter{parcounter}

\graphicspath {{pictures/}} 

\crefname{cor}{corollary}{corollaries}
\crefname{defn}{definition}{definitions}
\crefname{observation}{observation}{observations}
\crefname{problem}{problem}{problems}
\crefname{reminition}{remark-definition}{remark-definition}
\crefname{thm}{theorem}{theorems}

\def\*#1{\bm{#1}}

\begin{document}

\begin{abstract} 
 
We explicitly characterize when the Milnor number at the origin of a polynomial or power series (over an algebraically closed field $\kk$ of arbitrary characteristic) is the minimum of all polynomials with the same Newton diagram, which completes works of Kushnirenko \cite{kush-poly-milnor} and Wall \cite{wall}. Given a fixed collection of $n$ convex integral polytopes in $\rr^n$, we also give an explicit characterization of systems of $n$ polynomials supported at these polytopes which have the maximum number (counted with multiplicity) of isolated zeroes on $\kk^n$, or more generally, on the complement of a union of coordinate subspaces of $\kk^n$; this completes the program (undertaken by many authors including Khovanskii \cite{khovanus}, Huber and Sturmfels \cite{hurmfels-bern}, Rojas \cite{rojas-toric}) of the extension to $\kk^n$ of Bernstein's theorem \cite{bern} on number of solutions of $n$ polynomials on $\nktorus$. Our solutions to these two problems are connected by the computation of the intersection multiplicity at the origin of $n$ hypersurfaces determined by $n$ generic polynomials.  
\end{abstract}

\maketitle


%
%
%

\maketitle 

\tocless\section{Introduction}{}
\addtocontents{toc}{\protect\setcounter{tocdepth}{1}} 
There is a vast literature on the theory of {\em Newton polyhedra}, initiated by V.\ I.\ Arnold's hypothesis that `reasonable' invariants of objects (e.g.\ singularities, varieties, etc) associated to a `typical' (system of) analytic function(s) or polynomial(s) should be computable in terms of their {\em Newton diagrams} or {\em Newton polytopes} (\cref{Newton-definition}). In this article we revisit two of the original questions that shaped this theory, namely the question of computing the Milnor number\footnote{\label{milnor-footnote}Let $f \in \kk[x_1, \ldots, x_n]$, where $\kk$ is an algebraically closed field. Then the {\em Milnor number} $\mu(f)$ of $f$ at the origin is the dimension (as a vector space over $\kk$) of the quotient ring of $\kk[[x_1, \ldots, x_n]]$ modulo the ideal generated by partial derivatives of $f$ with respect to $x_j$'s.} of the singularity at the origin of the hypersurface determined by a generic polynomial or power series, and the question of computing the number (counted with multiplicity) of isolated zeroes of $n$ generic polynomials in $n$ variables. The first question was partially solved in a classical work of Kushnirenko \cite{kush-poly-milnor} and a subsequent work of Wall \cite{wall}; Bernstein \cite{bern}, following the work of Kushnirenko \cite{kush-poly-milnor}, solved the second question for the `torus' $\nktorus$ (where $\kk$ is an algebraically closed field and $\kk^* := \kk \setminus \{0\}$), and many other authors (including Khovanskii \cite{khovanus}, Huber and Sturmfels \cite{hurmfels-bern}, Rojas \cite{rojas-toric}) gave partial solutions for the case of the affine space $\kk^n$. Extending the approach from Bernstein's proof in \cite{bern} of his theorem, we give a complete solution to the first problem, and complete the program of extending Bernstein's theorem to $\kk^n$ (or more generally, to the complement of a union of coordinate subspaces of $\kk^n$). \\


In \cite{kush-poly-milnor} Kushnirenko gave a beautiful expression for a lower bound on the generic Milnor number and showed that a polynomial (or power series) attains this bound in the case that its Newton diagram is {\em convenient}\footnote{Kushnirenko used the term {\em commode} in French; `convenient' is also widely used, see e.g.\ \cite{boubakri-greuel-markwig}.} (which means that the Newton diagram contains a point on each coordinate axis), and it is {\em Newton non-degenerate}, i.e.\ the following is true (see \cref{inndefinition} for a precise formulation):
\begin{savenotes}
\begin{align}
\parbox{.84\textwidth}{%
for each {\em weighted order}\footnote{A {\em weighted order} corresponding to weights $(\nu_1, \ldots, \nu_n) \in \zz^n$ is the map $\nu:\kk[x_1, \ldots, x_n] \to \zz$ given by $\nu(\sum a_\alpha x^\alpha ) := \min\{\sum_{k=1}^n \alpha_k\nu_k: a_\alpha \neq 0\}$.} $\nu$ on $\kk[x_1, \ldots, x_n]$ with positive weights, the partial derivatives of the corresponding {\em initial form}\footnote{Given a weighted order $\nu$, the {\em initial form} of $f = \sum_\alpha a_\alpha x^\alpha \in \kk[x_1, \ldots, x_n]$ is the sum of all $a_\alpha x^\alpha$ such that $\nu(x^\alpha) = \nu(f)$.} of $f$ do not have any common zero on $\nktorus$.}  \tag{NND}
\end{align}
\end{savenotes}
However, it is straightforward to construct examples which show that 
\begin{prooflist}
\item \label{isolated-defect} if the Newton diagram is not convenient, Newton non-degeneracy of $f$ does not imply `finite determinacy' of $f$, i.e.\ it does not  imply that the origin is an isolated singular point of $f = 0$ (take e.g.\ $f := x_1^q \cdots x_n^q$, where $q$ is an integer $\geq 2$ not divisible by $\charac(\kk)$), 
\item \label{necessary-defect} even if the Newton diagram is convenient, Newton non-degeneracy is not necessary for the Milnor number to be generic (see \cref{Milnor-not-Newton}),
\item \label{existential-defect} if $\charac(\kk) > 0$, then there are (convenient) diagrams $\Gamma$ such that no polynomial with Newton diagram $\Gamma$ is Newton non-degenerate (see \cref{no-(I)NND}).
\end{prooflist}
Wall \cite{wall} found another sufficient criterion, called {\em inner Newton non-degeneracy} by Boubakri, Greuel and Markwig \cite{boubakri-greuel-markwig}, for the Milnor number to be generic. He showed that inner Newton non-degeneracy implies finite determinacy, correcting defect \ref{isolated-defect}, and that Kushnirenko's formula computes the Milnor number for inner Newton non-degenerate polynomials. However, inner Newton non-degeneracy also suffers from defects \ref{necessary-defect} and \ref{existential-defect} (see \cref{Milnor-not-inner,no-(I)NND}), and it remained a topic of ongoing investigations to completely understand what makes the Milnor number generic (see e.g.\ the works of Boubakri, Greuel and Markwig \cite{boubakri-greuel-markwig}, Greuel and Nguyen \cite{greuel-nguyen}). \\

In \cite{kush-poly-milnor} Kushnirenko proved another beautiful result that the number (counted with multiplicity) of isolated solutions on $\nktorus$ of $n$ generic (Laurent) polynomials $f_1, \ldots, f_n$ with identical Newton polytope $\scrP$ is bounded by $n! \vol(\scrP)$, and that the bound is attained if and only if the following non-degeneracy condition is satisfied (in which case we say that $f_1, \ldots, f_n$ are {\em BKK non-degenerate})
\begin{align}
\parbox{.72\textwidth}{%
for each non-trivial weighted order $\nu$ on $\kk[x_1, \ldots, x_n]$, the corresponding initial forms of $f_1, \ldots, f_n$ do not have any common zero on $\nktorus$.}  \tag{BKK} \label{bkk}
\end{align}
Bernstein \cite{bern} removed the restriction of $f_i$'s having identical Newton polytopes. He showed that the general upper bound for the number (counted with multiplicity) of isolated solutions on $\nktorus$ is the {\em mixed volume} (\cref{mixed-defn}) of Newton polytopes of $f_1, \ldots, f_n$, and that the bound is attained if and only if $f_1, \ldots, f_n$ are BKK non-degenerate. A number of authors have worked on extending Bernstein's result, also known as the Bernstein-Kushnirenko-Khovanskii bound (in short {\em BKK bound}), to the case of $\kk^n$; in particular, Khovanskii \cite{khovanus} gave a formula for the extended BKK bound in the special case that the Newton polytopes are such that generic systems have no zeroes in any (proper) coordinate subspace of $\kk^n$; Huber and Sturmfels \cite{hurmfels-bern} introduced the notion of {\em stable mixed volumes} and used it to give a formula in the case that  the Newton polytopes are such that generic systems have no non-isolated zeroes in $\kk^n$; Rojas \cite{rojas-toric} extended the method of Huber and Sturmfels to give a formula in terms of stable mixed volumes that works without any restriction on Newton polytopes\footnote{The formulae that appear in \cite[Theorem I]{rojas-wang} and \cite[Affine Point Theorem II]{rojas-toric} for the number of roots and intersection multiplicity are incorrect - see \cref{counter-section}. However the formula from \cite[Corollary 1]{rojas-toric} in terms of stable mixed volumes is correct - this can be seen either by directly adapting the proof of the result of Huber and Sturmfels \cite{hurmfels-bern} to arbitrary Newton polytopes and arbitrary characteristics, or by observing that our formula (\cref{bkk-thm}), which works for all Newton polytopes and characteristics, is in the case of zero characteristic equivalent to the formula in \cite[Corollary 1]{rojas-toric}.}. However, the conditions for the attainment of the bound had not been characterized - only some sufficient conditions were known, see e.g.\ \cite[Section 4]{khovanus}, \cite[Main Theorem 2]{rojas-toric}. 
%

\subsection{Main results}{}
All the results of this article are valid over an algebraically closed field $\kk$ of arbitrary characteristic. 

\subsubsection{Intersection multiplicity} Since the Milnor number of the singularity at the origin of $f = 0$ is simply the {\em intersection multiplicity} (\cref{intersection-multiplicity}) at the origin of the partial derivatives of $f$, the problem of understanding genericness of the Milnor number naturally leads to the problem of understanding genericness of intersection multiplicity at the origin of hypersurfaces determined by $n$ polynomials $f_1, \ldots, f_n$ with fixed Newton diagrams. We give a complete solution to this problem:

\begin{thm}[\Cref{generic-thm}] \label{generic-thm-0}
The intersection multiplicity is the minimum if and only if the following non-degeneracy condition holds:
\begin{align}
\parbox{.72\textwidth}{%
for each weighted order $\nu$ on $\kk[x_1, \ldots, x_n]$ with positive weights and for each coordinate subspace $K$ of $\kk^n$,  the corresponding initial forms of $f_1|_K, \ldots, f_n|_K$ do not have any common zero on $\nktorus$.}  \tag{$\overline{\text{BKK}}_0$} \label{bkk-bar-0}
\end{align}
\end{thm}

We also give a new formula for the minimum intersection multiplicity (\cref{multiplicity-thm}). Given a fixed collection of $n$ diagrams in $\rr^n$, it can be combinatorially characterized (\cref{mult-support-solution}) if the intersection multiplicity at the origin is infinite for every collection of polynomials with these Newton diagrams; if this is not the case, then the the set of polynomials satisfying \eqref{bkk-bar-0} is Zariski dense in the space of polynomials with these Newton diagrams (\cref{bkk-existence}). Finally, we give a criterion which is equivalent to \eqref{bkk-bar-0}, but is computationally less expensive (\cref{weak-lemma}).

\subsubsection{Milnor number} Let $f$ be a polynomial or a power series in $n$ variables over $\kk$, and $\Gamma_1, \ldots, \Gamma_n$ be the Newton diagrams of partial derivatives of $f$. Assume that the minimum intersection multiplicity at the origin of all polynomials with these Newton diagrams is finite (this can be determined combinatorially due to \cref{mult-support-solution}). Then applying \cref{generic-thm-0} we show that 

\begin{thm}[\Cref{milnor-thm}]
The Milnor number of $f$ at the origin is the minimum if and only if $\dfd{1}, \ldots, \dfd{n}$ satisfy \eqref{bkk-bar-0}, in which case we say that $f$ is {\em partially Milnor non-degenerate}. 
\end{thm}
In particular, while (inner) Newton non-degeneracy is determined by {\em partial derivatives of initial forms} of $f$, partial Milnor non-degeneracy is determined by {\em initial forms of partial derivatives} of $f$. Note that
\begin{itemize}[wide]
\item A priori it is not clear if partially Milnor non-degenerate polynomials exist. Indeed, the collection of polynomials $f_1, \ldots, f_n$ which arise as derivatives of some polynomial is (contained in) a proper closed subvariety of the space of all polynomials with Newton diagrams $\Gamma_1, \ldots, \Gamma_n$, and it is a priori conceivable that this subvariety is contained in the complement of the Zariski dense set of polynomials which satisfy \eqref{bkk-bar-0}. Assertion \eqref{milnor-existence} of \cref{milnor-thm} ensures that this is not the case.\\


\item In particular, partial Milnor non-degeneracy does not suffer from any of the defects \ref{isolated-defect}--\ref{existential-defect}. \\

\item It follows from Kushnirenko's \cite{kush-poly-milnor} and Wall's \cite{wall} works that if a polynomial is Newton non-degenerate and has convenient Newton polytope, or if it is inner non-degenerate, then it is partially Milnor non-degenerate; we also give direct proofs of these results (\cref{newton-wall-lemma}). In \cref{newton-wall-lemma} we also show that if the singularity at the origin of $f = 0$ is isolated, then Newton non-degeneracy implies partial Milnor non-degeneracy, i.e.\ Newton non-degeneracy plus finite determinacy guarantee the minimality of the Milnor number at the origin; this seems to be a new observation. \\

\item In particular our results give an effective solution (\cref{kush-solution}), valid for fields of arbitrary characteristic, to the following problem considered and solved in the characteristic zero case by Kushnirenko \cite{kush-poly-milnor}: 

\begin{problem} \label{kush-problem}
Given a subset $\Sigma$ of $(\zz_{\geq 0})^n$, determine if there exists a power series $f$ in $n$ variables supported at $\Sigma$ such that the Milnor number at the origin is finite. If there exists such a function, then also determine the minimum possible Milnor number of such $f$.
\end{problem}

\end{itemize}

\subsubsection{Extension of the BKK bound to $\kk^n$} 


Given  a fixed collection $\mscrP := (\scrP_1, \ldots, \scrP_n)$ of $n$ convex integral polytopes in $\rr^n$, we characterize those systems of polynomials supported at these polytopes for which the number (counted with multiplicity) of isolated zeroes on $\kk^n$ (or more generally, on the complement of the union of a given set of coordinate subspaces of $\kk^n$) is the maximum (\cref{generic-bkk-thm}).
%
%
More precisely, we (combinatorially) divide the coordinate subspaces of $\kk^n$ into two (disjoint) groups $\tildeTP$ and $\tildeTprimeP$ such that for a given system $f = (f_1, \ldots, f_n)$ of polynomials supported at $\mscrP$, 
\begin{itemize}
\item for every $K \in \tildeTprimeP$, if $f$ has non-isolated roots on the `torus'\footnote{If $K = \bigcap_{j \in J} \{x_j= 0\}$, then the `torus' of $K$ is $K \cap \{\prod_{i \not\in J} x_i \neq 0\}$. \label{torus-footnote}} of $K$, then the number of its isolated roots on $\kk^n$ is less than that of a generic system.
\item if $K \in \tildeTP$,  then the existence of non-isolated roots of $f$ on the torus of $K$ does not necessarily imply that the number of isolated roots of $f$ on $\kk^n$ is less than that of a generic system. 
\end{itemize}
%

\begin{thm}[A special case of \cref{generic-bkk-thm}] \label{generic-bkk-thm-0} 
The number (counted with multiplicity) of isolated zeroes on $\kk^n$ of a  system $f = (f_1, \ldots, f_n)$ of polynomials supported at $\mscrP$ is the maximum if and only if the following two conditions hold:
\begin{defnlist}
\item \label{TPrime-property} for all $K' \in \tildeTprimeP$, and for all weighted order $\nu$ on $\kk[x_1, \ldots, x_n]$ which is either `centered at infinity' on $K'$ or `centered at the torus' of some $K \in \tildeTP$, the corresponding initial forms of $f_1|_{K'}, \ldots, f_n|_{K'}$ have no common zero on $\nktorus$, and
\item \label{TP-property} for all $K \in \tildeTP$ and for all weighed orders $\nu$ on $\kk[x_1, \ldots, x_n]$ centered at the torus of some $K' \in \tildeTprimeP$, the corresponding initial forms of $f_1|_K, \ldots, f_n|_K$ have no common zero on $\nktorus$. 
\end{defnlist}
\end{thm}

 We also give an equivalent version of the non-degeneracy criterion from \cref{generic-bkk-thm-0} which is computationally less expensive (\cref{weak-bkkriterion}), and a new formula for the maximum possible number of isolated zeroes on $\kk^n$ (\cref{bkk-thm}). 

\subsection{Idea of proof of non-degeneracy criteria, and differences from Bernstein's theorem}
\label{idea-section}
\subsubsection{} \label{sufficient-sketch} Our proof of the correctness of non-degeneracy criteria for the intersection multiplicity at the origin and for the extension of BKK bound follows Bernstein's \cite{bern} polynomial homotopy approach for the proof of correctness of BKK non-degeneracy. In particular, the basic idea of the proof of {\bf sufficiency} of the non-degeneracy criteria is as follows: given a sytem $f = (f_1, \ldots, f_n)$ of polynomials in $x := (x_1, \ldots, x_n)$ we consider a one parameter family of systems $f(x,t)= (f_1(x,t), \ldots, f_n(x,t))$ of polynomials that $f(x,0) = f$. If for generic $t$ the intersection multiplicity at the origin of $f_1(x,t), \ldots, f_n(x,t)$ is smaller than that of $f_1, \ldots, f_n$, then there is a curve $C$ consisting of non-zero roots of $f_1(x,t), \ldots, f_n(x,t)$ for $t \neq 0$ which `approaches the origin' as `$t$ approaches zero'. Similarly, if for generic $t$ the number of isolated roots of $f_1(x,t), \ldots, f_n(x,t)$ on $\kk^n$ is greater than that of $f_1, \ldots, f_n$, then there is a curve $C$ consisting of isolated roots of $f_1(x,t), \ldots, f_n(x,t)$ for $t \neq 0$ such that as `$t$ approaches zero', $C$ either `goes to infinity', or `approaches' a non-isolated root of $f_1, \ldots, f_n$. If $K$ is the smallest coordinate subspace of $\kk^n$ containing $C$ and $\nu$ is an appropriate weighted order such that $\nu(x_i)$ is proportional to the order of vanishing of $x_i|_C$ at the limiting point on $C$ as $t$ approaches $0$, then the non-degeneracy criteria are violated for $K$ and $\nu$.

\subsubsection{} \label{necessary-sketch} Similarly, the basic idea of our proof of the {\bf necessity} of the non-degeneracy criteria is as follows: given a degenerate system $f = (f_1, \ldots, f_n)$ of polynomials and a coordinate subspace $K$ of $\kk^n$ such that \eqref{bkk-bar-0} fails for $f_1|_K, \ldots, f_n|_K$, we try to construct following Bernstein \cite{bern} an explicit deformation of $f(x,t)$ of $f$ such that for generic $t$ there is a non-zero root of $f_1(x,t), \ldots, f_n(x,t)$ on $K$ which approaches the origin as $t$ approaches zero; if this is possible, then it follows that the intersection multiplicity at the origin of a degenerate system is higher than the minimum possible value. Similarly, if the non-degeneracy criterion for the extended BKK bound fails for $f_1|_K, \ldots, f_n|_K$, we try to construct a deformation $f(x,t)$ of $f$ such that for generic $t$ there is an isolated root of $f_1(x,t), \ldots, f_n(x,t)$ on $K$ which approaches either infinity or a non-isolated root of $f$ as $t$ approaches zero. 

\subsubsection{} \label{pathologies} The proof of sufficiency of \eqref{bkk-bar-0} follows easily from the arguments outlined in \cref{sufficient-sketch}. However, the proof of necessity of \eqref{bkk-bar-0} runs into a problem in the case that it fails for $f_1|_K, \ldots, f_n|_K$ for some coordinate subspace $K$ such that the dimension of $K$ is {\em less than} the number $N_K$ of $f_j$'s which do not identically vanish on $K$. In that case generic systems have {\em no} solution on the `torus' (see \cref{torus-footnote}) of $K$, so that Bernstein's deformation trick outlined in \cref{necessary-sketch} can not be executed. However, we show in \cref{reduction-lemma} that this scenario can be ignored: if $f_1|_K, \ldots, f_n|_K$ is degenerate for some $K$ such that $\dim (K) < N_K$, then there is a coordinate subspace $K' \supset K$ such that $f_1|_{K'}, \ldots, f_n|_{K'}$ is also degenerate and $\dim (K') = N_{K'}$. It follows that the deformation trick can be performed on $K'$. The necessity of \eqref{bkk-bar-0} follows then in a straightforward way via the arguments sketched in \cref{necessary-sketch}. Regarding the non-degeneracy criterion for the extended BKK bound, however, one has to be more careful and show that for generic $t$ the extra root of $f(x,t)$ is in fact isolated. 

\subsubsection{} \label{pathologies2} While the non-degeneracy criterion \eqref{bkk-bar-0} for intersection multiplicity is simple and very close to the non-degeneracy criterion \eqref{bkk} for the BKK bound, the non-degeneracy criterion in \cref{generic-bkk-thm-0} for the extended BKK bound is evidently more complicated. One of the reasons for this is the presence of non-isolated zeroes: indeed, if a system of $n$ polynomials has a non-isolated root on $\nktorus$, then it follows from Bernstein's theorem that the number of isolated roots on $\nktorus$ is {\em less} than that of a generic system\footnote{provided a generic system has at least one isolated root on $\nktorus$.}. However, for the case of $\kk^n$, it is possible to have two systems of polynomials with identical Newton polytopes such that both have the maximum possible number of isolated roots, but one has non-isolated roots on $\kk^n$ while the other has only isolated roots (\cref{b-example}). More precisely, one has the following:   
\begin{prop}[A special case of \cref{non-isolated-cor}] \label{non-isolated-lemma-0}
Let $\mscrP:= (\scrP_1, \ldots, \scrP_n)$ be a collection of $n$ convex integral polytopes in $(\rr_{\geq 0})^n$. Then each coordinate subspace $K$ of $\kk^n$ has one (and only one) of the following properties: 
\begin{enumerate}
\item \label{non-isolated-<-0} a generic system supported at $\mscrP$ has finitely many (possibly zero) roots on the `torus' (see \cref{torus-footnote} for the definition) $K^*$ of $K$; existence of non-isolated roots on $K^*$ implies that the number of isolated roots on $\kk^n$ is less than that of a generic system.
\item \label{non-isolated-not-<-0} no system supported at $\mscrP$ has any isolated root on $K^*$; existence of non-isolated roots on $K^*$ does not necessarily imply that the number of isolated roots on $\kk^n$ is less than that of a generic system.
\end{enumerate}
\end{prop}
It follows that while trying to detect degeneracy of $f_1, \ldots, f_n$ for the extended BKK bound, for every coordinate subspace $K$ of $\kk^n$,
\begin{itemize}
\item non-isolated roots of $f_1|_K, \ldots, f_n|_K$ should be detected only if $K$ is as in assertion \eqref{non-isolated-<-0} of \cref{non-isolated-lemma-0} (this is partly why condition \ref{TPrime-property} of \cref{generic-bkk-thm-0} is applied only to coordinate subspaces from $\tildeTprimeP$), and
\item a criterion should be formulated to detect if any isolated root of $f_1|_K, \ldots, f_n|_K$ is {\em non-isolated} in the set of zeroes of $f_1, \ldots, f_n$ in $\kk^n$ (\cref{c-example}); this is partly condition \ref{TP-property} of \cref{generic-bkk-thm-0}. 
\end{itemize}

\subsection{Idea of proof of the formulae for intersection multiplicity and the extended BKK bound}
\label{formula-idea-section}

Our formulae for intersection multiplicity (\cref{multiplicity-thm}) and extended BKK bound (\cref{bkk-thm}) are quite different from the formulae arrived in \cite{hurmfels-bern} and \cite[Corollary 1]{rojas-toric} via Huber and Sturmfels' polynomial homotopy arguments, and we do not know of any direct proof of their equivalence. %
Huber and Sturmfels' formulae originate from taking points in a generic fiber of the map determined by a non-degenerate system of polynomials, and tracking how many of these approach the origin (for the intersection multiplicity) or the pre-image of the origin (for the extended BKK bound). On the other hand, our formulae originate from the same approach as of Bernstein's proof of the BKK bound: given a non-degenerate system of polynomials $f_1, \ldots, f_n$, we consider the curves in the zero set of $f_2, \ldots, f_n$, and sum up the order at the origin (for the intersection multiplicity) or at points at infinity (for the extended BKK bound) of the restriction of $f_1$ along these curves. On the level of convex geometry, our formulae generalize the following (well-known) property of the mixed volume of $n$ convex bodies $\scrP_1, \ldots, \scrP_n$: 
\begin{align}
\mv(\scrP_1, \ldots, \scrP_n)
	&= \sum_\omega \omega(\scrP_1)\mv(\ld_\omega(\scrP_2), \ldots, \ld_\omega(\scrP_n)) \tag{$*$}
	\label{mixed-expression}
\end{align}
where $\mv$ denotes the mixed volume, the sum on the right hand side is over all `directions' $\omega$, $\omega(\scrP_1)$ is the maximum of $\langle \omega, v \rangle$ over all $v \in \scrP_1$, and $\ld_\omega(\scrP_j)$ is the `leading face' of $\scrP_j$ in the direction of $\omega$. In particular, as in the case of \eqref{mixed-expression}, the expressions on the right hand sides of our formulae from \cref{multiplicity-thm,bkk-thm} are {\em not} symmetric in their arguments. It will be nice to have a convex geometric argument for their symmetric dependence on the input polytopes. 

\begin{rem} \label{convenient-remark}
Combining Kushnirenko's \cite{kush-poly-milnor} results with ours we see that if the Newton diagram of a polynomial $f$ is convenient, then the Milnor number of the singularity at the origin of $f = 0$ is given by Kushnirenko's formula from \cite{kush-poly-milnor} iff $f$ is Milnor non-degenerate\footnote{provided there exist Newton non-degenerate or inner non-degenerate polynomials with the same diagram - which depends on the characteristic of the field.}. In particular,  in this case Kushnirenko's and our formulae compute the same number. However, the latter expresses the number as a sum of positive numbers, whereas the former is an alternating sum - it should be interesting to find a direct proof for the equivalence of these formulae.
\end{rem}

\subsection{Organization}
In \cref{example-section} we give some examples of new phenomena that one encounters when extending BKK bound from the torus to the affine space. \Cref{notection} introduces some notations and conventions to be used throughout this article. In \cref{mult-section} we state the solution to the intersection multiplicity problem, which we apply in \cref{milnor-section} to the problem of generic Milnor number. \Cref{milnor-section} also contains some open problems (\cref{open-problem}) on Milnor non-degeneracy. In \cref{affine-section} we state the results on extended BKK bound. \Crefrange{generic-proof-section}{bkk-proof-section} are devoted to the proofs of results from \cref{mult-section,affine-section}. \Cref{bkk-section} contains some technical results on existence and deformations of systems of polynomials which satisfy certain stronger non-degeneracy conditions; we suspect these are well known to experts, but we could not find any reference in the literature. \Cref{technical-section} contains a technical lemma on the algebraic definition of intersection multiplicity which we use in the proofs of computations of intersection multiplicity (\cref{multiplicity-thm}) and the extended BKK bound (\cref{bkk-thm}). In \cref{counter-section} we give some counter examples to \cite[Theorem I]{rojas-wang} and \cite[Affine Point Theorem II]{rojas-toric}.

\subsection{Acknowledgments}
I thank Pierre Milman and Dmitry Kerner for reading some of the earlier drafts of this article and providing inputs which made it substantially richer and more readable.\\

In a sequel we apply the results of this article to the study of the general `affine B\'ezout problem' of counting number of isolated solutions of an arbitrary (i.e.\ possibly degenerate) system of polynomials: we start from the extended BKK bound as the first approximation, and correct it in a finite number of steps to find the exact number of solutions. Explicit non-degeneracy criteria (such as the ones from this article) are required at each step to detect if the correct bound has been reached. 

{\small
\tableofcontents}

\addtocontents{toc}{\protect\setcounter{tocdepth}{2}}

\section{Examples of oddities regarding extension of BKK bound to $\kk^n$}\label{example-section}

For polynomials $f_1, \ldots f_n \in \kk[x_1, \ldots, x_n]$ and a subset $W$ of $\kk^n$, below we write $\multisof_W$ for the number (counted with multiplicity) of isolated zeroes of $f_1, \ldots, f_n$ which appear on $W$. 


\begin{example}\label{b-example}
Let $g_1, g_2, g_3 \in \kk[x,y,z]$ be polynomials supported at the tetrahedron $\scrT$ with vertices $(0,0,0), (1,1,0), (0, 1, 1), (1,0,1)$, i.e.\ 
$g_j = 	a_j +b_jxy + c_jyz + d_jzx$, $a_j, b_j, c_j, d_j \in \kk$, $j = 1, 2, 3$. Let
\begin{align*}
f_j &:= 
			\begin{cases}
			 g_j	 & \text{if}\ j=1,2,\\
			 zg_j & \text{if}\ j=3.
			\end{cases}
\end{align*}
Then
\begin{defnlist}
\item \label{b-example-generic-0} If all the coefficients are generic, then the set $V(f_1, f_2, f_3)$ of zeroes of $f_1, f_2, f_3$ on $\kk^3$ is isolated. Moreover, $V(f_1, f_2, f_3) \subset (\kk^*)^3$ and Bernstein's theorem implies that
\begin{align*}
[f_1,f_2,f_3]^{iso}_{\kk^3}  = [g_1,g_2,g_3]^{iso}_{(\kk^*)^3} = 6\vol(\scrT) = 2.
\end{align*}
\item \label{b-example-generic-1} On the other hand, if $a_1 = a_2$, $b_1 = b_2$, and the rest of the coefficients are generic, then $\{z = a_1 + b_1xy = 0\}$ is a positive dimensional component of $V(f_1, f_2, f_3)$. However, all the zeroes of $f_1, \ldots, f_3$ on $(\kk^*)^3$ are still isolated, and Bernstein's theorem guarantees that $[f_1,f_2,f_3]^{iso}_{\kk^3} = 2$. 
\item \label{b-example-non-generic} Finally, if $a_1 = a_2 = a_3$, $b_1 = b_2 = b_3$, and the rest of the coefficients are generic, then again $\{z = a_1 + b_1xy = 0\}$ is the positive dimensional component of $V(f_1, f_2, f_3)$. However, \eqref{bkk} fails for the weighted order $\nu$ with weights $(-1,1,2)$ for $(x,y,z)$, and Bernstein's theorem implies that $[f_1,f_2,f_3]^{iso}_{\kk^3} < 2$. 
\end{defnlist}
\end{example}

\begin{example}\label{c-example}
Take $f_1, f_2 \in \kk[x_1, x_2]$ such that the {\em Newton polytope} (see \cref{Newton-definition}) $\scrP_i$ of each $f_i$ contains the origin, the mixed volume $\mv(\scrP_1, \scrP_2)$ of $\scrP_1$ and $\scrP_2$ is non-zero, and $f_1, f_2$ are BKK non-degenerate (i.e.\ they satisfy \eqref{bkk}). Let $f_{3,1}, f_{3,2}, f_{4,1}, f_{4,2}$ be  arbitrary polynomials in $\kk[x_1, x_2]$, and set 
\begin{align*}
f_j &:= x_3f_{j,1}(x_1, x_2) + x_4f_{j,2}(x_1,x_2) \in \kk[x_1, x_2, x_3, x_4],\ j = 3, 4.
\end{align*} 
Then 
\begin{defnlist}
\item If the coefficients of $f_{3,1}, f_{3,2}, f_{4,1}, f_{4,2}$ are generic, then Bernstein's theorem implies that 
\begin{align*}
[f_1,f_2,f_3, f_4]^{iso}_{\kk^4} = [f_1,f_2,x_3, x_4]^{iso}_{\kk^4} = \mv(\scrP_1, \scrP_2). 
\end{align*}
\item \label{c-example-non-generic} Now fix a common zero $z$ of $f_1,f_2$ in $(\kk^*)^2$. If $f_{3,1}(z) = f_{4,1}(z) = 0$, then $\{(z_1,z_2, t, 0): t \in \kk\} \subseteq  V(f_1, f_2, f_3, f_4)$, so that $(z_1, z_2, 0,0)$ is no longer an isolated zero of $f_1, f_2, f_3, f_4$. It follows that $[f_1,f_2,f_3, f_4]^{iso}_{\kk^4} <  \mv(\scrP_1, \scrP_2)$. 
\end{defnlist}
\end{example}

\section{Notations, conventions and remarks} \label{notection}
Throughout this article, $\kk$ is an algebraically closed field of arbitrary characteristic and $n$ is a {\em positive} integer.

\begin{defn} \label{Newton-definition}
 Let $h := \sum_\alpha c_\alpha x^{\alpha} \in \kk[x_1, \ldots, x_n]$.
\begin{itemize}
\item The {\em support} of $h$ is $\supp(h) := \{\alpha: c_\alpha \neq 0\} \subseteq \zz^n$. 
\item The {\em Newton polytope} $\np(h)$ of $h$ is the convex hull $\conv(\supp(h))$ of $\supp(h)$ in $\rr^n$.
\item  The {\em Newton diagram} $\nd(h)$ of $h$ is the union of the compact faces of $(\rr_{\geq 0})^n + \np(h)$. 
\item A {\em diagram} in $\rr^n$ is the Newton diagram of some polynomial in $\kk[x_1, \ldots, x_n]$. 
\end{itemize}
\end{defn}

\begin{notation}
We write $[n] := \{1, \ldots, n\}$. For $I \subseteq [n]$ and $k = \kk$ or $\rr$, we write
\begin{align*}
\ki & := \{(x_1, \ldots, x_n) \in k^n: x_i =  0\ \text{if}\ i \not\in I\}  \cong k^{|I|}
\end{align*}
and denote by $\pi_I: k^n \to \ki$ the projection in the coordinates indexed by $I$, i.e.\
\begin{align}
\text{the $j$-th coordinate of $\pi_I(x)$}&:= 
			\begin{cases}
			x_j &\text{if}\ j \in I\\
			0	 &\text{if}\ j \not\in I.
			\end{cases} \label{pi_I}
\end{align}
For $\Gamma \subseteq \rr^n$, we write 
$$\Gamma^I := \Gamma \cap \ri$$
We denote by $\kk^*$ the {\em algebraic torus} $\kk\setminus \{0\}$, and write
\begin{align*}
\Kstari &:=  \{\prod_{i\in I} x_i \neq 0\} \cap \Ki \cong (\kk^*)^{|I|}
\end{align*}
Note that $\Kii{\emptyset} = \Kstarii{\emptyset} = \{0\}$. We write $\Ai$ for the coordinate ring $\kk[\Ki] = \kk[x_i: i \in I]$ of $\Ki$. 
\end{notation}

\begin{defn} \label{val-definition}
Let $I :=  \{i_1, \ldots, i_m\} \subseteq [n]$, where $m := |I|$. A {\em monomial valuation} or a  {\em weighted order} on $\Ai$ is a map $\nu: \Ai\setminus \{0\} \to \zz$ defined by 
\begin{align*}
\nu(\sum_\alpha c_\alpha x^\alpha) := \min\{v \cdot \alpha: c_\alpha \neq 0\}
\end{align*}
where  $v:= (v_{i_1}, \ldots, v_{i_m}) \in \zz^m$ is such that 
\begin{itemize}
\item either $v = (0, \ldots, 0)$ (the trivial valuation), or
\item $\gcd(v_{i_1}, \ldots, v_{i_m}) = 1$ (see \cref{gcd-convention} below).
\end{itemize}
We will often abuse the notation and identify $\nu$ with $v$. Given $I^* \subseteq I$, we say that the {\em center} of $\nu$ is $\Kstarii{I^*}$ if 
\begin{itemize}
\item $\nu_{i_j} \geq 0$ for all $j$, $1 \leq j \leq m$, and
\item $I^*  = \{i_j: \nu_{i_j} = 0\}$. 
\end{itemize}
In particular, we say that $\nu$ is {\em centered at the origin} if $\nu_{i_j} > 0$ for each $j$, $1 \leq j \leq m$. If there exists $j$ such that $\nu_{i_j} < 0$, then we say that the $\nu$ is {\em centered at infinity}. We write $\Vi$ (resp.\ $\Vzeroi$, $\Vi_\infty$) for the space of all monomial valuations (resp.\ monomial valuations centered at the origin, monomial valuations centered at infinity) on $\Ai$. We also write $\V := \Vii{[n]}$, $\Vzero := \Vzeroii{[n]}$, $\V_\infty := \Vii{[n]}_\infty$. 
\end{defn}

\begin{defn}
A {\em weighted degree} $\omega$ on $\Ai$ is the negative of a monomial valuation on $\Ai$. We say that $\omega$ is {\em centered at infinity} on $\Ai$ if $\omega(x_i) > 0$ for at least one $i \in I$, or equivalently, if $-\omega \in \Vi_\infty$. We denote by $\wtsi$ the space of all weighted degrees centered at infinity on $\Ai$, and write $\wts := \wtsii{[n]}$.  
\end{defn}

\begin{rem}
We want to emphasize that by our definition the greatest common divisor of weights of each non-trivial monomial valuation or weighted degree is $1$.
\end{rem}

\begin{defn} \label{val-diagram-defn}
Let $I \subseteq [n]$, $\nu \in \Vi$, $\omega \in \wtsi$, $\Gamma \subseteq \ri$, $f = \sum_\alpha c_\alpha x^\alpha \in \Ai$. Define
\begin{align*}
\nu(\Gamma) &:= \inf\{\nu \cdot \alpha: \alpha \in \Gamma\} \\
\omega(\Gamma) &:= \sup\{\omega \cdot \alpha: \alpha \in \Gamma\}
\end{align*}
The {\em initial face} (resp.\ {\em leading face}) of $\Gamma$ and the {\em initial form} (resp.\ {\em leading form}) of $f$ corresponding to $\nu$ (resp.\ $\omega$) are defined as
\begin{align*}
\In_\nu(\Gamma)
	&:= \{\alpha \in \Gamma: \nu \cdot \alpha = \nu(\Gamma)\} 
	&
\ld_\omega(\Gamma)
	&:= \{\alpha \in \Gamma: \omega \cdot \alpha = \omega(\Gamma)\} \\
\In_\nu(f)
	&:= \sum_{\alpha\in \In_\nu(\supp(f))} c_\alpha x^\alpha 
	&
\ld_\omega(f)
	&:= \sum_{\alpha\in \ld_\omega(\supp(f))} c_\alpha x^\alpha
\end{align*}
\end{defn}

\begin{rem}
Note that if $\Gamma = \emptyset$ (resp.\ if $f = 0$) then $\nu(\Gamma) = \infty$, $\omega(\Gamma) = -\infty$ (resp.\ $\nu(f) = \infty$, $\omega(f) = -\infty$). Also, if $\nu$ is the trivial valuation, then $\In_\nu(\Gamma) = \Gamma$ and $\In_\nu(f) = f$ for all $\Gamma$ and $f$. 
\end{rem}

\begin{convention} \label{gcd-convention}
Given $k_1, \ldots, k_l \in \zz$, $d := \gcd(k_1, \ldots, k_l)$ is defined as follows:
\begin{itemize}
\item if each $k_j = 0$, then $d = 0$.
\item otherwise $d$ is the {\em positive} greatest common divisor of the non-zero elements in $\{k_1, \ldots, k_l\}$.
\end{itemize}
\end{convention}

\begin{convention} \label{coconvention}
An assertion of the form ``codimension of $X$ in $Y$ is $m$'', where $X$ is a subvariety of $Y$ and $\dim(Y) < m$, means that $X = \emptyset$. 
\end{convention}

\begin{defn} \label{mixed-defn}
Let $\scrP_1, \ldots, \scrP_n$ be convex polytopes in $\rr^n$. The $n$-dimensional {\em mixed volume} $\mv(\scrP_1, \ldots, \scrP_n)$ of $\scrP_1, \ldots, \scrP_n$, we denote the coefficient of $\lambda_1 \cdots \lambda_n$ in the homogeneous polynomial $\vol(\lambda_1\scrP_1 + \cdots + \lambda_n\scrP_n)$. If $\scrP_1 + \cdots + \scrP_k$ has dimension $\leq k$, $1 \leq k \leq n$, then we write $MV(\scrP_1, \ldots, \scrP_k)$ to denote the $k$-dimensional mixed volume of $\scrP_1, \ldots, \scrP_k$. 
\end{defn}

\begin{rem}
Note that according to our definition $\mv(\scrP, \ldots, \scrP) = n!\vol(\scrP)$. There is another convention present in the literature in which mixed volume is defined to be $1/n!$ times the mixed volume from our definition. We chose to follow the convention that makes our formulae simpler. 
\end{rem}

\section{Intersection multiplicity} \label{mult-section}
In this section we state our results on intersection multiplicity at the origin of generic systems of $n$ polynomials or power series in $n$ variables. \Cref{mult-support-solution} characterizes collections of diagrams which admit polynomials with finite intersection multiplicity at the origin. \Cref{multiplicity-thm} gives a formula for the minimum intersection multiplicity and \cref{multiplicity-explanation} describes the meaning of the terms that appear in the formula. The necessary and sufficient conditions for the attainment of the minimum intersection multiplicity are described in \cref{generic-thm}. \Cref{weak-lemma} gives an equivalent criterion with fewer conditions. 

\begin{defn} \label{Isolated-defn}
Let  $\mscrG := (\Gamma_1, \ldots, \Gamma_m)$ be an ordered collection of diagrams in $\rr^n$ and $I$ be a non-empty subset of $[n]$. Define
\begin{align}
\NiG &:= \{j \in [m]: \Gamma^I_j \neq \emptyset\} \label{NiG}
\end{align}
We say that
\begin{itemize}
\item $f_1, \ldots, f_m \in \kk[x_1, \ldots, x_n]$ are {\em $\mscrG$-admissible} iff $\nd(f_j) = \Gamma_j$, $1 \leq j \leq m$.
\item  $\mscrG$ is {\em $I$-isolated at the origin} iff $\left|\NiG\right| \geq |I|$.
\end{itemize}
\end{defn}

The choice of the term {\em isolated} in \cref{Isolated-defn} is motivated by the following lemma:

\begin{lemma} 
\label{isolated-lemma}
Let $\mscrG := (\Gamma_1, \ldots, \Gamma_m)$ be a collection of diagrams in $\rr^n$ and $I \subseteq [n]$. Assume $0 \not\in \Gamma_i$ for each $j$, $1 \leq j \leq m$.  
\begin{enumerate}
\item If $\mscrG$ is not $I$-isolated at the origin, then for all $\mscrG$-admissible $f_1, \ldots, f_m \in \kk[x_1, \ldots, x_n]$, $V(f_1, \ldots, f_m) \cap \Ki$ is non-empty with positive dimensional components.  
\item If $\mscrG$ is $I$-isolated at the origin, then $V(f_1, \ldots, f_m) \cap \Kstari$ is isolated for generic $\mscrG$-admissible $f_1, \ldots, f_m \in \kk[x_1, \ldots, x_n]$. 
\end{enumerate}
\end{lemma}

\begin{proof}
Note that $ \{j : \Gamma^I_J \neq \emptyset\} = \{j: f_j|_{\Ki} \not \equiv 0\}$. If $\mscrG$ is not $I$-isolated at the origin, then it follows that $V(f_1, \ldots, f_m) \cap \Ki$ is defined by less than $|I|$ elements, which implies the first assertion (note that the assumption $0 \not\in \bigcup_j \Gamma_j$ is necessary). The second assertion follows from Bernstein's theorem.
\end{proof}

\begin{defn} \label{intersection-multiplicity}
The {\em intersection multiplicity} $\multzerof$ at the origin of $f_1, \ldots, f_n \in \kk[x_1, \ldots, x_n]$ is the dimension (as a vector space over $\kk$) of the quotient ring of $\kk[[x_1, \ldots, x_n]]$ modulo the ideal generated by $f_1, \ldots, f_n$. Let $\mscrG := (\Gamma_1, \ldots, \Gamma_n)$ be a collection of $n$ diagrams in $\rr^n$. Define
\begin{align*}
\multzeroGamma := \min\{\multzerof: f_1, \ldots, f_n\ 
			\text{are $\mscrG$-admissible polynomials in $\kk[x_1, \ldots, x_n]$}\} 
\end{align*}
\end{defn}

Note that $0 \leq \multzeroGamma \leq \infty$. \Cref{isolated-lemma} immediately gives a combinatorial way to determine if $\multzeroGamma$ is zero or infinity: 

\begin{cor}[cf.\ {\cite[Lemma 2]{rojas-toric}, \cite[Proposition 5]{herrero}}] 
\label{mult-support-solution}
Let $\Gamma_1, \ldots, \Gamma_n$ be diagrams in $\rr^n$.
\begin{enumerate}
\item $\multzeroGamma = 0$ iff $0 \in \Gamma_i$ for some $i$, $1 \leq i \leq n$.
\item Assume $\multzeroGamma \neq 0$. Then $\multzeroGamma< \infty$ iff $(\Gamma_1, \ldots, \Gamma_n)$ is $I$-isolated at the origin for all $I \subseteq [n]$. \qed
\end{enumerate}
\end{cor}

%

%

\Cref{multiplicity-thm} below describes the formula for the minimum intersection multiplicity; it uses \cref{mult-star-defn} which in turn uses notions from \cref{val-definition,val-diagram-defn}.

\begin{defn} \label{mult-star-defn}
Given polyhedra $\Gamma_1, \ldots, \Gamma_n$ in $\rr^n$, define
\begin{align}
\multstar{\Gamma_1}{\Gamma_n}
	&:=  \sum_{\nu \in \Vzero} \nu(\Gamma_1) ~ 
	\mv(\In_\nu(\Gamma_2), \ldots, \In_\nu(\Gamma_n))
	\label{mult-star}
\end{align}
\end{defn}

\begin{thm}[Formula for generic intersection multiplicity]\label{multiplicity-thm}
Let $\Gamma_1, \ldots, \Gamma_n$ be diagrams in $\rr^n$. Let 
\begin{align}
\scrI_{\mscrG,1} := \{I \subseteq [n]: I \neq \emptyset,\ |\NiG| = |I|,\ 1 \in \NiG\} \label{I1-list}
\end{align}
where $\NiG$ is as in \eqref{NiG}. Assume $0 < \multzeroGamma < \infty$. Then with notations as in \cref{mult-star-defn}, 
\begin{align}
\multzeroGamma 
	&= 	\sum_{I \in \scrI_{\mscrG,1}}
			\multzero{\pi_{I'}(\Gamma_{j_1})}{\pi_{I'}(\Gamma_{j_{n-k}})} \times
			\multstar{\Gamma^{I}_1, \Gamma^{I}_{j'_1}}{\Gamma^{I}_{j'_{k-1}}}  
	\label{multiplicity-formula}
\end{align}
where for each $I \in \scrI_{\mscrG,1}$,
\begin{itemize}
\item $I' := [n]\setminus I$, $k := |I|$,
\item $j_1, \ldots, j_{n-k}$ are elements of $[n]\setminus \NiG$,
\item $j'_1, \ldots, j'_{k-1}$ are elements of $\NiG\setminus \{1\}$.
\end{itemize}
\end{thm} 

\begin{proof}
See \cref{multiplicity-proof-section}. 
\end{proof}

\begin{convention} \label{formula-convention}
In this article, in particular in \eqref{mult-star} and \eqref{multiplicity-formula}, we use the following conventions:
\begin{defnlist}
\item $\infty$ times $0$ is interpreted as $0$. 
\item Empty intersection products and mixed volumes are defined as $1$. In particular, the term $\mv(\In_\nu(\Gamma_2), \ldots, \In_\nu(\Gamma_n))$ from \eqref{mult-star} in $n = 1$ case, and the term $ \multzero{\pi_{I'}(\Gamma_{j_1})}{\pi_{I'}(\Gamma_{j_{n-k}})}$ from \eqref{multiplicity-formula} in $I = [n]$ case are both defined as $1$. 
\end{defnlist}
\end{convention}

\begin{example} \label{ex0}
Let $f_1 = a_1x + b_1y + c_1z$, $f_2 = a_2x^3 + b_2xz^2 + c_2y^3 + d_2yz^2$, $f_3 = a_3x^2 + b_3xz^2 + c_3y^2 + d_3yz^2$, where $a_i, b_i, c_i, d_i$'s are generic elements of $\kk$. We use \eqref{multiplicity-formula} to compute $[f_1,f_2,f_3]_0$.
It is straightforward to check that $\scrI_{\mscrG,1} = \{\{1,2,3\},\{3\}\}$. From \eqref{multiplicity-formula} we see that
\begin{align*}
[f_1,f_2,f_3]_0 
	&= [\Gamma_1, \Gamma_2, \Gamma_3]^*_0 +
		 [\pi_{\{1,2\}}(\Gamma_2), \pi_{\{1,2\}}(\Gamma_3)]_0 ~ [\Gamma^{\{3\}}_1]^*_0 
\end{align*}
Note that $\pi_{\{1,2\}}(\Gamma_j)$, for each $j \in \{2,3\}$, is the diagram of a linear polynomial with no constant terms. Therefore $[\pi_{\{1,2\}}(\Gamma_2), \pi_{\{1,2\}}(\Gamma_3)]_0 = 1$. It is straightforward to see that $\Gamma_2 + \Gamma_3$ has two facets with inner normals in $(\rr_{> 0})^3$, and these inner normals are $\nu_1 := (1,1,1)$ and $\nu_2 := (2,2,1)$. Then 
\begin{align*}
[f_1,f_2,f_3]_0 
	&= \nu_1(\Gamma_1)\mv(\In_{\nu_1}(\Gamma_2), \In_{\nu_1}(\Gamma_3)) 
				+ \nu_2(\Gamma_1)\mv(\In_{\nu_2}(\Gamma_2), \In_{\nu_2}(\Gamma_3)) 
				+ 1 \cdot \ord_z(f_1|_{x=y=0}) \\
	&= 1 \cdot 4 + 1 \cdot 1 + 1 = 6
\end{align*}
\end{example}

\begin{rem} [`Explanation' of the right hand side of \eqref{multiplicity-formula}] \label{multiplicity-explanation}
Use the notations from \cref{multiplicity-thm}. Let $f_1, \ldots, f_n$ be generic $(\Gamma_1, \ldots, \Gamma_n)$-admissible polynomials.
\begin{defnlist}
\item For each $I \subseteq [n]$, note that $j \in \NiG$ iff $\Ki \not\subseteq V(f_j)$. Since $0 < \multzeroGamma < \infty$, it follows that 
\begin{itemize}
\item  $|\NiG| \geq |I|$ for each $I \subseteq [n]$. 
\item  If $|\NiG| =  |I|$ then there are $n-k$ polynomials (where $k = |I|$) $f_{j_1}, \ldots, f_{j_{n-k}}$ which identically vanish on $\Ki$. 
\end{itemize}
\item 
$ \multzero{\pi_{I'}(\Gamma_{j_1})}{\pi_{I'}(\Gamma_{j_{n-k}})}$ is precisely the intersection multiplicity of the hypersurfaces defined by $f_{j_1}, \ldots, f_{j_{n-k}}$ along $\Ki$.
\item $\multstar{\Gamma^{I}_1, \Gamma^{I}_{j'_1}}{\Gamma^{I}_{j'_{k-1}}}  $ is the sum (counted with multiplicity) of order of vanishing at the origin of restrictions of $f_1$ at the branches of the curves on $\Kstari$ defined by $f_{j'_1}, \ldots, f_{j'_{k-1}}$. 
\end{defnlist}
\end{rem}

\begin{rem} \label{mult-monotone}
It is not too hard to show that under \cref{formula-convention} identity \eqref{multiplicity-formula} is equivalent to the following identity: 
\begin{align}
\multzeroGamma 
	&= 	\sum_{\scriptsize \begin{array}{c} I \in [n] \\ I \neq \emptyset \end{array}}
			\sum_{\scriptsize \begin{array}{c} 1 \not\in J \subseteq [n]\\ |J| = n - |I| \end{array}}
			\multzero{\pi_{I'}(\Gamma_{j_1})}{\pi_{I'}(\Gamma_{j_{n-k}})} \times
			\multstar{\Gamma^{I}_1, \Gamma^{I}_{j'_1}}{\Gamma^{I}_{j'_{k-1}}} \label{multiplicity-formula'}
\end{align}
Given a diagram $\Gamma$ and a point $\alpha$ in $(\rr_{\geq 0})^n$, we say that $\alpha$ is {\em below} $\Gamma$ iff there exists $\nu \in \Vzero$ such that $\nu(\alpha) < \nu(\Gamma)$. Identity \eqref{multiplicity-formula'} provides an easy way to see the monotonicity of intersection multiplicity, i.e.\ the following fact (for which we do not know of any simple algebraic proof):
\begin{align*}
\parbox{.9\textwidth}{\em If no point of $\Gamma'_j$ lies below $\Gamma_j$ for all $j$, $1 \leq j \leq n$, then $\multzero{\Gamma'_1}{\Gamma'_n} \geq \multzeroGamma$.}
\end{align*}
Indeed, it suffices to prove the assertion for the case that $\Gamma'_j = \Gamma_j$ for $j = 2, \ldots, n$. Expanding the second term in the sum on the right hand side of \eqref{multiplicity-formula'} gives:
\begin{align*}
\multzeroGamma 
	&= 	\sum_{\scriptsize \begin{array}{c} I \in [n] \\ I \neq \emptyset \end{array}}
			\sum_{\scriptsize \begin{array}{c} 1 \not\in J \subseteq [n]\\ |J| = n - |I| \end{array}}
			\multzero{\pi_{I'}(\Gamma_{j_1})}{\pi_{I'}(\Gamma_{j_{n-k}})} 
		 \sum_{\nu \in \Vzeroi}  \nu(\Gamma^I_1) ~ 
		\mv(\In_\nu(\Gamma^I_{j'_1}), \ldots, \In_\nu(\Gamma^I_{j'_{k-1}})) 			
\end{align*}
Since $\nu(\Gamma'^I_1) \geq \nu(\Gamma^I_1)$ for all $\nu \in \Vzeroi$ and all $I$, the assertion follows.  
\end{rem}

We now describe the non-degeneracy condition for the intersection multiplicity at the origin. 

\begin{defn} \label{generic-defn}
\mbox{}
\begin{itemize}
\item We say that polynomials $f_1, \ldots, f_m$, $m \geq n$, in $(x_1, \ldots, x_n)$ are {\em BKK non-degenerate at the origin} iff $\In_\nu(f_1), \ldots, \In_\nu(f_m)$ have no common zero in $\nktorus$ for all $\nu \in \Vzero$.
\item 
Let $\mscrG:= (\Gamma_1, \ldots, \Gamma_n)$ be a collection of $n$ diagrams in $\rr^n$ such that $\multzeroGamma < \infty$. We say that polynomials $f_1, \ldots, f_n$ are {\em $\mscrG$-non-degenerate} iff they are $\mscrG$-admissible and $f_1|_{\Ki}, \ldots, f_n|_{\Ki}$ are BKK non-degenerate at the origin for every non-empty subset $I$ of $[n]$.
\end{itemize}
\end{defn}

\begin{thm} [Non-degeneracy for intersection multiplicity] \label{generic-thm}
Let $\mscrG:=(\Gamma_1, \ldots, \Gamma_n)$ be a collection of $n$ diagrams in $\rr^n$ such that $\multzeroGamma < \infty$. Let $f_1, \ldots, f_n$ be $\mscrG$-admissible polynomials. Then the following are equivalent:
\begin{enumerate}
\item $\multzerof = \multzeroGamma$.
\item $f_1, \ldots, f_n$ are $\mscrG$-non-degenerate.
\end{enumerate}
\end{thm}

\begin{proof}
See \cref{generic-proof}. 
\end{proof}
\begin{rem}
Note that we define $\mscrG$-non-degeneracy only in the case that $\multzeroGamma < \infty$. In fact if $\multzeroGamma = \infty$, then every system of $\mscrG$-admissible polynomials violates the condition from \cref{generic-defn}. In view of the non-degeneracy conditions for the extended BKK bound in \cref{affine-section}, it seems that for the correct notion of non-degeneracy which would extend to $\multzeroGamma = \infty$ case, one has to 
\begin{itemize}
\item replace ``$\In_\nu(f_1), \ldots, \In_\nu(f_m)$ have no common zero in $\nktorus$'' by ``the subvariety of $\nktorus$ defined by $\In_\nu(f_i)$'s has the {\em smallest possible} dimension,''
\item and add a set of conditions analogous to property \ref{NTSP-property} of $\mscrP$-non-degeneracy (\cref{generic-bkk-defn}) to ensure `correct infinitesimal intersections' of the subvarieties in the torus determined by restrictions of the polynomials to different coordinate subspaces. 
\end{itemize}
\end{rem}

In the proof of \cref{generic-thm} in \cref{generic-proof-section} we use the following result which shows that while testing for $\mscrG$-non-degeneracy one can omit certain coordinate subspaces of $\kk^n$. 

\begin{lemma} \label{weak-lemma}
Let $\mscrG:= (\Gamma_1, \ldots, \Gamma_n)$ be a collection of $n$ diagrams in $\rr^n$. Define 
\begin{align}
\nG  &:= \{I \subseteq [n]:\  I \neq \emptyset,\ |\NiG| = |I|\}, \label{NG}
\end{align}
where $\NiG$ is as in \eqref{NiG}. Let $f_1, \ldots, f_n$ be $\mscrG$-admissible polynomials in $(x_1, \ldots, x_n)$. Assume $\multzeroGamma < \infty$. Then the following are equivalent: 
\begin{enumerate}
\item $f_i$'s are $\mscrG$-non-degenerate. 
\item $f_1|_{\Ki}, \ldots, f_n|_{\Ki}$ are BKK non-degenerate at the origin for every $I \in \nG$.
\end{enumerate} 
\end{lemma}

\begin{proof}
This is precisely \cref{weak-lemma2}.
\end{proof}


\section{Milnor number} \label{milnor-section}
In this section we apply the results from \cref{mult-section} to study Milnor number at the origin of generic polynomials or power series. We define the notion of (partial) Milnor non-degeneracy (\cref{partial-defn}) and show in \cref{milnor-thm} that it is the correct notion of non-degeneracy in this context. As an immediate corollary we give an effective solution to \cref{kush-problem} (\cref{kush-solution}). We also recall (\cref{inndefinition}) the classical notions of non-degeneracy, give examples to illustrate that they are not necessary for the minimality of Milnor number (\cref{Milnor-not-Newton,Milnor-not-inner}), and give direct arguments (\cref{newton-wall-lemma}) to show that they are special cases of partial Milnor non-degeneracy (at least when the origin is an isolated singularity). We end this section with statements of some open problems. 

\subsection{(Partial) Milnor non-degeneracy}

\begin{defn} \label{partial-defn}
For a diagram $\Gamma \subseteq \rr^n$, let $\dGammad{i}$ be the Newton diagram of $\dgd{i}$ for a generic polynomial $g \in \kk[x_1, \ldots, x_n]$ with Newton diagram $\Gamma$. Let $f \in\kk[x_1, \ldots, x_n]$ be a polynomial with Newton diagram $\Gamma$ and $\Gamma_j$ be the Newton diagram of $\dfd{j}$, $1 \leq j \leq n$. We say that 
\begin{itemize}
\item  $f$ is {\em partially Milnor non-degenerate} iff $\multzeroGamma < \infty$ and $\dfd{1}, \ldots, \dfd{n}$ are $(\Gamma_1, \ldots, \Gamma_n)$-non-degenerate (see \cref{generic-defn}). 
\item  $f$ is {\em Milnor non-degenerate} iff $f$ is partially Milnor non-degenerate and $\multzeroGamma = \multzero{\dGammad{1}}{\dGammad{n}}$.
\end{itemize}
\end{defn}

\begin{rem} \label{milnor-remark}
\mbox{}
\begin{defnlist} [ ref= \therem(\alph{defnlisti})]
\item $\dGammad{i}$ depends on the characteristic of $\kk$.
\item In positive characteristic, it is possible for polynomials with the same Newton diagram to have distinct Newton diagrams for a partial derivative. In particular, it is possible that $\nd(f) = \Gamma$, but $\nd(\dfd{i}) \neq \dGammad{i}$ for some $i$, $1 \leq i \leq n$.
\item In zero characteristic partial Milnor non-degeneracy is equivalent to Milnor non-degeneracy. 
\end{defnlist}
\end{rem}
Let $f, \Gamma_1, \ldots, \Gamma_n$ be as in \cref{partial-defn}. Let $\scrA$ be the space of polynomials $g$ such that $\supp(g) \subseteq \supp(f)$. Recall the definition of $\mu(f)$ from \cref{milnor-footnote}. 

\begin{thm}[Non-degeneracy for Milnor number at the origin] \label {milnor-thm}
\mbox{}
\begin{enumerate}
\item \label{milnor-bound} $\mu(f) \geq \multzeroGamma \geq \multzero{\dGammad{1}}{\dGammad{n}}$.
\item \label{milnor-existence} The set of polynomials $g \in \scrA$ which satisfy $\mu(g) = \multzeroGamma$ contains a non-empty open subset of $\scrA$. 
\item \label{milnor-partial-genericness} The following are equivalent:
\begin{enumerate}
\item $f$ is partially Milnor non-degenerate. 
\item $\mu(f) = \multzeroGamma < \infty$. 
\end{enumerate}
\item \label{milnor-genericness} The following are equivalent:
\begin{enumerate}
\item $f$ is Milnor non-degenerate. 
\item $\mu(f) = \multzero{\dGammad{1}}{\dGammad{n}} < \infty$.
\end{enumerate}
\end{enumerate}
\end{thm}

\begin{proof}
Assertions \eqref{milnor-bound}, \eqref{milnor-partial-genericness} and \eqref{milnor-genericness} follow from \cref{mult-support-solution,multiplicity-thm,generic-thm,mult-monotone}. We now prove assertion \eqref{milnor-existence}. Without any loss of generality we may assume that $\multzeroGamma < \infty$. Let $I \subseteq [n]$ and $\nu \in \Vzeroi$. It suffices to show that there is an open subset $\scrU$ of $\scrA$ such that $V(\In_\nu(\dgd{1}|_{\Ki}), \ldots, \In_\nu(\dgd{n}|_{\Ki})) \cap \nktorus= \emptyset$ for all $g \in \scrU$.\\

Assume without loss of generality that \ $I = \{1, \ldots, k\}$. Take a generic $g \in \scrA$. Due to \cref{weak-lemma} we may also assume that $\dgd{j}|_{\Ki} \not\equiv 0$ for precisely $k$ values of $j$; denote them by $1 \leq j_1 < \cdots < j_k \leq n$. Write 
\begin{align*}
g &= g_0(x_1, \ldots, x_k) + \sum_{i = k+1}^n x_ig_i(x_1, \ldots, x_k)  \\
	&\qquad\quad + ~ \text{terms with quadratic or higher order in $(x_{k+1}, \ldots, x_n)$}
\end{align*}
If $g_0 \equiv 0$, then each $j_i > k$, and 
\[
V(\In_\nu(\dgd{1}|_{\Ki}),\ldots, \In_\nu(\dgd{n}|_{\Ki}))= V(\In_\nu(g_{j_1}), \ldots, \In_\nu(g_{j_k}))
\]
Since $g_{j_l}$'s are independent of each other, and since dimension of the (convex hull of the) sum of supports of $\In_\nu(g_{j_l})$'s is smaller than $k$, it follows that for generic $g \in \scrA$, $V(\In_\nu(g_{j_1}), \ldots, \In_\nu(g_{j_k})) \cap \nktorus  = \emptyset$ (\cref{herrero}), as required. So assume $g_0 \not\equiv 0$. Then for each $i$, $1 \leq i \leq k$, if $\dx{\In_\nu(g_0)}{x_i} \not\equiv 0$, then $\In_\nu(\dgd{i}) = \dx{\In_\nu(g_0)}{x_i} $. If $z \in V(\In_\nu(\dgd{1}|_{\Ki}),\ldots, \In_\nu(\dgd{n}|_{\Ki})) \cap \nktorus$, then it follows that
\begin{prooflist}
\item \label{zero-partial-x} $(\dx{\In_\nu(g_0)}{x_i})(z) = 0$, $1 \leq i \leq k$. 
\end{prooflist}
Let $\Lambda$ be the `smallest' linear subspace of $\rr^n$ such that the support of $\In_\nu(g_0)$ is contained in a translation of $\Lambda$; let $l := \dim(\Lambda)$. Choose a basis $\beta_1, \ldots, \beta_k$ of $\zz^k$ such that 
\begin{prooflist}[resume]
\item $\beta_1, \ldots, \beta_l$ is a basis of $\Lambda$. 
\item $\beta_1, \ldots, \beta_{k-1}$ is a basis of $\nu^\perp := \{\beta \in \zz^k: \nu \cdot \beta = 0\}$.
\end{prooflist}
Let $y_i := x^{\beta_i}$, $1 \leq i \leq k$. Then $(y_1, \ldots, y_k)$ is a system of coordinates on $\Kstari$. Observation \ref{zero-partial-x} implies that
\begin{prooflist}[resume]
\item \label{zero-partial-y} $(\dx{\In_\nu(g_0)}{y_i})(z) = 0$, $1 \leq i \leq k$. 
\end{prooflist}
In $(y_1, \ldots, y_k)$-coordinates, $\In_\nu(g_0)$ is of the form $y_k^dg^*_0(y_1, \ldots, y_l)$, where $d := \nu(g_0) > 0$. Observation \ref{zero-partial-y} implies that
\begin{prooflist}[resume]
\item $g^*_0(z) = (\dx{g^*_0}{y_1})(z) = \cdots = (\dx{g^*_0}{y_l})(z) = 0$. 
\end{prooflist}
An appropriate Bertini-type theorem in arbitrary characteristics, e.g.\ \cite[Theorem 3.1]{jelonek-bertini} implies that for generic $g \in \scrA$, $V(g^*_0, \dx{g^*_0}{y_1},\ldots, \dx{g^*_0}{y_l}) \cap \nktorus = \emptyset $. This completes the proof of assertion \eqref{milnor-existence}. 
\end{proof}

Let $\Sigma \subseteq (\zz_{\geq 0})^n$. For each $i = 1, \ldots, n$, define $\Gamma_i$ to be the Newton diagram of $\dfd{i}$ for any $f$ with $\supp(f) = \Sigma$. The following result, which is immediate from \cref{milnor-thm}, combined with \cref{mult-support-solution,multiplicity-thm}, gives an effective solution to \cref{kush-problem}. 


\begin{cor}[Solution to \cref{kush-problem}] \label{kush-solution}
The following are equivalent:
\begin{enumerate}
\item there exists a power series $f$ supported at $\Sigma$ such that $\mu(f) < \infty$,
\item $\multzeroGamma < \infty$.
\end{enumerate}
In the case that either of these conditions holds, the minimum value of all $\mu(f)$ such that $\supp(f) = \Sigma$ is precisely $\multzeroGamma$. \qed
\end{cor}

\subsection{Relation between Milnor non-degeneracy and the classical non-degeneracy conditions}

\begin{defn}[(Inner) Newton non-degeneracy {\cite[Section 3]{boubakri-greuel-markwig}}] \label{inndefinition}
Let $f  = \sum_\alpha c_\alpha x^\alpha \in \kk[x_1, \ldots, x_n]$. We write $\jjj(f)$ for the ideal generated by the partial derivatives of $f$. A {\em $C$-polytope} is a diagram containing a point with positive coordinates on every coordinate axis. If no point of $\supp(f)$ lies `below'\footnote{See the sentence immediately following identity \eqref{multiplicity-formula'}} $\scrP$, then for a face $\Delta$ of $\scrP$, we write $\In_\Delta(f) := \sum_{\alpha \in \Delta} c_\alpha x^\alpha$. An {\em inner face} of a polytope $\scrP$ is a face not contained in any coordinate subspace. We say that 
\begin{itemize}
\item $f$ is {\em Newton non-degenerate} iff $\jjj(\In_\nu(f))$ has no zero in $\nktorus$ for each $\nu \in \Vzero$. 
\item $\nd(f)$ is {\em convenient} if it is a $C$-polytope.
\item $f$ is {\em inner Newton non-degenerate} with respect to a $C$-polytope $\scrP$ iff 
\begin{itemize}
\item no point in $\supp(f)$ lies `below' $\scrP$.
\item for every inner face $\Delta$ of $\scrP$ and for every non-empty subset $I$ of $[n]$, 
\begin{align*}
\Delta^I \neq \emptyset \im V(\jjj(\In_\Delta(f))) \cap \Kstari = \emptyset
\end{align*}
\end{itemize}
\item $f$ is {\em inner Newton non-degenerate} iff it is inner Newton non-degenerate with respect to a $C$-polytope $\scrP$.
\end{itemize}
\end{defn}

\begin{example}[Newton non-degeneracy is not necessary for Milnor non-degeneracy]  \label{Milnor-not-Newton}
Let $f := x_1 + (x_2 + x_3)^q$, where $q \geq 2$ is not divisible by $\charac(\kk)$. Then $\np(f)$ is convenient and $f$ is not Newton non-degenerate. However, $f$ is Milnor non-degenerate with $\mu(f) = 1$. 
\end{example}

\begin{example}[Inner Newton non-degeneracy is not necessary for Milnor non-degeneracy] \label{Milnor-not-inner}
Let $f := x_3^3 + x_1(x_1 + x_2)^2 + x_2^6 + x_2^4x_3$. Note that $\nd(f)$ is convenient. Let $\Delta$ be the facet of $\nd(f)$ with inner normal $(1,1,1)$. Then $\Delta$ is an inner face of $\nd(f)$ and $\In_\Delta(f) = x_3^3 + x_1(x_1 + x_2)^2$. it is straightforward to check that for all but finitely many characteristics (including zero),
\begin{itemize}
\item The origin is an isolated singular point of $V(f)$.
\item If $I = \{1,2\}$, then $\Delta^I \neq \emptyset$ and $ V(\jjj(\In_\Delta(f))) \cap \Kstari \neq \emptyset$. It follows that $f$ is not inner Newton non-degenerate.
\item $f$ is Milnor non-degenerate.
\end{itemize}
\end{example}

\begin{example} [Diagram which admits no (inner) Newton non-degenerate polynomial] \label{no-(I)NND}
If $\charac(\kk) > 0$ and $\Gamma$ is any diagram which has a vertex all of whose coordinates are divisible by $\charac(\kk)$, then every polynomial with Newton diagram $\Gamma$ is (inner) Newton degenerate. 
\end{example}

\begin{defn} \label{compatible-defn}
Let $I \subseteq [n]$, $\nu \in \Vzeroi$, $\nu' \in \Vzero$. We say that $\nu'$ is {\em compatible with} $\nu$ iff $\nu'|_{\Ai}$ is proportional to $\nu$. 
\end{defn}

We now give direct proofs that the classical non-degeneracy conditions are special cases of partial Milnor non-degeneracy, at least when $\mu(f) < \infty$. We use the following lemma, which is straightforward to verify.

\begin{lemma} \label{non-empty-lemma}
Let $\Gamma_1, \ldots, \Gamma_m$ be diagrams in $\rr^n$, $I \subseteq [n]$ and $\nu \in \Vzeroi$. Let $\nu' \in \Vzero$ be such that
\begin{enumerate}
\item $\nu'$ is compatible with $\nu$ (\cref{compatible-defn}),
\item $\nu'(x_{i'})/\max\{\nu(x_i): i \in I\} \gg 1$ for all $i' \not \in I$.
\end{enumerate}
Then $\In_{\nu'}(\Gamma_j) = \In_{\nu}(\Gamma^I_j)$ for all $j \in [m]$ such that $\Gamma^I_j \neq \emptyset$ .\qed
\end{lemma}

\begin{prop} \label{newton-wall-lemma}
In each of the following cases $f$ is partially Milnor non-degenerate:
\begin{enumerate}
\item \label{newton-wall-1} $\nd(f)$ is convenient and $f$ is Newton non-degenerate.
\item \label{newton-wall-2} $f$ is Newton non-degenerate and $\mu(f) < \infty$. 
\item \label{newton-wall-3} $f$ is inner Newton non-degenerate.
\end{enumerate}
\end{prop}

\begin{proof}
Note that the first assertion is a special case of the second one. Nonetheless, we start with a direct proof of the first assertion since it is easier to see. Let $\Gamma := \nd(f)$ and $\Gamma_i := \nd(\dfd{i})$, $1 \leq i \leq n$. Pick a non-empty subset $I$ of $[n]$ and $\nu \in \Vzeroi$. Assume $\nd(f)$ is convenient and $f$ is Newton non-degenerate. Since $\nd(f)$ is convenient, $\Gamma^I \neq \emptyset$. Therefore we can find $\nu' \in \Vzero$ such that $\nu'$ is compatible with $\nu$ and $\In_{\nu'}(\Gamma) \subset \ri$. Then $\In_{\nu'}(f)$ depends only on $(x_i: i \in I)$. Since $f$ is Newton non-degenerate, it follows that $\partial(\In_{\nu'}(f))/\partial x_i$, $i \in I$, do not have any common zero in $\nktorus$. Now note that for each $i \in I$ such that $\partial(\In_{\nu'}(f))/\partial x_i$ is not identically zero, 
\[
\partial(\In_{\nu'}(f))/\partial x_i 
	= \partial(\In_{\nu}(f|_{\Ki}))/\partial x_i 
	=  \In_{\nu}((\dfd{i})|_{\Ki})
\]
This implies that  $\In_{\nu}((\dfd{i})|_{\Ki})$, $i \in I$, do not have any common zero in $\Kstari$, as required for partial Milnor non-degeneracy of $f$. \\

Now we prove the second assertion. Assume $\mu(f) < \infty$ and $f$ is {\em not} partially Milnor non-degenerate. It suffices to show that $f$ is not Newton non-degenerate. Pick $I \subseteq [n]$ and $\nu \in \Vzeroi$ such that $\In_\nu(\dfd{1}|_{\Ki}), \ldots, \In_\nu(\dfd{n}|_{\Ki})$ have a common zero in $\nktorus$. We may assume that $I = \{1, \ldots, k\}$, $1 \leq k \leq n$. At first consider the case that $f|_{\Ki} \not\equiv 0$. Pick $\nu' \in \Vzero$ such that 
\begin{itemize}
\item $\nu'$ satisfies \cref{non-empty-lemma} with $m = n$ and $\Gamma_i =  \nd(\dfd{i})$, $i = 1, \ldots, n$, 
\item $\In_{\nu'}(f) \in \kk[x_1, \ldots, x_k]$. 
\end{itemize} 
Then $\jjj(\In_{\nu'}(f))$ have a common zero in $\nktorus$, so that $f$ is not Newton non-degenerate, as required. Now assume that $f|_{\Ki} \equiv 0$. Without any loss of generality we may assume that $\dfd{i}|_{\Ki} \not\equiv 0$ iff $i = k+1, \ldots, l$, $k < l \leq n$. Then we can write
$$f = \sum_{i= k+1}^{l} x_if_i + \sum_{i > j >k} x_ix_jf_{i,j}$$
where $f_{k+1}, \ldots, f_{n'} \in \kk[x_1, \ldots, x_k]$ and $f_{i,j} \in \kk[x_1, \ldots, x_n]$. It is straightforward to check that there are positive integers $m_{k+1}, \ldots, m_n$ such that if $\nu'$ is the monomial valuation corresponding to weights $(\nu_1, \ldots, \nu_k, m_{k+1}, \ldots, m_n)$, then 
\begin{align*}
\nu'(f) = m_{k+1} + \nu(f_{k+1}) = \cdots = m_l + \nu(f_l) < \nu'(x_ix_jf_{i,j})\ \text{for each}\ i\geq j > k. 
\end{align*}
But then $\In_{\nu'}(f) = \sum_{i=k+1}^l x_i \In_{\nu}(f_i)$ and therefore $\partial (\In_{\nu'}(f))/\partial x_i = \In_\nu(\dfd{i}|_{\Ki})$, $i = k+1, \ldots, n$. It follows that $\jjj(\In_{\nu'}(f))$ have a common zero in $\nktorus$, and therefore $f$ is not Newton non-degenerate, as required.\\


It remains to prove the third assertion. Assume $f$ is inner Newton non-degenerate with respect to some $C$-polytope $\scrP$. For each $i' \not\in I$, let $M_{i'}$ be the set of {\em integral} elements in $\scrP$ of the form $\alpha + e_{i'}$ with $\alpha \in \ri$ (where $e_{i'}$ is the $i'$-th unit vector), and define
\begin{align*}
m_{i'} &:= \min\{\nu(\alpha'-e_{i'}): \alpha' \in M_{i'}\} \\
\Delta_{i'} &:= \{\alpha' \in M_{i'}: \nu(\alpha'-e_{i'}) = m_{i'}\}
\end{align*}
Let $\Delta_0 := \In_{\nu}(\scrP^I)$. Given $\eta \in \Vzero$, we write $\scrP_\eta$ for the face $\In_\eta(\scrP)$ of $\scrP$. Let $\nu'_0 \in \Vzero$ be a weighted order such that 
\begin{align*}
\nu'(x_i)
	&= 
	\begin{cases}
		\nu(x_i) 			& \text{if}\ i \in I, \\
		\gg 1				 &\text{otherwise.}
	\end{cases}
\end{align*}
Then $\scrP_{\nu'_0} = \Delta_0$. In particular, if $I $ is a proper subset of $[n]$, then $\scrP_{\nu'_0}$ is not an inner face. Assume without loss of generality that $I = \{1, \ldots, k\}$. We now modify $\nu'_0$ as follows: decrease the ratio of weights of $x_{k+1}$ and $x_1$ while keeping the ratios of weights of $x_j$ and $x_1$ fixed for every other $x_j$ - until the corresponding face contains a point with positive $x_{k+1}$-coordinate. Then repeat the same process for $x_{k+2}, \ldots, x_n$. Let $\nu'_1$ be the resulting weighted order. Note that 
\begin{prooflist}
\item \label{inner'-0} $\scrP_{\nu'_1}^I = \Delta_0$.
\item \label{inner'-j} For each $j > k$, either $M_j \cap \scrP_{\nu'_1} = \emptyset$ or  $M_j \cap \scrP_{\nu'_1} = \Delta_j$. 
\end{prooflist}
Let $\tilde I := \{i: 1 \leq i \leq k,$ $\scrP_{\nu'_1}$ does not contain an element with positive $x_i$-coordinate$\}$. Then $\scrP_{\nu'_1}$ is not an inner face iff $\tilde I$ is non-empty. Assume without loss of generality that $\tilde I = \{l+1, \ldots, k\}$, $1 \leq l \leq k\}$. Define $\nu'_2$ by modifying $\nu'_1$ in the same way as the construction of $\nu'_1$, by decreasing relative weights of $x_i$'s for $i \in \tilde I$. Then $\scrP_{\nu'_2}$ is an inner face containing $\scrP_{\nu'_1}$. In particular, $\scrP_{\nu'_2}^{I\setminus \tilde I} \neq \emptyset$. By inner non-degeneracy of $f$ it follows that $\jjj(\In_{\scrP_{\nu'_2}}(f))$ does not have any zero in $\Kstarii{I \setminus \tilde I}$. It is straightforward to check using \ref{inner'-0} and \ref{inner'-j} that the partial derivatives of $\In_{\scrP_{\nu'_2}}(f)$ that do not identically vanish on $\Kii{I \setminus \tilde I}$ belong to the set of $\In_{\nu}((\dfd{j})|_{\Ki})$, $1\leq j \leq n$. Since for each $i \in \tilde I$, $x_i$ does not appear in the expression of $\In_{\nu}((\dfd{j})|_{\Ki})$ for any $j$, $1 \leq j \leq n$, it follows that $\In_{\nu}((\dfd{j})|_{\Ki})$, $1 \leq j \leq n$, do not have any common zero in $\Kstari$, as required. 
\end{proof}

\begin{problem} \label{open-problem}
\mbox{}
\begin{enumerate}
\item Find a more `natural' definition of Milnor non-degeneracy. In particular, is it possible to replace the artificial looking condition ``$\multzeroGamma = \multzero{\dGammad{1}}{\dGammad{n}}$'' 
 by ``$\Gamma_j = \dGammad{j},\ 1 \leq j \leq n$''? 
 \item In the statement of \cref{newton-wall-lemma}, can ``partially Milnor non-degenerate'' be replaced by ``Milnor non-degenerate''?
\item Find an expression for $\multzero{\dGammad{1}}{\dGammad{n}}$ which is `more closely' related to $\Gamma$ than our formula \eqref{multiplicity-formula}. In particular, find a direct proof that if $\Gamma$ is convenient and no coordinate of any of its vertices is divisible by $\charac(\kk)$, then $\multzero{\dGammad{1}}{\dGammad{n}}$ equals an alternating sum of volumes as in Kushnirenko's original formula.
\end{enumerate}

\end{problem}

\section{Extending BKK bound to complements of unions of coordinate subspaces} \label{affine-section}
In this section we extend Bernstein's theorem to the complement of an arbitrary union of coordinate subspaces in $\kk^n$. The main results are \cref{bkk-thm}, which gives a formula for the maximum number of isolated solutions of polynomials with given Newton polytopes, and \cref{generic-bkk-thm}, which gives an explicit characterization of those systems of polynomials which attain this maximum. \Cref{weak-bkkriterion} gives an equivalent characterization involving fewer conditions. 

\subsection{Extended BKK bound}
Throughout this section $\mscrP:=(\scrP_1, \ldots, \scrP_m)$ is an ordered collection of convex integral polytopes in $(\rr_{\geq 0})^n$. 

\begin{defn}
We say that $f_1, \ldots, f_m \in \kk[x_1, \ldots, x_n]$ are {\em $\mscrP$-admissible} iff $\np(f_j) = \scrP_j$, $1 \leq j \leq m$. Let $I \subseteq [n]$. We say that 
\begin{itemize}
\item  $\mscrP$ is {\em $\Kstari$-isolated} if $V(f_1, \ldots, f_n) \cap \Kstari$ is isolated for generic $\scrP$-admissible $f_1, \ldots, f_n$. 
\item  $\mscrP$ is {\em $\Kstari$-trivial} if $V(f_1, \ldots, f_n) \cap \Kstari$ is empty for generic $\scrP$-admissible $f_1, \ldots, f_n$. 
\end{itemize}
\end{defn}

The following result gives combinatorial characterizations of $I \subseteq [n]$ such that $\mscrP$ is $\Kstari$-isolated or $\Kstari$-trivial. It also shows that if $\mscrP$ is {\em not} $\Kstari$-isolated, then the zeroes of {\em every} collection of $\mscrP$-admissible system of polynomials on $\Kstari$ is non-isolated.  

\begin{prop} \label{herrero}
Let $I \subseteq [n]$.
 \begin{enumerate}
\item The following are equivalent:
\begin{enumerate}
\item $\mscrP$ is $\Kstari$-trivial.
\item There exists $J \subseteq [m]$ such that $|\{j \in J: \scrP^I_j \neq \emptyset\}| > \dim(\sum_{j \in J} \scrP^I_j)$.
\end{enumerate}
\item The following are equivalent:
\begin{enumerate}
\item $\mscrP$ is $\Kstari$-isolated.
\item Either $\mscrP$ is $\Kstari$-trivial or $|\{j: 1 \leq j \leq m,\ \scrP^I_j \neq \emptyset\}| = |I|$.
\end{enumerate}
 \item If $\mscrP$ is not $\Kstari$-isolated, then for all $\mscrP$-admissible $f_1, \ldots, f_m \in \kk[x_1, \ldots, x_n]$,
 \begin{enumerate}
  \item $V(f_1, \ldots, f_m) \cap \Kstari$ is non-empty, and
 \item  every component of $V(f_1, \ldots, f_m) \cap \Ki$ is positive dimensional.
 \end{enumerate}
 \end{enumerate}
\end{prop}

\begin{proof}
Follows from \cite[Propositions 5 and 6]{herrero}.
\end{proof}

We now isolate a collection of coordinate subspaces of $\kk^n$ such that all the isolated zeroes on $\kk^n$ of a generic $\mscrP$-admissible system of polynomials lie on the complement of their union. Define
\begin{align}
\NiP &:= \{j: \scrP^I_j \neq \emptyset\},\ I \subseteq [n] \label{NiP} \\
\SP &:= \{I \subseteq [n]: \text{there is $\tilde I \supseteq I$ such that $|\NiiP{\tilde I}| < |\tilde I|$}\}  \label{SP} 
\end{align}
The following is an immediate corollary of \cref{herrero}. 
\begin{cor} \label{Pisolated-cor}
Let $f_1, \ldots, f_m$ be $\scrP$-admissible polynomials. Then
\begin{enumerate}
\item If $S \in  \SP$, then each point of $V(f_1, \ldots, f_m) \cap \Kii{S}$ is non-isolated in $V(f_1, \ldots, f_m) \subseteq \kk^n$. 
 \item If $f_i$'s are generic, then each point of $V(f_1, \ldots, f_m)\setminus \bigcup_{S \in \SP} \Kii{S}$ is isolated. \qed
\end{enumerate}
\end{cor}

Now we introduce some notations which will be convenient in dealing with complements of coordinate subspaces of $\kk^n$. 

\begin{defn} \label{VIS}
Let $\mscrS$ be a collection of subsets of $[n]$ and $I = \{i_1, \ldots, i_k\}$ be a non-empty subset of $[n]$. We denote by $\Ki_{\mscrS}$ the set $\Ki \setminus \bigcup_{S \in \mscrS} \Kstarii{S}$, by $\mscrS^I$ the following collection of subsets of $I$:
\begin{align}
\mscrS^I := \{S \in \mscrS: S \subseteq I\} \label{SI}
\end{align} 
and by $\overbar \mscrS$ the closure of $\mscrS$ under inclusion, i.e.\ 
\begin{align}
\overbar \mscrS := \{J \in [n]: J \subseteq S\ \text{for some}\ S \in \mscrS\}. \label{barS}
\end{align} 
Note that 
$$\Ki \setminus \Ki_{\mscrS} = \bigcup_{S \in \mscrS^I} \Kstarii{S},\ \text{and}\ 
\Ki \setminus \Ki_{\overbar \mscrS} = \bigcup_{S \in {\overbar \mscrS}^I} \Kii{S}.$$
 We denote by $\VSi$ the set of monomial valuations $\nu$ on $\Ki$ which are centered at $\Ki \setminus \Ki_{\mscrS}$, i.e.\ for which there exists $S \in \mscrS^I$ such that 
\begin{itemize}
\item  $\nu(x_s)= 0$ for all $s \in S$. 
\item  $\nu(x_s)> 0$ for all $s \in I \setminus S$. 
\end{itemize}
Finally, for $I = [n]$, we write $\kk^n_\mscrS$ for $\Kii{[n]}_\mscrS$ and $\VS$ for $\VSii{[n]}$. 
\end{defn}

\begin{example} \label{KnS-remark}
\mbox{}
\begin{itemize}
\item $\Ki_{\emptyset} =  \Ki$ and $\VSiii{I}{\emptyset} = \emptyset$; in particular $\kk^n_\emptyset= \kk^n$. 
\item $\Ki_{\{\emptyset\}} = \Ki \setminus \{0\}$ and $\VSiii{I}{\{\emptyset\}} = \Vzeroi$. 
\item $\Ki_{2^I\setminus\{I\}} = \Kstari$. 
\end{itemize}
\end{example}

\begin{defn} \label{generic-intersection-mult-defn}
Let $f_1, \ldots f_n \in \kk[x_1, \ldots, x_n]$. For each $z \in \kk^n$, we write $\multf_z$ for the intersection multiplicity of $f_i$'s at $z$, and for $W \subseteq \kk^n$, we define
\begin{align*}
\multf_W &:= \sum_{z \in W}\multf_z 
\end{align*} 
In other words, $\multf_W$ is the number (counted with multiplicity) of the zeroes of $f_1, \ldots, f_n$ on $W$. We also write $\multisof_W$ for the number (counted with multiplicity) of the zeroes of $f_1, \ldots, f_n$ on $W$ which are isolated in the (possibly larger) set of zeroes of $f_1, \ldots, f_n$ on $\kk^n$. Note that $\multisof_W$ is always finite, whereas $\multf_W$ is either infinite or equal to $\multisof_W$. Let $\mscrP:= (\scrP_1, \ldots, \scrP_n)$ be a collection of $n$ convex integral polytopes in $(\rr_{\geq 0})^n$. Define $\multP_W$ (resp.\ $\multisoP_W$) to be $\multf_W$ (resp.\ $\multisof_W$) for generic $\mscrP$-admissible $f_1, \ldots, f_n$. 
\end{defn}

%

The following corollary is simply a reformulation of \cref{Pisolated-cor} in terms of the notations introduced in \cref{VIS,generic-intersection-mult-defn}. 

\begin{cor} \label{Pisolated-cor2}
Let $\mscrP:= (\scrP_1, \ldots, \scrP_n)$ be a collection of $n$ convex integral polytopes in $(\rr_{\geq 0})^n$. Let $\SP$ be as in \eqref{SP}. Then 
\begin{enumerate}
 \item For each collection $\mscrS$ of subsets of $[n]$, $\multisoP_{\KnbarS} = \multP_{\kk^n_{\overbar \mscrS \cup \SP}} < \infty$.
 \item In particular, $\multisoP_{\kk^n} =  \multP_{\kk^n_{\SP}} < \infty$.  \qed
\end{enumerate}
\end{cor}

Given a collection $\mscrS$ of subsets of $[n]$, \cref{bkk-thm} below gives the formula for the number of isolated zeroes of a generic system of polynomials on $\KnS$. The formula uses \cref{multstarinfty-defn}; it is instructive to compare formula \eqref{bkk-formula} with formula \eqref{multiplicity-formula} from \cref{multiplicity-thm}. 

\begin{defn} \label{multstarinfty-defn}
Given a collection $\mscrS$ of subsets of $[n]$ and a collection $\mscrP := (\scrP_1, \ldots, \scrP_n)$ of $n$ convex integral polytopes in $(\rr_{\geq 0})^n$, define 
\begin{align*}
\multstarinfty 
	&:=  \sum_{\omega \in \wts} \omega(\scrP_1)
	\mv(\ld_\omega(\scrP_2), \ldots, \ld_\omega(\scrP_n))\\
\multstarS
	&:=  \sum_{\nu \in \VSi}  \nu(\scrP_1) 
	\mv(\In_\nu(\scrP_2), \ldots, \In_\nu(\scrP_n)) \\
\multstarinftyS &:= \multstarinfty - \multstarS
\end{align*}
\end{defn}

\begin{thm}[Extended BKK bound]
\label{bkk-thm}
With the notation from \cref{multstarinfty-defn}, define 
\begin{align}
\begin{split}
\ISPone &:= \{I \subseteq [n]:  I \not\in \overbar \mscrS \cup \SP \cup \{\emptyset\},\ |\NiP| = |I|,\ \text{$\mscrP$ is $\Kstari$-non-trivial},\ 1 \in \NiP \} 
\end{split} \label{I1-list-P}
\end{align}
where $\NiP$ is as in \eqref{NiP}. Let $\Gamma_j$ be the Newton diagram of generic polynomials supported at $\scrP_j$, $1 \leq j \leq n$. Then
\begin{align}
\multisoP_{\KnbarS} 
	&= 	\sum_{I \in \ISPone}
			\multzero{\pi_{I'}(\Gamma_{j_1})}{\pi_{I'}(\Gamma_{j_{n-k}})} \times
			\multstarinftySS{\scrP^{I}_1, \scrP^{I}_{j'_1}}{\scrP^{I}_{j'_{k-1}}}{(\overbar \mscrS \cup \SP)^I}  
	\label{bkk-formula}
\end{align}
where for each $I \in \ISPone$,
\begin{itemize}
\item $I' := [n]\setminus I$, $k := |I|$,
\item $j_1, \ldots, j_{n-k}$ are elements of $[n]\setminus \NiP$,
\item $j'_1, \ldots, j'_{k-1}$ are elements of $\NiP \setminus \{1\}$,
\item $(\overbar\mscrS \cup \SP)^I  = \{S \in \overbar \mscrS \cup \SP: S \subseteq I\} $ (as in \eqref{SI}).
\end{itemize}
\end{thm} 

\begin{proof}
See \cref{bkk-proof-section}.
\end{proof}

\begin{example} \label{ex0-bkk}
Consider $f_1,f_2, f_3$ from \cref{ex0}. We compute $[f_1,f_2,f_3]_{\kk^3}$ by applying \cref{bkk-thm} with $\mscrS = \emptyset$ (recall from \cref{KnS-remark} that $\kk^3 = \kk^3_\emptyset$). Let $\scrP_i$ be the Newton polytope of $f_i$ for each $i$, and set $\mscrP := (\scrP_1, \scrP_2, \scrP_3)$. It is straightforward to see that $\SP = \emptyset$, $\ISPone = \{\{1,2,3\},\{3\}\}$. Moreover, $\scrP_2 + \scrP_3$ has precisely two facets with outer normals in $\rr^3\setminus (\rr_{\leq 0})^3$, and the outer normals are $\omega_1 := (1,1,1)$ and $\omega_2 := (2,2,1)$. It follows that
\begin{align*}
[f_1,f_2,f_3]_{\kk^3}
	&= [\scrP_1, \scrP_2, \scrP_3]^*_\infty +
		 [\pi_{\{1,2\}}(\Gamma_2), \pi_{\{1,2\}}(\Gamma_3)]_0 ~ [\scrP^{\{3\}}_1]^*_\infty \\
	&= \omega_1(\scrP_1)\mv(\ld_{\omega_1}(\scrP_2), \ld_{\omega_1}(\scrP_3)) 
				+ \omega_2(\scrP_1)\mv(\ld_{\omega_2}(\scrP_2), \ld_{\omega_2}(\scrP_3)) 
				+ 1 \cdot \deg_z(f_1|_{x=y=0}) \\
	&= 1 \cdot 2 + 2 \cdot 3 + 1 = 9
\end{align*}
\end{example}

\subsection{Non-degeneracy condition for the extended BKK bound}
In general it is possible that a given polynomial system has an isolated zero on a coordinate subspace $K$ of  $\kk^n$, even though generic systems with the same Newton polytopes (as those of the given system) have no zeroes on $K$; for example, the system $(x + y - 1, 2x - y - 2)$ (over a field of characteristic not equal to two) has an isolated zero on the coordinate subspace $y = 0$. \Cref{absolute-defn} introduces a class of coordinate subspaces for which this is not possible. 


\begin{defn} \label{absolute-defn}
Let $\mscrP:= (\scrP_1, \ldots, \scrP_n)$ be a collection of $n$ convex integral polytopes in $(\rr_{\geq 0})^n$. For $I \subseteq [n]$, define $\NiP$ as in \eqref{NiP}. We say that $\mscrP$ is {\em $\Kstari$-exotrivial} if
\begin{defnlist}
\item  \label{absolute--less} there is $\tilde I \supseteq I$ such that $|\NiiP{\tilde I}| < |\tilde I|$, i.e.\ $I \in \SP$, or
\item \label{absolute-equal}  $\mscrP$ is $\Kstari$-trivial, and there is $\tilde I \supseteq I$ such that 
\begin{defnlist}
\item $\mscrP$ is $\Kstarii{\tilde I}$-trivial,
\item $|\NiiP{\tilde I}| = |\tilde I|$, 
\item $|\NiiP{I^*}| > |I^*|$ for each $I^*$ such that $I \subseteq I^* \subsetneqq \tilde I$. 
\end{defnlist}
\end{defnlist}
\end{defn}

\begin{prop} \label{absolute-prop}
Let $\mscrP:= (\scrP_1, \ldots, \scrP_n)$ be a collection of $n$ convex integral polytopes in $(\rr_{\geq 0})^n$ and $I \subseteq [n]$. If $\mscrP$ is $\Kstari$-exotrivial, then 
\begin{align*}
\multisof_{\Kstari} = 0
\end{align*}
for every $\mscrP$-admissible system of polynomials $f_1, \ldots, f_n$.
\end{prop}

\begin{proof}
It is proven in \cref{absolute-cor}.
\end{proof}

\begin{rem}
The prefix `exo' in {\em exotrivial} is to convey that $I$ satisfies the conclusion of \cref{absolute-prop} due to an `external' reason, namely the presence of some $\tilde I \supseteq I$ as in \cref{absolute-defn}. The other reason for $I$ satisfying the conclusion of \cref{absolute-prop} is `internal': the presence of some $\tilde I \subseteq I$ satisfying certain properties; we omit the discussion here since it is not relevant for the purpose of this article.
\end{rem}

\Cref{generic-bkk-defn} introduces $\mscrP$-non-degeneracy, and \cref{generic-bkk-thm}, which is the main result of this section, shows that this is the correct non-degeneracy criterion for the extended BKK bound. 

\begin{defn} \label{generic-bkk-defn}
Let $f_1, \ldots, f_m$, $m \geq n$, be polynomials in $(x_1, \ldots, x_n)$ and $\mscrS$ be a collection of subsets of $[n]$.  
\begin{itemize}
\item We say that 
\begin{itemize}
\item $f_i$'s are {\em BKK non-degenerate at infinity} if $\ld_\omega(f_1), \ldots, \ld_\omega(f_m)$ have no common zero in $\nktorus$ for all $\omega \in \wts$. 
\item $f_i$'s are {\em BKK non-degenerate at $\mscrS$} if $\In_\nu(f_1), \ldots, \In_\nu(f_m)$ have no common zero in $\nktorus$ for all $\nu \in \VS$. 
\end{itemize}
\item Let $\mscrP:= (\scrP_1, \ldots, \scrP_n)$ be a collection of $n$ convex integral polytopes in $(\rr_{\geq 0})^n$. Let $\mscrS$ be a collection of subsets of $[n]$. Define
\begin{align}
\TSP &:= \{I \subseteq [n]: I \not \in \overbar \mscrS \cup \SP,\ \mscrP\ \text{is $\Kstari$-exotrivial}\}  \label{TSP}\\
\TSPrime &:= \{I \subseteq [n]: I \not\in \overbar \mscrS \cup \SP,\ \mscrP\ \text{is {\em not} $\Kstari$-exotrivial}\}  \label{TSPrime}
\end{align}
\item 
We say that polynomials $f_1, \ldots, f_n$ are {\em $\mscrP$-non-degenerate on $\KnbarS$} iff 
\begin{defnlist}
\item \label{Padmissible} They are $\mscrP$-admissible.
\item \label{TSPrime-property} For all $I \in \TSPrime$, $f_1|_{\Ki}, \ldots, f_n|_{\Ki}$ are BKK non-degenerate both at infinity and at $(\overbar \mscrS  \cup \TSP)^I = \{S \in \overbar \mscrS  \cup \TSP: S \subseteq I\}$. 
\item \label{NTSP-property}  For all $I \in \TSP$, $f_1|_{\Ki}, \ldots, f_n|_{\Ki}$ are BKK non-degenerate at $\TSPrimeI = \{S \in \TSPrime: S \subseteq I\}$. 
\end{defnlist}
\end{itemize}
\end{defn}

%

\begin{thm}[Non-degeneracy condition for the extended BKK bound] \label{generic-bkk-thm}
Let $\mscrP:= (\scrP_1, \ldots, \scrP_n)$ be a collection of $n$ convex integral polytopes in $(\rr_{\geq 0})^n$ and $f_1, \ldots, f_n$ be $\mscrP$-admissible polynomials in $(x_1, \ldots, x_n)$.
Let $\TSP$ be as in \eqref{TSP}. Then 
\begin{enumerate}
\item \label{generic-assertion-0} $\multisof_{\KnbarS} = \multisof_{\KnSS{\overbar \mscrS \cup \TSP}} \leq \multP_{\KnSS{\overbar \mscrS \cup \TSP}} < \infty$.
\item \label{generic-assertion} The following are equivalent:
\begin{enumerate}
\item $f_i$'s are $\mscrP$-non-degenerate on $\KnbarS$
\item $\multisof_{\KnbarS} = \multP_{\KnSS{\overbar\mscrS \cup \TSP}} $
\end{enumerate} 
\item \label{isolated-assertion}  If $f_i$'s are $\mscrP$-non-degenerate on $\KnbarS$, then 
\begin{align*}
\multisof_{\KnbarS} = \multf_{\KnSS{\overbar\mscrS \cup \TSP}}
\end{align*}
In particular, in this case all zeroes of $f_1, \ldots, f_n$ on $\KnSS{\overbar\mscrS \cup \TSP}$ are isolated. 
\end{enumerate}
\end{thm}

\begin{proof}
See \cref{generic-bkk-proof}.
\end{proof}

\begin{example}[\Cref{b-example} continued]\label{b-example-contd}
Let $\mscrP := (\scrP_1, \scrP_2, \scrP_3)$, where $\scrP_i$ are the Newton polytopes of $f_i$ from \cref{b-example}. We computed in \cref{b-example} that $[\scrP_1, \scrP_2, \scrP_3]_{\kk^3} = 2$. Now we examine when $f_1,f_2,f_3$ are $\mscrP$-non-degenerate on $\kk^3$, so take $\mscrS = \emptyset$. It is straightforward to check that $\SP = \emptyset$, $\TSP=\{ \{1,2\}\}$ and $\TSPrime$ is the collection of all non-empty subsets of $\{1,2,3\}$ excluding $\{1,2\}$. In cases \ref{b-example-generic-0} and \ref{b-example-generic-1} of \cref{b-example}, it is straightforward to check that $f_1, f_2, f_3$ are $\mscrP$-non-degenerate, so that \cref{generic-bkk-thm} implies that the extended BKK bound is attained in these cases, as we saw in \cref{b-example}. On the other hand, in case \ref{b-example-non-generic} the discussion from \cref{b-example} shows that for $I = \{1, 2, 3\}$ the BKK non-degeneracy at infinity fails with $\omega = (1,-1,-2)$; this violates condition \ref{TSPrime-property} of $\mscrP$-non-degeneracy and therefore \cref{generic-bkk-thm} implies that the number of isolated zeroes of $f_1, f_2, f_3$ is less than $[\scrP_1, \scrP_2, \scrP_3]_{\kk^3} = 2$, as observed in \cref{b-example}. 
\end{example}

\begin{example}[\Cref{c-example} continued]\label{c-example-contd}
Let $\mscrP := (\scrP_1, \scrP_2, \scrP_3, \scrP_4)$, where $\scrP_i$ are the Newton polytopes of $f_i$ from \cref{c-example}, and let $\mscrS := \emptyset$. It is straightforward to check that $\SP = \emptyset$, $\TSP=\{ \{1,2,3,4\}\}$ and $\TSPrime$ is the collection of all proper non-empty subsets of $\{1,2,3,4\}$.
It follows that in case \ref{c-example-non-generic} of \cref{c-example}, condition \ref{NTSP-property} of $\mscrP$-non-degeneracy fails with $I = \{1,2,3,4\}$ and the monomial valuation $\nu$ corresponding to weights $(0,0,1,2)$ for $(x_1, \ldots, x_4)$. \Cref{generic-bkk-thm} therefore implies that in this case the number of isolated zeroes of $f_1, f_2, f_3, f_4$ is less than $[\scrP_1, \scrP_2, \scrP_3, \scrP_4]_{\kk^4}$, as we found out in \cref{c-example}. 
\end{example}

%
\Cref{generic-bkk-thm} implies a dichotomy among coordinate subspaces of $\kk^n$: 

\begin{cor} \label{non-isolated-cor}
Let $\mscrP = (\scrP_1, \ldots, \scrP_n)$ be a collection of $n$ convex integral polytopes in $(\rr_{\geq 0})^n$ and $\mscrS$ be a collection of subsets of $[n]$. Assume $\multP_{\KnbarS} > 0$. Then for every $\mscrP$-admissible system $f = (f_1, \ldots, f_n)$ of polynomials and every subset $I$ of $[n]$ such that $I \not \in \overbar{\mscrS}$ (so that $\Kstari \subseteq \KnbarS$),
\begin{enumerate}
\item \label{non-isolated-not-<} If $I \in \TSP$, then
the existence of non-isolated roots of $f$ on $\Kstari$ does not necessarily imply that 
$\multisof_{\KnbarS} < \multisoP_{\KnbarS}$. 
\item \label{non-isolated-<} If $I \in \TSPrime$, then the
existence of non-isolated roots of $f$ on $\Kstari$ implies that 
$\multisof_{\KnbarS} < \multisoP_{\KnbarS}$. 
\end{enumerate}
\end{cor}

\begin{proof}
This is proven in \cref{non-isolated-section}.
\end{proof}

The following result describes a criterion which is equivalent to $\mscrP$-non-degeneracy, but requires checking fewer conditions.

\begin{lemma} \label{weak-bkkriterion}
Let $\mscrP := (\scrP_1, \ldots, \scrP_n)$ be a collection of $n$ convex integral polytopes in $(\rr_{\geq 0})^n$ and $\mscrS$ be a collection of subsets of $[n]$. Define 
\begin{align}
\TSPstar  &:= \{I\in \TSP:\  |\NiP| = |I|\}, \label{TSPstar}\\
\NSP &:= \{I \in \TSPrime: \mscrP\ \text{is $\Kstari$-non-trivial}\},  \label{NSP} 
\end{align}
where $\NiP, \TSP, \TSPrime$ are respectively as in \eqref{NiP}, \eqref{TSP} and \eqref{TSPrime}. Let $f_1, \ldots, f_n \in \kk[x_1, \ldots, x_n]$. Then the following are equivalent: 
\begin{enumerate}
\item $f_i$'s are $\mscrP$-non-degenerate on $\KnbarS$.
\item
\begin{enumerate}[label=(\roman{enumii})]
\item property \ref{Padmissible} of $\mscrP$-non-degeneracy holds,
\item property \ref{TSPrime-property} of $\mscrP$-non-degeneracy holds with $\TSPrime$ replaced by $\NSP$,
\item property \ref{NTSP-property} of $\mscrP$-non-degeneracy holds with $\TSP$ replaced by $\TSPstar$.
\end{enumerate} 
\end{enumerate}
\end{lemma}

\begin{proof}
This is \cref{weak-bkk-lemma}. 
\end{proof}

\begin{remexample}[Warning]
In property \ref{NTSP-property} of $\mscrP$-non-degeneracy $\TSPrime$ can {\em not} be replaced by $\NSP$. Indeed, let 
\begin{align*}
f_1 &= 1 + x+y + z \\
f_2 &= 1 + x + 2y + 3z \\
f_3 &= y(ax + by + cz)+ z(a'x+b'y+c'z) \\
f_4 &= w(1+x) 
\end{align*}
where $a,b,c,a',b',c'$ are generic elements in $\kk^*$. Consider the ordering $(w,x,y,z)$ of the variables. If $\mscrP$ is the collection of newton polytopes of $f_1, \ldots, f_4$, then $\TSP$ (resp.\ $\TSPrime$) is the collection of all subsets of $\{1,2,3,4\}$ containing (resp.\ not containing) $1$. It is straightforward to check that \ref{NTSP-property} of of $\mscrP$-non-degeneracy on $\kk^4$ fails with $I = \{1, 2,3,4\}$ and $\nu = (1,0,1,1)$ and therefore the the number of isolated solutions of $f_1, \ldots, f_4$ on $\kk^4$ is less than $[\scrP_1, \scrP_2, \scrP_3]^{iso}_{\kk^4}$. However, if $\TSPrime$ is replaced by $\NSP$ in \ref{NTSP-property}, then this system would be $\mscrP$-non-degenerate on $\kk^4$. 
\end{remexample}

\section{Proof of the non-degeneracy condition for intersection multiplicity} \label{generic-proof-section}
In this section we prove \cref{generic-thm} following the outline in \cref{idea-section}. We start with clarifying what we mean by a `branch' of a curve and the corresponding `initial coefficients' of the (restrictions of) coordinate functions. The initial coefficients of a branch corresponds in an obvious way to a common zero of the corresponding initial forms of polynomials defining the curve (\cref{branch-lemma}), and this is in a sense the basic observation behind Bernstein's non-degeneracy conditions. In \cref{blow-up-intersection,blow-up-coordinates} we compile some (straightforward) facts about {\em weighted blow-ups} at the origin. The main technical result of \cref{pre-multiplicity-section} is \cref{reduction-lemma} which is our key to circumvent the problem outlined in \cref{pathologies}. In \cref{generic-proof} we use these results to prove \cref{generic-thm}.  

\subsection{Preparatory results} \label{pre-generic-section}

\begin{defn} \label{branchnition}
 Let $C$ be a (possibly reducible) curve on a variety $X$. Let $\pi_{C'}: C' \to C$ be a fixed desingularization of $C$ and $\bar C'$ be a fixed non-singular compactification of $C'$. A {\em branch} of $C$ is the germ of a point $z \in C'$, i.e.\ it is an equivalence class of the equivalence relation $\sim$ on pairs $B := (Z,z)$ such that
\begin{itemize}
\item $Z$ is an open subset of an irreducible component of $\bar C'$ and $z \in Z$. 
\item $(Z, z) \sim (Z',z')$ iff $z = z'$ and $Z \cap Z'$ is open in both $Z$ and $Z'$. 
\end{itemize}
In the case that $z \in \pi_{C'}^{-1}(O)$ for some $O \in C$, we say that $B$ is a {\em branch at $O$}. If $z \in \bar C' \setminus C$, we say that $B$ is a {\em branch at infinity} (with respect to $X$). 
\end{defn}

\begin{defn} \label{branch-remark}
Assume $X \subseteq \kk^n$. Let $B := (Z,z)$ be a branch of a curve $C \subseteq X$. Identify $Z^* := Z \setminus z$ with its image on $C$ and let $I_B :=  \{i: x_i|_{Z^*} \not\equiv 0\}$. Note that $\Kii{I_B}$ is the smallest coordinate subspace of $\kk^n$ which contains $Z^*$. For $f \in \kk[x_1, \ldots, x_n]$, we write $f|_B$ for $f|_{Z^*}$ and $\nu_B(f)$ for $\ord_z(f|_B)$. Let $I := I_B$ and $d :=  \gcd(\nuB(x_i): i \in I_B)$. We denote by $\nuBi$ the monomial valuation in $\Vi$ corresponding to weights $\nu_B(x_i)/d$ for $x_i$ for each $i \in I$. Fix an arbitrary element $\phi_B$ in $\kk(X)$ such that $\phi_B|_B$ is well defined and $\ord_z(\phi|_B) = 1$. Define
\begin{align*}
\In_B(x_j) &:= 
	\begin{cases}
		0	&\text{if}\ j \not\in I \\
		\left. \frac{x_j}{(\phi_B)^{\nuB(x_i)}} \right|_z
			&\text{if}\ j \in I.
	\end{cases}\\
\In(B) &:= (\In_B(x_1), \ldots, \In_B(x_n)) \in \Kstari
\end{align*} 
Note that $\In(B)$ depends on the choice of $\phi_B$. In all cases below, whenever a branch $B$ is considered, the corresponding $\phi_B$ is assumed to be fixed in the beginning; in other words, branches are to be understood as {\em triplets} $(Z,z, \phi)$. 
\end{defn}

\begin{rem}
In general $\nuBi$ is {\em not} the restriction of $\nu_B$ to the coordinate ring $\Ai$ of $\Ki$. Indeed, $\nuBi$ is a monomial valuation, but unless $B$ is a coordinate axis, $\nu_B|_{\Ai}$ is not even a discrete valuation, e.g.\ if $f$ is a non-zero polynomial in $\Ai$ that vanishes on $B$, then $\nu_B(f) = \infty$.
\end{rem}

If $\charac(\kk) = 0$, then the weights of $\nuBi$ are proportional to the exponents of initial terms of Puiseux series expansions of $x_j$ determined by $B$ and $\In(B)$ is (up to multiplication by a non-zero constant) the vector of coefficients of these initial terms. The following lemma states the fact, which is straightforward to verify, that $\In(B)$ is a common zero of the initial forms (corresponding to $\nuBi$) of all polynomials which vanish on $B$. 

\begin{lemma} \label{branch-lemma}
Let $f_1, \ldots, f_m \in \kk[x_1, \ldots, x_n]$ and $B$ be a branch of a curve contained in $V(f_1, \ldots, f_m)$. Let $I :=I_B$ and $\nu := \nuBi$. Then $\In(B) \in V(\In_\nu(f_1|_{\Ki}), \ldots, \In_\nu(f_m|_{\Ki})) \cap \Kstari$. \qed 
\end{lemma}

Now we interpret common zeroes of initial forms of $f_j$'s in terms of {\em weighted blow-ups} of $\kk^n$ at the origin. 

\begin{defn} \label{blow-up-defn}
Let $\nu  \in \Vzero$. The {\em $\nu$-weighted blow up $\Xnu$ at the origin} with respect to $(x_1, \ldots, x_n)$-coordinates is the blow up of $\kk^n$ at the ideal $\qqq_N$ generated by all monomials $x^\alpha$ in $(x_1, \ldots, x_n)$ such that $\nu (x^\alpha) = Np$, where $p := \lcm(\nu_1, \ldots, \nu_n)$ and $N$ is a sufficiently large integer. The exceptional divisor $\Enu$ has a natural structure of the weighted projective space $\pp(\nu_1, \ldots, \nu_n)$ with (weighted) homogeneous coordinates $[x_1: \cdots : x_n]$. We call $\Enustar := \Enu \setminus V(x_1\cdots x_n)$ the set of {\em interior points} of $\Enu$. Note that $\Enustar \cong (\kk^*)^{n-1}$ and $\Xnu$ is non-singular at every point of $\Enustar$. 
\end{defn}

The next two results follow from standard facts about toric varieties: 

\begin{lemma} \label{blow-up-intersection}
Let $\nu, \Xnu, \Enustar$ be as in \cref{blow-up-defn}. Let $z \in \nktorus$ and let $[z] \in \Enustar$ be the point whose homogeneous coordinates are the same as the coordinates of $z$. Let $f_1, \ldots, f_m \in \kk[x_1, \ldots, x_n]$. The following are equivalent:
\begin{enumerate}
\item $z$ is a common zero of $\In_\nu(f_j)$, $1 \leq j \leq m$.
\item $[z]$ is in the intersection on $\Xnu$ of the strict transforms of $\{f_j = 0\}$. \qed
\end{enumerate}
\end{lemma}

\begin{lemma} \label{blow-up-coordinates}
Let $I \subseteq [n]$, $\nu \in \Vzeroi$, $\nu' \in \Vzero$ be compatible with $\nu$. Let $\Xnu$ (resp.\ $\Xprimenuprime$) be the $\nu$-weighted (resp.\ $\nu'$-weighted) blow up of $\Ki$ (resp.\ $\kk^n$). 
\begin{enumerate}
\item $\Xnu$ can be identified with the strict transform of $\Ki$ in $\Xprimenuprime$.
\item If $f \in \kk[x_1, \ldots, x_n]$ is such that $\In_{\nu'}(f) = \In_\nu(f|_{\Ki})$, then under the above identification the strict transform of $\{f|_{\Ki} = 0\}$ in $\Xnu$ corresponds to the intersection of the strict transforms of  $\Ki$ and $\{f = 0\}$ on $\Xprimenuprime$.
\item \label{coordinate-assertion} Assume $\gcd(\nu'(x_i): i \in I) = 1$. Let $k := |I|$. Then under the above identification $\Enustar$ is contained in an non-singular open subset $W$ of $\Xprimenuprime$ isomorphic to $ (\ktorus)^{k-1} \times \kk^{n-k+1}$ with respect to coordinates $(z_1, \ldots, z_n)$ such that
\begin{enumerate}
\item $z_1, \ldots, z_k$ are monomials in $(x_i: i \in I)$.
\item $\nu'(z_1) = \nu(z_1) = 1$, $\nu'(z_i) = 0$, $2 \leq i \leq n$.
\item for all $i' \not\in I$, $z_{i'} = x_{i'}/z_1^{m_{i'}}$ for some positive integer $m_{i'}$.
\item $\Enuprime \cap W = V(z_1)$.
\item $\Xnu \cap W  = V(z_{k+1}, \ldots, z_n)$. 
\item $\Enu \cap W = \Enustar \cap W = V(z_1, z_{k+1}, \ldots, z_n)$. \qed
\end{enumerate}
\end{enumerate} 
\end{lemma}

Our next result (\cref{reduction-lemma}) plays a crucial role in this article; it is the key to \cref{weak-lemma,weak-bkkriterion} which state that when testing for our non-degeneracy conditions, we may omit the coordinate subspaces which are problematic for the reason described in \cref{pathologies}. 

\begin{lemma} \label{reduction-lemma}
Let $I$ be a non-empty subset of $[n]$ and $f_1, \ldots, f_m \in \kk[x_1, \ldots, x_n][x_i^{-1}: i \in I]$. Let $\TI:= \{j \in [m]: f_j|_{\Kstari} \equiv 0\}$. Assume 
\begin{enumerate}
\item \label{less-than-assumtion} $|\TI| <  n - |I|$.
\item \label{initial-assumption} there exists $\nu \in \Vi$ such that $\In_\nu(f_1|_{\Kstari}), \ldots, \In_\nu(f_n|_{\Kstari})$ have a common zero $u \in \nktorus$. 
\end{enumerate}
Then there exists $\tilde I \supsetneqq I$ and $\tilde \nu \in \Vii{\tilde I}$ such that 
\begin{enumerate}
\setcounter{enumi}{2}
\item \label{compatible-reduction} $\tilde \nu$ is compatible with $\nu$. 
\item \label{positive-reduction} $\tilde \nu(x_j) > 0$ for each $j \in \tilde I\setminus I$.
\item \label{zero-reduction} $\In_{\tilde \nu}(f_1|_{\Kstarii{\tilde I}}), \ldots, \In_{\tilde \nu}(f_n|_{\Kstarii{\tilde I}})$ have a common zero in $\tilde u \in \nktorus$ such that $\pi_I(\tilde u) = \pi_I(u)$. 
\end{enumerate}

\end{lemma}

\begin{proof}
We may assume without any loss of generality that $I = \{1, \ldots, k\}$ and $\TI = \{1, \ldots, l\}$, $1 \leq k < n$, $l < n-k$. Let $a := \pi_I(u)$. Then $a =(a_1, \ldots, a_k, 0, \ldots, 0)$ for some $a_1, \ldots, a_k  \in \kk^*$. At first consider the case that $\nu$ is the trivial valuation. Assumption \eqref{initial-assumption} then says that $a$ is a common zero of $f_1, \ldots, f_m$ in $\Kstari$. Let $(z_1, \ldots, z_n)$ be the translation of $(x_1, \ldots, x_n)$ by $a$, i.e.\
\begin{align*}
z_j &:=
			\begin{cases}
			x_j - a_j &\text{if}\ 1 \leq j \leq k\\
			x_j	&\text{if}\ k < j \leq n.
			\end{cases}
\end{align*}
Let $X_{\nu^*}$ be the weighted blow up of $\kk^n$ at $a$ in $(z_1, \ldots, z_n)$ coordinates corresponding to a monomial valuation $\nu^*$ such that
\begin{itemize}
\item $\nu^*(z_j) > 0$, $1 \leq j \leq k-1$,
\item $\nu^*(z_k) = 1$,
\item  $\nu^*(z_j) \gg 1$ for $k < j \leq n$. 
\end{itemize}
Let $E_{\nu^*}$ be the exceptional divisor of the blow up $X_{\nu^*} \to \kk^n$. Pick a point $a'$ on the intersection of $E_{\nu^*}$ with the strict transform $Y'$ of $\Kstari$. Since $l < n-k$, there is an irreducible component $V$ of $V(f_1, \ldots, f_l)$ properly containing $\Kstari$. Then the strict transform $V'$ of $V$ properly contains $Y'$. Choose a curve $C' \subseteq V'$ such that $a' \in C' \not\subseteq Y' \cup E_{\nu^*}$. Let $B = (Z',z')$ be a branch of $C'$ such that $z'$ is in the pre-image (in the desingularization of $C'$) of $a'$. Let $\tilde I := \{i : 1 \leq i \leq n, x_i|_B \not\equiv 0\}$, $\tilde \nu := \nuBii{\tilde I}$ and $\In(B) := (\In_B(x_1), \ldots, \In_B(x_n)) \in \Kstarii{\tilde I}$ (\cref{branch-remark}). Then
\begin{itemize}
\item For $1 \leq j \leq k$, $\nuB(z_j) > 0$,  so that $\nuB(x_j) = 0$. In particular $\In_B(x_j) = a_j$, $1 \leq j \leq k$, and $I \subseteq \tilde I$. 
\item Since $C' \not\subseteq Y'$, it follows that $I \subsetneqq \tilde I$.
\item $\nuB(x_{i'}) \gg 1$ for each $i' \in \tilde I\setminus I$. It follows that $\In_{\tilde \nu}(f_j) = \In_\nu(f_j)$ for each $j$, $l+1 \leq j \leq m$. (\cref{non-empty-lemma}). This implies that $\In(B)$ is a common zero of $\In_{\tilde \nu}(f_j|_{\Kii{\tilde I}})$, $l+1 \leq j \leq m$.
\item Since $C \subset V(f_j)$ for each $j =1, \ldots, l$, it follows that $\In(B)$ is a common zero of $\In_{\tilde \nu}(f_j|_{\Kii{\tilde I}})$, $1 \leq j \leq l$ (\cref{branch-lemma}).
\end{itemize}
This proves the lemma for the case that $\nu$ is the trivial valuation. \\

Now assume $\nu$ is not the trivial valuation. Note that the hypotheses and conclusions do not change if we multiply any of the $f_j$'s by a monomials in $(x_i: i \in I)$. It follows that after a monomial change of coordinates on $\Kstari$ we may assume that 
\begin{prooflist}
\item $f_j \in \kk[x_1, \ldots, x_n]$, $ 1 \leq j \leq m$,
\item \label{nu-standard}  $\nu(x_j) = 0 $, $1 \leq j < k$, and $\nu(x_k) = 1$.  
\end{prooflist}
Assumption \eqref{initial-assumption} then implies that $f_1, \ldots, f_m$ have a common zero $a = (a_1, \ldots, a_{k-1}, 0, \ldots, 0)$ with $a_1, \ldots, a_{k-1} \in \kk^*$. Now consider the translation $(z_1, \ldots, z_n)$ of $(x_1, \ldots, x_n)$ by $a$, and define $X_{\nu^*}$, $E_{\nu^*}$ and $Y'$ as in the proof of the trivial case. Then there is an affine open subset $U$ of $X_{\nu^*}$ with coordinates $(z_1/(z_k)^{\lambda_1}, \ldots, z_n/(z_k)^{\lambda_n})$, where $\lambda_j = \nu^*(z_j)$, $1 \leq j \leq n$, such that $E_{\nu^*} \cap Y' \cap U \neq \emptyset$. Pick $a' \in E_{\nu^*} \cap Y' \cap U$. Now the same arguments as in the proof of the trivial case complete the proof.
\end{proof}

Let $\mscrG:= (\Gamma_1, \ldots, \Gamma_n)$ be a collection of $n$ diagrams in $\rr^n$. Define $\nG$ and $\NiG$ as in \eqref{NG} and \eqref{NiG}. 

\begin{cor}[\Cref{weak-lemma}] \label{weak-lemma2}
 Let $f_1, \ldots, f_n$ be $\mscrG$-admissible polynomials in $(x_1, \ldots, x_n)$. Assume $\multzeroGamma < \infty$. Then the following are equivalent: 
\begin{enumerate}
\item \label{all-non-degnerate} $f_1|_{\Ki}, \ldots, f_n|_{\Ki}$ are BKK non-degenerate at the origin for every $I \subseteq [n]$.
\item \label{IG-non-degenerate} $f_1|_{\Ki}, \ldots, f_n|_{\Ki}$ are BKK non-degenerate at the origin for every $I \in \nG$.
\end{enumerate} 
\end{cor}

\begin{proof}
We may assume $\multzeroGamma \neq 0$ (since otherwise both assertions are always true). We only have to prove \eqref{IG-non-degenerate} $\im$ \eqref{all-non-degnerate}. It suffices to prove the following claim: ``if there is $I \subseteq [n]$ such that $f_1|_{\Ki}, \ldots, f_m|_{\Ki}$ are BKK degenerate at the origin, then there is $\tilde I \in \nG$ such that $f_1|_{\Kii{\tilde I}}, \ldots, f_m|_{\Kii{\tilde I}}$ are BKK degenerate at the origin.''\\

We prove the claim by induction on $n-|I|$. If $|I| = n$, then $I \in \nG$, so there is nothing to prove. Now assume that the claim is true whenever $|I| > k$, $1 \leq k \leq n$. Pick $I \subseteq [n]$ such that $|I| = k$ and $f_1|_{\Ki}, \ldots, f_m|_{\Ki}$ are BKK degenerate at the origin. Since $0 < \multzeroGamma < \infty$, \cref{mult-support-solution} implies that $|\NiG| \geq k$. If $|\NiG| = k$, then $I \in \nG$, so assume that $|\NiG| > k$. Then \cref{reduction-lemma} implies that there exists $\tilde I \supsetneqq I$ such that $f_1|_{\Kii{\tilde I}}, \ldots, f_n|_{\Kii{\tilde I}}$ are BKK degenerate at the origin. Since $|\tilde I| > k$, we are done by induction. 
\end{proof}

\subsection{Proof of \cref{generic-thm}} \label{generic-proof} In view of \cref{mult-support-solution} it suffices to treat the case that $0 < \multzeroGamma < \infty$. At first we establish that $\mscrG$-non-degeneracy is sufficient for the intersection multiplicity to be generic.

\begin{convention}
Below sometimes we work with $\kk^{n+1}$ with coordinates $(x_1, \ldots, x_n, t)$. In those cases we usually denote the coordinates of elements of $\kk^{n+1}$ as pairs, with the last component of the pair denoting the $t$-coordinate. In particular, the origin of $\kk^{n+1}$ is denoted as $(0,0)$.
\end{convention}

\begin{claim}  \label{sufficient-lemma}
Let $f_1, \ldots, f_n$ and $g_1, \ldots, g_n$ be two systems of $\mscrG$-admissible polynomials such that $\multzero{g_1}{g_n} < \multzerof$. Then $f_i$'s are $\mscrG$-degenerate. 
\end{claim}

\begin{proof}
If $\multzerof = \infty$, \cref{branch-lemma} implies that $f_i$'s are $\mscrG$-degenerate. So we may assume that $\multzerof < \infty$. Let $Z$ be the subscheme of $\kk^{n+1}$ defined by polynomials $h_i := tg_i + (1-t)f_i \in \kk[x_1, \ldots, x_n,t]$, $1 \leq i \leq n$. \Cref{isolated-family} implies that there exists an open subset $\tilde U$ of $\kk^{n+1}$ containing $(0,0)$ and $(0,1)$ such that 
\begin{prooflist}[series=sufficient-list]
\item \label{U-tilde-1} $\tilde Z := Z \cap \tilde U$ has dimension $1$.
\item \label{U-tilde-1.5} No irreducible component of $\tilde Z$ is `vertical', i.e.\ contained in a hyperplane of the form $\{t = \epsilon\}$ for some $\epsilon \in \kk$. 
\item \label{U-tilde-2} For generic $\epsilon \in \kk$, the points on $\tilde Z \cap \{t = \epsilon\}$ are isolated zeroes of $h_i|_{t = \epsilon}$, $1 \leq i \leq n$. 
\end{prooflist}
Since $\tilde Z$ is a complete intersection, we can write $\tilde Z = \sum m_i \tilde Z_i$, where $\tilde Z_i$'s are irreducible compoents of $\tilde Z$ and $m_i$'s are intersection multiplicities of $h_1, \ldots, h_n$ along $\tilde Z_i$'s. We may assume that $\tilde Z_1$ is the line $\{0\} \times \kk$. For each $\epsilon \in \kk$, let $H_\epsilon := V(t - \epsilon) \subseteq \kk^{n+1}$.  Then
\begin{align*}
m_1 
	&= H_1 \cdot m_1\tilde Z_1 = H_1 \cdot \tilde Z
	= \multzero{h_1|_{t=1}}{h_n|_{t=1}} = \multzero{g_1}{g_n}
\end{align*}
It follows that 
\begin{align*}
H_0 \cdot \tilde Z = \multzero{h_1|_{t=0}}{h_n|_{t=0}} = \multzerof > m_1
\end{align*}
Since $H_0$ intersects $\tilde Z$ properly, it follows that there is another irreducible component $\tilde Z_2$ of $\tilde Z$ such that $(0, 0) \in \tilde Z_2$. Pick a branch $B$ of $\tilde Z_2$ at $(0,0)$. Let $\tilde I := I_B \subseteq [n+1]$ and $\tilde \nu := \nuBii{\tilde I}$ (\cref{branch-remark}). Note that $\{n+1\} \subsetneqq \tilde I$. Let $I := \tilde I \cap [n]$ and $\nu := \tilde \nu|_{\Ai}$. Since $g_j$ and $f_j$ have the same Newton diagram for each $j$, it follows that $\tilde \nu(f_j) = \tilde \nu(g_j) = \nu(f_j)$. Since $\tilde \nu(t) > 0$, it follows that $\In_{\tilde \nu}(h_j|_{\Kii{\tilde I}}) = \In_\nu(f_j|_{\Ki})$, $1 \leq j \leq n$. \Cref{branch-lemma} then implies that $\In_\nu(f_1|_{\Ki}), \ldots, \In_\nu(f_n|_{\Ki})$ have a common zero in $\Kstari$. Since $\nu$ is a (positive multiple of some) valuation in $\Vzeroi$, it follows that $f_i$'s are $\scrG$- degenerate, as required.
\end{proof}

It remains to prove that $\mscrG$-non-degeneracy is necessary for the intersection multiplicity to be generic. Take $f_1, \ldots, f_n$ which are not $\mscrG$-non-degenerate. We have to show that $\multzerof  > \multzeroGamma$. Since $\multzeroGamma < \infty$, it suffices to consider the case that $0 < \multzerof< \infty$. \\

Let $\nG$ and $\NiG$ be as in \eqref{NG} and \eqref{NiG}. \Cref{weak-lemma2} implies that there is $I \in \nG$ and $\nu \in \Vzeroi$ such that $\In_\nu(f_j|_{\Ki})$, $ 1 \leq j \leq n$, have a common zero $z \in \Kstari$. We may assume that $I = \NiG = \{1, \ldots, k\}$. Choose $y = (y_1, \ldots, y_k, 0, \ldots, 0) \in \Kstari$ such that $f_j(y) \neq 0$ for all $j$, $1 \leq j \leq k$. Let $C$ be the rational curve on $\Ki$ parametrized by $c(t):= (c_1(t), \ldots, c_n(t)) :\kk \to \Ki$ given by
\begin{align}
c_j(t) := 
			\begin{cases}
			z_jt^{\nu_j} + (y_j - z_j)t^{\nu'_j} &\text{if}\ 1 \leq j \leq k, \\
			0 &\text{if}\ k < j \leq n.
			\end{cases} \label{c_j}
\end{align}
where $\nu_j := \nu(x_j)$ and $\nu'_j$ is an integer greater than $\nu_j$, $1 \leq j \leq k$. Note that $c(0) = 0$ and $c(1) = y$. Let $g_1, \ldots, g_k$ be $(\Gamma_1, \ldots, \Gamma_k)$-admissible polynomials such that $\In_\nu(g_j|_{\Ki})(z) \neq 0$ for each $j$, $1 \leq j \leq k$. Then for each $j$, $1 \leq j \leq k$,
\begin{align}
\ord_t(g_j|_C) = \nu(\Gamma^I_j) < \ord_t(f_j|_C) \label{g_j<f_j}
\end{align}
Let $\mu_j := \nu(\Gamma^I_j)$, $1 \leq j \leq k$. Then for each $j \leq k$, $t^{-\mu_j}(f_j(c(t))$ and $t^{-\mu_j}(g_j(c(t))$ are polynomials in $t$. Let $U$ be a neighborhood of the origin on $\kk^n$ such that the origin is the only point in $V(f_1, \ldots, f_n) \cap U$. Then $T := c^{-1}(U) \subseteq \kk$ is an open neighborhood of $0$ in $\kk$. Define 
\begin{align}
h_j &:=
	\begin{cases}
	(t^{-\mu_j}f_j(c(t)))g_j - (t^{-\mu_j}g_j(c(t)))f_j & \text{if}\ 1 \leq j \leq k,\\
	f_j & \text{if}\ k < j \leq n. 
	\end{cases} \label{h_j}
\end{align}
Identity \eqref{g_j<f_j} implies that 
\begin{prooflist}
\setcounter{prooflisti}{3}
\item \label{h_j-at-0} $h_j(x,0)$ is a non-zero constant times $f_j$ for each $j$, $1 \leq j \leq k$.
\end{prooflist}
Let $Z$ be the subscheme of $U \times T$ defined by $h_j(x,t)$, $1 \leq j \leq n$. Let $\tilde Z$ be the union of the irreducible components of $Z$ containing $(0,0)$. Since $0$ is an isolated point on $V(f_1|_{\Ki}, \ldots, f_k|_{\Ki})$, it follows that $(0,0)$ is an isolated point of $Z \cap \{t=0\}$. This implies that $\tilde Z $ has pure dimension one. In particular, 
\begin{prooflist}[resume]
\item the `$t$-axis' $Z_t := \{0\} \times T$ and $\tilde C :=\{ (c(t),t): t \in T\}$ are irreducible components of $\tilde Z$. 
\end{prooflist}
Choose an open subset $U^*$ of $U \times T$ such that $Z \cap U^* = Z_t \cap U^*$. Note that $Z_t \cap U^*= \{0\} \times T^*$ for some open subset $T^*$ of $T$. Then
\begin{prooflist}[resume]
\item for each $\epsilon \in T^*$, the origin is an isolated zero of $h_1|_{t= \epsilon}, \ldots, h_n|_{t = \epsilon}$. 
\end{prooflist} 
Since $\tilde Z$ is a complete intersection, we can write $\tilde Z = \sum m_i \tilde Z_i$, where $\tilde Z_i$'s are irreducible components of $\tilde Z$ and $m_i$'s are intersection multiplicities of $h_1, \ldots, h_n$ along $\tilde Z_i$'s. We may assume that $\tilde Z_1 = Z_t$ and $\tilde Z_2 = \tilde C$. For each $\epsilon \in T$, let $H_\epsilon$ be the hypersurface $U \times \{\epsilon\}$ in $U \times T$. Then for all $\epsilon \in T^*$, 
\begin{align*}
m_1 
	&= H_\epsilon \cdot m_1Z_t = H_\epsilon \cdot (Z \cap U^*)
	= \multzero{h_1|_{t=\epsilon}}{h_n|_{t=\epsilon}}
\end{align*}
where $\multzero{h_1|_{t=\epsilon}}{h_n|_{t=\epsilon}}$ is the intersection multiplicity at the origin of $h_1|_{t=\epsilon}, \ldots, h_n|_{t=\epsilon}$. Now choose a neighborhood $\tilde U$ of $(0,0)$ in $U \times T$ such that $Z \cap \tilde U \cap H_0$ consists of only $\{(0,0)\}$. Since $H_0$ intersects each $\tilde Z_i$ properly, it follows by \ref{h_j-at-0} and definition of intersection multiplicity that
\begin{align*}
\multzerof
	&= (Z \cap \tilde U) \cdot H_0 
	= (m_1 Z_t + m_2 \tilde C + \cdots) \cdot H_0
	> m_1 Z_t \cdot H_0
	= m_1
\end{align*}
Since $h_1|_{t=\epsilon}, \ldots, h_n|_{t=\epsilon}$ are $\mscrG$-admissible for generic $\epsilon \in T^*$, this implies that $\multzerof  > \multzeroGamma$, as required. \qed

\section{Proof of non-degeneracy conditions for the extended BKK bound} \label{generic-bkk-section}
In this section we prove \cref{generic-bkk-thm} following the approach outlined in \cref{idea-section}. In \cref{pre-bkk-section} we study properties of exotrivial coordinate subspaces and prove \cref{weak-bkkriterion}. The main technical result of \cref{pre-bkk-section} is \cref{trivial-lemma} which is the basis for property \ref{NTSP-property} of $\mscrP$- non-degeneracy. In \cref{generic-bkk-proof} we use these results to establish \cref{generic-bkk-thm}

\subsection{Some properties of exotrivial coordinate subspaces} \label{pre-bkk-section}
Throughout this section $\mscrP:= (\scrP_1, \ldots, \scrP_n)$ is a collection of $n$ convex integral polytopes in $(\rr_{\geq 0})^n$, $f_1, \ldots, f_n$ are $\mscrP$-admissible polynomials in $(x_1, \ldots, x_n)$, and $\mscrS$ is a collection of subsets of $[n]$. In the proof of the result below we use the notion of {\em complete BKK non-degeneracy} (\cref{bkk-defn}). 


\begin{lemma}[\Cref{absolute-prop}]\label{absolute-cor}
Let $I \subseteq [n]$. If $\mscrP$ is $\Kstari$-exotrivial, then the zero set of $f_1, \ldots, f_n$ on $\kk^n$ has no isolated point on $\Kstari$.
\end{lemma}

\begin{proof}
If there exists $\tilde I \supseteq I$ such that $|\NiiP{\tilde I}| < |\tilde I|$, then the statement is clear; so assume $\mscrP$ is $\Kstari$-trivial and there is $\tilde I \supseteq I$ as in property \ref{absolute-equal} of \cref{absolute-defn}. By restricting all $f_j$'s to $\Kii{\tilde I}$, we may assume that $\tilde I = [n]$. Let $V_f$ be the set of isolated zeroes of $f_1, \ldots, f_n$ in $\kk^n$. It suffices to consider the case that $V_f \neq \emptyset$.\\

Let $g_1, \ldots, g_n$ be $\mscrP$-admissible polynomials which are {\em completely BKK non-degenerate} on $\kk^n$ (\cref{bkk-existence}). Let $h_i := tg_i + (1-t)f_i \in \kk[x_1, \ldots, x_n,t]$, $1 \leq i \leq n$. \Cref{isolated-family} implies that there exists an open subset $\tilde U$ of $\kk^{n+1}$ containing $V_f \times \{0\}$ which satisfies properties \ref{U-tilde-1}--\ref{U-tilde-2} from the proof of \cref{sufficient-lemma}. Let $\tilde Z := V(h_1, \ldots, h_n) \cap \tilde U$ and $H_\epsilon := V(t-\epsilon) \subseteq \kk^{n+1}$, $\epsilon \in \kk$. Let $\bar Z, \bar H_\epsilon$ be the closures respectively of $\tilde Z, H_\epsilon$ in $\pp^n \times \pp^1$. Pick $z \in V_f$ and an irreducible component $Z^*$ of $\bar Z$ containing $(z,0)$. Since $Z^*$ intersects $\bar H_0$ properly, and since $\bar H_0$ is linearly equivalent to $\bar H_1$, it follows that $Z^*$ intersects $\bar H_1$ as well. Pick a branch $B$ of $Z^*$ at $(z^*,1) \in Z^* \cap \bar H_1$. Let $\tilde I := I_B \subseteq [n+1]$,  $\tilde \nu := \nuBii{\tilde I}$ (\cref{branch-remark}), $I^* := \tilde I \cap [n]$ and $\nu^* := \tilde \nu|_{\Aii{I^*}}$. Since $g_j$ and $f_j$ have the same Newton polytope for each $j$, it follows that $\tilde \nu(f_j) = \tilde \nu(g_j) = \nu^*(g_j)$. Since $\tilde \nu(t-1) > 0$, it follows that $\tilde \nu(t) = 0$ and $\In_{\tilde \nu}(h_j|_{\Kii{\tilde I}}) = t\In_{\nu^*}(g_j|_{\Kii{I^*}})$, $1 \leq j \leq n$. \Cref{branch-lemma} then implies that $\In_{\nu^*}(g_1|_{\Kii{I^*}}), \ldots, \In_{\nu^*}(g_n|_{\Kii{I^*}})$ have a common zero in $\Kstarii{I^*}$. But then $\mscrP$ is {\em not} $\Kstarii{I^*}$-trivial (\cref{bkk-lemma}). By our assumption on $I$ it follows that $I^* \subsetneqq I$. Since $z$ is in the closure of $\Kstarii{I^*}$ in $\kk^n$, it follows that $z \not\in \Kstari$, as required. 
\end{proof}


\begin{lemma}[\Cref{weak-bkkriterion}] \label{weak-bkk-lemma}
Let $f_1, \ldots, f_n \in \kk[x_1, \ldots, x_n]$ and the notations be as in \cref{weak-bkkriterion}. Then
\begin{enumerate}
\item \label{NSP-non-degenerate} property \ref{TSPrime-property} of $\mscrP$-non-degeneracy holds iff it holds with $\TSPrime$ replaced by $\NSP$,
\item \label{TSPstar-non-degenerate} property \ref{NTSP-property} of $\mscrP$-non-degeneracy holds iff it holds with $\TSP$ replaced by $\TSPstar$.
\end{enumerate} 
\end{lemma}

\begin{proof}
Since $\NSP \subseteq \TSP'$, for assertion \eqref{NSP-non-degenerate} we only have to show the $(\Leftarrow)$ implication. It follows by \cref{absolute-defn} that for every $I \in \TSP' \setminus \NSP$, there exists $\tilde I \in \NSP$ such that $I \subsetneqq \tilde I$ and $|\NiiP{I^*}| > |I^*|$ for each $I^*$ such that $I \subseteqq I^* \subsetneqq \tilde I$. \Cref{herrero} implies that $|\NiiP{\tilde I}| = |\tilde I|$. Since restricting all $f_j$'s to $\Kii{\tilde I}$ yields a system with the same number of non-zero polynomials as the number of variables, it suffices to prove the following claim: ``if there is $I \subseteq [n]$ such that $|\Ni| > |I|$ and $f_1|_{\Ki}, \ldots, f_n|_{\Ki}$ are BKK degenerate at infinity (resp.\ at $(\overbar \mscrS  \cup \TSP)^I$), then there is $\tilde I \supsetneqq I$ such that $f_1|_{\Kii{\tilde I}}, \ldots, f_n|_{\Kii{\tilde I}}$ are BKK degenerate at infinity (resp.\ at $(\overbar \mscrS  \cup \TSP)^{\tilde I}$).'' This follows by an immediate application of \cref{reduction-lemma}, completing the proof of assertion \eqref{NSP-non-degenerate}. Assertion \eqref{TSPstar-non-degenerate} follows from a similar application of \cref{reduction-lemma}. 
\end{proof}

\Cref{trivial-lemma,trivial-cor} below are key to property \ref{NTSP-property} of $\mscrP$-non-degeneracy. 

\begin{lemma} \label{trivial-lemma}
Assume $\mscrP$ is $\nktorus$-trivial. Let $\nu$ be a monomial valuation in $\V$ centered at $\Kstari$, $I \subseteq [n]$. Assume $\In_\nu(f_1),  \ldots,\In_\nu(f_n)$ have a common zero $a \in \nktorus$. Then $\pi_{I}(a) \in \Kstari$ is a non-isolated point of the zero-set $V(f_1, \ldots, f_n)$ of $f_1, \ldots, f_n$ on $\kk^n$. 
\end{lemma}

\begin{proof}
It is straightforward to check that $\pi_{I}(a)$ is in $V(f_1, \ldots, f_n)$; we only have to show that it is non-isolated in there. \Cref{herrero} implies that there is $J \subseteq [n]$ such that  $p:= \dim(\sum_{j \in J} \scrP_j) <|J|$. Let $\Pi$ be the (unique) $p$-dimensional linear subspace of $\rr^n$ such that $\sum_{j \in J} \scrP_j$ is contained in a translate of $\Pi$. Let $\Pi_0 := \{\alpha \in \Pi: \nu \cdot \alpha = 0\}$ and $r := \dim(\Pi_0)$. Then either $p=r$, or $p = r+1$. We consider these cases separately:\\

{\bf Case 1: $\*{p = r}$.} Let $\alpha_1, \ldots, \alpha_r$ be a basis of $\Pi_0$. In this case $\Pi = \Pi_0$, so that for each $j \in J$, $f_j$ is $\nu$-homogeneous and is a linear combination of monomials in $ x^{\alpha_1}, \ldots, x^{\alpha_r}$. For each $i$, $1 \leq i \leq n$, let $c_i := a^{\alpha_i}$. Let $\tilde Y$ be the sub-variety of $\nktorus$ determined by $x^{\alpha_i} - c_i$, $1 \leq i \leq r$ and $Y$ be the closure of $\tilde Y$ in $\kk^n$. Then 
\begin{prooflist}
\item $Y$ is irreducible of codimension $r$ in $\kk^n$. 
\item $Y \subseteq V(f_j :j \in J)$, since $\tilde Y \subseteq V(f_j :j \in J)$. 
\item \label{C} $\pi_I(a) \in Y$. Indeed, it is clear if $\nu$ is the trivial valuation. Otherwise let $C$ be the curve parametrized by $x_i := a_it^{\nu_i}$, $t \in \kk$. Then $C \cap \nktorus \subseteq \tilde Y$, so that $C \subseteq Y$. Now note that $\pi_I(a) \in C$. 
\end{prooflist}
It follows that one of the irreducible components of $V(f_j :j \in J)$ containing $\pi_I(a)$ has codimension smaller than $|J|$. This implies the lemma in this case.\\

{\bf Case 2: $\*{p = r+1}$.} Choose a basis $\alpha_1, \ldots, \alpha_n$ of $\zz^n$ such that 
\begin{prooflist}[resume]
\item $\alpha_1, \ldots, \alpha_r$ is a basis of $\Pi_0$. 
\item $\alpha_1, \ldots, \alpha_{n-1}$ is a basis of $\nu^\perp := \{\alpha \in \zz^n: \nu \cdot \alpha = 0\}$. 
\item $\alpha_1, \ldots, \alpha_{r}, \alpha_n$ is a basis of $\Pi$. 
\item $\nu \cdot \alpha_n = 1$. 
\end{prooflist}
Let $\tilde Y$ and $Y$ be as in the proof of case 1. Let $y_i := x^{\alpha_i}$, $1 \leq i \leq n$. Then $y_i$'s form a system of coordinates on $\nktorus$ and the projection onto $(y_{r+1}, \ldots, y_n)$ restricts to an isomorphism $\tilde Y \cong (\kk^*)^{n-r}$. Let $\beta_1, \ldots, \beta_n \in \zz^n$ be such that $x_j = \prod_{i=1}^n y_i^{\beta_{j,i}}$, $1 \leq j \leq n$. Then $Y$ is the closure of the image of the map 
\begin{align*}
&(\kk^*)^{n-r} \ni (y_{r+1}, \ldots, y_n) 
	 \mapsto (\tilde c_1 \prod_{i=r+1}^n y_i^{\beta_{1,i}}, \ldots, 
			\tilde c_n \prod_{i=r+1}^n y_i^{\beta_{n,i}}) \in \kk^n, \\
& \text{where}\ \tilde c_j
	= \prod_{i=1}^r c_i^{\beta_{j,i}},\ 1 \leq j \leq n.										
\end{align*}
Let $Y'$ be the closure of the image of the map 
\begin{align*}
&(\kk^*)^{n-r} \ni (y_{r+1}, \ldots, y_n) 
	 \mapsto (\prod_{i=r+1}^n y_i^{\beta_{1,i}}, \ldots, 
			 \prod_{i=r+1}^n y_i^{\beta_{n,i}}) \in \kk^n								
\end{align*}
Then $Y$ is isomorphic to $Y'$ via the map $(x_1, \ldots, x_n) \mapsto (\tilde c_1^{-1}x_1, \ldots, \tilde c_n^{-1}x_n)$. Note that $Y'$ is an affine toric variety corresponding to the semigroup generated by $\beta'_j := (\beta_{j,r+1}, \ldots, \beta_{j,n})$, $1 \leq j \leq n$. For each $j$, $1 \leq j \leq n$,
\begin{align}
\nu(x_j) = \sum_{i=1}^n \beta_{j,i}\nu(y_i) = \beta_{j,n} \label{beta-j-n}
\end{align}
Let $\Theta$ be the coordinate hyperplane of $\zz^{n-r}$ generated by the first $n-r-1$ axes. \eqref{beta-j-n} implies that $\beta'_j$'s lie on $\Theta$ for $j \in I$ and `above' $\Theta$ otherwise. Let $\Sigma$ be the cone in $\rr^{n-r}$ generated by $\beta'_1, \ldots, \beta'_n$. Note that $(0, \ldots, 0, 1)$ is in the dual $\check{\Sigma}$ of $\Sigma$. Choose an edge $E$ of $\check{\Sigma}$ such that 
\begin{prooflist}[resume]
\item \label{E-1} $E$ is an edge of the face of $\check \Sigma$ whose relative interior contains $(0, \ldots, 0, 1)$. 
\item the last coordinate of each non-zero element on $E$ is positive.
\end{prooflist}
Let $Z'$ be the torus invariant divisor of $Y'$ corresponding to $E$ and $Z$ be the isomorphic image of $Z'$ in $Y$. Let $\eta = (\eta_{r+1}, \ldots, \eta_n)$ be the smallest non-zero element on $E$ with integer coordinates. For each $b := (b_{r+1}, \ldots, b_n) \in (\kk^*)^{n-r}$, consider the curve $\gamma_b$ parametrized by $\kk^* \ni t \mapsto (b_{r+1}t^{\eta_{r+1}}, \ldots, b_nt^{\eta_n}) \in (\kk^*)^{n-r}$. Then
\begin{prooflist}[resume]
\item \label{limit-points}  the `limit at zero' of (the image of) $\gamma_b$ is a point on $Z$; we denote it by $\bar \gamma_b(0)$. (The closure of the image of) $\gamma_b$ is transversal to $Z$ at $\bar \gamma_b(0)$. The set of all $\bar \gamma_b(0)$ as $b$ varies over $(\kk^*)^{n-r}$ is open in $Z$. 
\item  It follows that $\ord_Z(y_n|_Y) = \ord_t(y_n|_{\gamma_b}) = \ord_t(b_nt^{\eta_n}) = \eta_n > 0$. 
\end{prooflist}
Fix $i$, $1 \leq i \leq s$. Write $f_i$ as $f_i = f_{i,0} + f_{i,1} + \cdots $, where $f_{i,j}$'s are $\nu$-homogeneous with $\nu(f_{i,0}) < \nu(f_{i,1}) < \cdots$. Let $\alpha \in \supp(f_{i,0})$, then for each $j$, $f_{i,j}/x^\alpha = \sum_\beta \lambda_{i,j,\beta} y_1^{\beta_1} \cdots y_r^{\beta_r}y_n^{\nu(f_{i,j})}$ with $\lambda_{i,j,\beta} \in \kk$. Pick $b \in (\kk^*)^{n-r}$. Since $f_{i,0}(a) = 0$, it follows that 
\begin{align*}
(f_i/x^\alpha)|_{\gamma_b} 
	= \sum_{j,\beta} \lambda_{i,j,\beta} c_1^{\beta_1} \cdots c_r^{\beta_r}(b_nt^{\eta_n})^{\nu(f_{i,j})}
	= \sum_{j > 0} \sum_\beta \lambda_{i,j,\beta} c_1^{\beta_1} \cdots c_r^{\beta_r}b_n^{\nu(f_{i,j})} t^{\nu(f_{i,j})\eta_n}
\end{align*}
Since $\bar \gamma_b(0) \in Z \subseteq \kk^n$, it follows that $\ord_t(x^\alpha) \geq 0$, so that $\ord_t(f_i|_{\gamma_b}) > 0$. Observation \ref{limit-points} then implies that $Z \subseteq V(f_i)$. It follows that $Z \subseteq V(f_1, \ldots, f_n)$. On the other hand, \ref{C} and \ref{E-1} imply that $\pi_I(a) \in Z$. Since $\dim(Z) = n- r - 1 > n-|J|$, this proves the lemma. 
\end{proof}

\begin{cor}  \label{trivial-cor}
Let $I \subseteq J \subseteq [n]$ and $\nu \in \Vii{J}$ with center in $\Kstari$. Assume $\mscrP$ is $\Kstarii{J}$-exotrivial and $\In_\nu(f_1|_{J})  \ldots,\In_\nu(f_n|_{J})$ have a common zero $a \in \nktorus$. Then $\pi_{I}(a) \in \Kstari$ is a non-isolated point of the zero-set $V(f_1, \ldots, f_n)$ of $f_1, \ldots, f_n$ on $\kk^n$. 
\end{cor}

\begin{proof}
Note that $\pi_{I}(a) \in V(f_1, \ldots, f_n)$. It suffices to consider the case that $J$ satisfies property \ref{absolute-equal} of \cref{absolute-defn} (since in the other case the statement is clear). Then there is $\tilde I \supseteq J$ such that $\mscrP$ is $\Kstarii{\tilde I}$-trivial, $|\NiiP{\tilde I}| = |\tilde I|$, and
\begin{itemize}
\item either $\tilde I = J$, 
\item or $|\NiiP{I^*}| > |I^*|$ for each $I^*$ such that $J \subseteq I^* \subsetneqq \tilde I$. 
\end{itemize} 
Replacing $f_j$'s by $f_j|_{\Kii{\tilde I}}$ and then applying \cref{reduction-lemma} we see that the hypothesis of the corollary remains true if we replace $J$ by $\tilde I$ and therefore the corollary reduces to the case that $J = [n]$. Then it follows from \cref{trivial-lemma}. 
\end{proof}

\subsection{Proof of \cref{generic-bkk-thm}} \label{generic-bkk-proof}
Assertion \eqref{generic-assertion-0} follows from a straightforward application of \cref{absolute-prop}. Assertion \eqref{isolated-assertion} follows from \cref{absolute-prop} and \cref{curve-claim}. 

\begin{claim}\label{curve-claim}
Let $\mscrS' := \overbar\mscrS \cup \TSP$. Assume the (Krull) dimension of $V(f_1, \ldots, f_n) \cap \KnSprime$ is non-zero. Then $f_1, \ldots, f_n$ are $\mscrP$-degenerate on $\KnbarS$.
\end{claim}

\begin{proof}
Take an irreducible curve $C \subseteq V(f_1, \ldots, f_n) \cap \KnSprime$ and let $\Ki$ be the smallest coordinate subspace of $\kk^n$ which contains $C$. Since $I \in \TSPrime$, applying \cref{branch-lemma} with a branch $B$ of $C$ at infinity shows that property \ref{TSPrime-property} of $\mscrP$-non-degeneracy fails.
\end{proof}

It remains to prove the second assertion. The $(\Rightarrow)$ implication follows from \cref{sufficiently-bkk} below.  

\begin{claim} \label{sufficiently-bkk}
If $f_1, \ldots, f_n$ and $g_1, \ldots, g_n$ are two systems of $\mscrP$-admissible polynomials such that $\multisof_{\KnSprime} < \multiso{g_1}{g_n}_{\KnSprime}$, then $f_i$'s are $\mscrP$-degenerate on $\KnbarS$. 
\end{claim} 

\begin{proof}
Let $h_i := tg_i + (1-t)f_i \in \kk[x_1, \ldots, x_n,t]$, $1 \leq i \leq n$. Let $V_f$ (resp.\ $V_g$) be the set of isolated zeroes of $f_1, \ldots, f_n$ (resp.\ $g_1, \ldots, g_n$) on $\KnSprime $. \Cref{isolated-family} implies that there exists an open subset $\tilde U$ of $\KnSS{\mscrS'} \times \kk$ which contains $V_f \times \{0\}$ and $V_g \times \{1\}$, and satisfies properties \ref{U-tilde-1} --\ref{U-tilde-2} from the proof of \cref{sufficient-lemma}. Without any loss of generality we may also assume that  
\begin{prooflist}[resume=sufficient-list]
\item \label{U-tilde-4} $\tilde U \cap \{t = 1\}$ (resp.\  $\tilde U \cap \{t = 0\}$) does not intersect any non-isolated zeroes of $g_1, \ldots, g_n$ (resp.\ $f_1, \ldots,f_n$) in $\KnSS{\mscrS'}$ . 
\end{prooflist}
Define $\tilde Z, \bar Z, H_\epsilon, \bar H_\epsilon$ as in the proof of \cref{absolute-cor}. Property \ref{U-tilde-4} implies that 
$$H_1 \cdot \tilde Z = \multiso{g_1}{g_n}_{\KnSS{\mscrS'}}  = \multiso{g_1}{g_n}_{\KnS} > \multisof_{\KnS} = H_0 \cdot \tilde Z$$
Since $\bar H_1 \cdot \bar Z = \bar H_0 \cdot \bar Z$, it follows that there exists $(z,0) \in \bar Z$ such that one of the following holds:
\begin{enumerate}[label= (\alph{enumi})]
\item \label{z-S} $z \in \Kstarii{S}$ for some $S \in \mscrS'$.
\item \label{z-infty} $z \in \pp^n \setminus \kk^n$.
\item \label{z-non-isolated} $z$ is a non-isolated point of $V(f_1, \ldots, f_n) \cap \KnSprime$.  
\end{enumerate}
Pick $z$ satisfying one of the preceding conditions. Let $Z^*$ be an irreducible component of $\tilde Z$ containing  $(z,0)$ and let $B$ be a branch of $Z^*$ at $(z,0)$. Define $\tilde I, \tilde \nu, I^* $ and $\nu^*$ as in the proof of \cref{absolute-cor}. Note that 
\begin{prooflist}[resume]
\item \label{I^*-1} $I^* \not\in  \mscrS'$ (since $Z^* \cap (\Kstarii{I^*} \times \kk)$ non-empty and open in $Z^*$). 
\end{prooflist}
Property \ref{U-tilde-2} from the proof of \cref{sufficient-lemma} implies that there are $\mscrP$-admissible systems of polynomials which have isolated zero(es) on $\Kstarii{I^*}$. \Cref{absolute-cor} then implies that $\mscrP$ is not $\Kstari$-exotrivial. Combining this with observation \ref{I^*-1}, we have
\begin{prooflist}[resume]
\item \label{I^*-2} $I^* \in \TSPrime$. 
\end{prooflist}
Since $g_j$ and $f_j$ have the same Newton polytope for each $j$, it follows that $\tilde \nu(f_j) = \tilde \nu(g_j) = \nu^*(f_j)$. Since $\tilde \nu(t) > 0$, it follows that $\In_{\tilde \nu}(h_j|_{\Kii{\tilde I}}) = \In_{\nu^*}(f_j|_{\Kii{I^*}})$, $1 \leq j \leq n$. \Cref{branch-lemma} then implies that 
\begin{prooflist}[resume]
\item \label{common-observation} $\In_{\nu^*}(f_1|_{\Kii{I^*}}), \ldots, \In_{\nu^*}(f_n|_{\Kii{I^*}})$ have a common zero in $\Kstarii{I^*}$. 
\end{prooflist}
Now note that
\begin{prooflist}[resume]
\item \label{KnS} If \ref{z-S} holds, then $\nu^* \in \VSii{I^*}$, so that \ref{common-observation} implies that $f_1|_{\Kii{I^*}}, \ldots, f_n|_{\Kii{I^*}}$ are BKK degenerate at $\mscrS'^I$.
\item \label{infinity}  If \ref{z-infty} holds, then $-\nu^* \in \wtsii{I^*}$, so that \ref{common-observation} implies that $f_1|_{\Kii{I^*}}, \ldots, f_n|_{\Kii{I^*}}$ are BKK degenerate at infinity. 
\item \label{z-non-isolated-1} If \ref{z-non-isolated} holds, then \cref{curve-claim} implies that $f_1, \ldots, f_n$ are $\mscrP$-degenerate on $\KnbarS$.
\end{prooflist}
Observations \ref{KnS}--\ref{z-non-isolated-1} imply that in every scenario $f_1, \ldots, f_n$ are $\mscrP$-degenerate on $\KnbarS$.
\end{proof}

It remains to prove that $\mscrP$-non-degeneracy on $\KnbarS$ is necessary for the number (counted with multiplicity) of isolated solutions in $\KnSprime$ to be maximal. So take $\mscrP$-admissible $f_1, \ldots, f_n$ which are not $\mscrP$-non-degenerate on $\KnbarS$. We will show that $\multisof_{\KnSprime}  < \multP_{\KnSprime}$.  

\begin{claim} \label{nu-claim}
There exists $I \in \NSP$ and $\nu \in \Vi$ such that $\In_\nu(f_1|_{\Ki}), \ldots, \In_\nu(f_n|_{\Ki})$ have a common zero $z$ in $\Kstari$, and one of the following holds:
\begin{enumerate}[label= (\alph{enumi}$'$)]
\item \label{nu-S} $\nu \in \VSprimei$, 
\item\label{nu-infinity} $-\nu \in \wtsi$, or
\item \label{nu-non-isolated} $\nu$ is centered at $\Kstarii{I'}$ for some $I' \in \TSPrime$ and $\pi_{I'}(z)$ is non-isolated in $V(f_1, \ldots, f_n)$. 
\end{enumerate}
\end{claim}
\begin{proof}
If property \ref{TSPrime-property} of $\mscrP$-non-degeneracy fails, then \cref{weak-bkk-lemma} implies that the claim is true and either \ref{nu-S} or \ref{nu-infinity} holds. On the other hand, if property \ref{NTSP-property} of $\mscrP$-non-degeneracy fails, then there exist $I' \subseteq J \subseteq [n]$ such that $J \in \TSP$, $I ' \in \TSPrime$, and $\In_{\eta}(f_1|_{J})  \ldots,\In_{\eta}(f_n|_{J})$ have a common zero $a \in \nktorus$ for some $\eta \in \Vii{J}$ with center in $\Kstarii{J'}$. \Cref{trivial-cor} implies that $\pi_{I'}(a)$ is a non-isolated point of $V(f_1, \ldots, f_n)$. Now pick the smallest $I \supseteq I'$ such that $|\NiP| = |I|$ (where $\NiP$ is as in \eqref{NiP}). Since $J$ is not exotrivial, it follows that $I \in \NSP$. An application of \cref{reduction-lemma} (with $I = I'$ and $\nu = $ the trivial valuation) implies that there exists $\nu \in \Vi$ centered at $\Kstarii{I'}$ and a common zero $z$ of $\In_\nu(f_1|_{\Ki}), \ldots, \In_\nu(f_n|_{\Ki})$ in $\nktorus$ such that $\pi_{I'}(z) = \pi_{I'}(a)$. Therefore \ref{nu-non-isolated} holds, and it completes the proof of the claim. 
\end{proof}

Let $I,\nu, z$ be as in \cref{nu-claim}. Since $I \in \NSP$, it follows that $|I| = |\NiP|$ . We may assume that $I = \NiP = \{1, \ldots, k\}$. Choose $(\scrP_1, \ldots, \scrP_k)$-admissible polynomials $g_1, \ldots, g_k$ such that
\begin{defnlist}[resume]
\item $g_1|_{\Ki}, \ldots, g_k|_{\Ki}$ are completely BKK non-degenerate (\cref{bkk-defn}).
\item  $\In_\nu(g_j|_{\Ki})(z) \neq 0$ for each $j$, $1 \leq j \leq k$. 
\end{defnlist}

\begin{claim} \label{isolated-claim}
For each $y \in \Kstari$, let $h_{y,i}(x) := g_i(y)f_i(x) - f_i(y)g_i(x)$, $1 \leq i \leq k$. 
\begin{enumerate}
\item \label{dominant-claim} The rational map $\phi: (\kk^*)^k \dashrightarrow (\kk^*)^k$ induced by $ (g_1/f_1, \ldots, g_k/f_k)$ is dominant.
\item There is a non-empty open subset $W$ of $\Kstari$ such that every $y \in W$ is an isolated zero of $h_{y,1}, \ldots, h_{y,k}$. 
\end{enumerate}
\end{claim}

\begin{proof}
\Cref{bkk-deformation} implies that for generic $c := (c_1, \ldots, c_k) \in (\kk^*)^k$, $g_1 - c_1f_1, \ldots, g_k - c_kf_k$ are BKK-non-degenerate, so that $\phi^{-1}(c)$ is finite (\cref{bkk-lemma}). On the other hand, since $\mscrP$ is $\Kstari$-non-trivial, it follows that $\multiso{\scrP_1}{\scrP_k}_{\Kstari} > 0$, so that \cref{sufficiently-bkk,bkk-lemma} imply
that $\phi^{-1}(c)$ is non-empty for generic $c \in(\kk^*)^k$. This implies the first assertion. For the second assertion, take $W$ to be an open subset of $\Kstari \setminus V(f_1\cdots f_k)$ such that $\phi^{-1}(\phi(w))$ is finite for all $w \in U$. 
\end{proof}

Let $W$ be as in \cref{isolated-claim}. Choose $y = (y_1, \ldots, y_k, 0, \ldots, 0) \in W\setminus V(f_1\cdots f_k)$. Let $C$ be the rational curve on $\Ki$ parametrized by $c(t) = (c_1(t), \ldots, c_n(t)) :\kk^* \to \Ki$ defined as in \eqref{c_j}. Then $c(1) = y$. For $c(0)$ there are three possible options: 

\begin{defnlist}[resume]
\item \label{nu-S-1} If \ref{nu-S} holds, then $c(0) \in \Kstarii{S}$ for some $S \in \mscrS'$.
\item \label{nu-infinity-1} If \ref{nu-infinity} holds, then at least one $c_j(t)$ has a pole at $0$; in other words, $c(0)$ is a point at infinity.
\item \label{nu-non-isolated-1} If \ref{nu-non-isolated} holds, then $c(0)$ is a non-isolated zero of $f_1, \ldots, f_n$ in $\kk^n$.
\end{defnlist}
%
Let $V_f$ be the set of isolated zeroes of $f_1, \ldots, f_n$ in $\KnSprime$. Define $h_1, \ldots, h_n$ as in \eqref{h_j}. \Cref{isolated-family} implies that there is an open subset $U^*$ of $\kk^{n+1}$ containing $V_f \times \{0\}$ such that $Z^* := V(h_1, \ldots, h_n) \cap U^*$ has only one dimensional `non-vertical' components. Let $\tilde Z$ be the union of irreducible components of $Z^*$ containing points from $V_f \times \{0\}$. For each $\epsilon \in \kk$, let $H_\epsilon$ be the hyperplane $\{t = \epsilon\}$ in $\kk^{n+1}$. Then for generic $\epsilon \in \kk$, 
\begin{align*}
\multiso{h_1|_{t=\epsilon}}{h_n|_{t=\epsilon}}_{\KnSprime} 
	\geq H_\epsilon \cdot \tilde Z 
	= H_0 \cdot \tilde Z
	= \multisof_{\KnSprime}
\end{align*}
On the other hand, observations \ref{nu-S-1}--\ref{nu-non-isolated-1} imply that $\{(t,c(t)): t \in \kk^*\}$ does not belong to any component of $\tilde Z$. Since for generic $\epsilon \in \kk^*$, $c(\epsilon)$ is an isolated zero of $h_1|_{t=\epsilon}, \ldots, h_n|_{t=\epsilon}$ (\cref{isolated-claim}), it follows that $\multiso{h_1|_{t=\epsilon}}{h_n|_{t=\epsilon}}_{\KnSprime} >  \multisof_{\KnSprime}$, as required. \qed


\subsection{Proof of \cref{non-isolated-cor}}\label{non-isolated-section}
\Cref{b-example}\ref{b-example-generic-1} validates the first assertion, and the second assertion follows from \cref{curve-claim}. \qed

\section{Proof of the intersection multiplicity formula} \label{multiplicity-section}
In this section we prove \cref{multiplicity-thm} using the approach outlined in \cref{stable-remark}. The computation of intersection multiplicity becomes easier if a generic system satisfies a property which is stronger than \eqref{bkk-bar-0}; \cref{pre-multiplicity-section} is devoted to the proof of existence of such systems. The proof of \cref{multiplicity-thm} is then given in \cref{multiplicity-proof-section}. 

\subsection{Strongly $\mscrG$-non-degenerate systems} \label{pre-multiplicity-section}
\begin{defn} \label{strongly-generic-defn}
Let $\mscrG:= (\Gamma_1, \ldots, \Gamma_n)$ be a collection of $n$ diagrams in $\rr^n$ and $f_1, \ldots, f_n$ be polynomials in $(x_1, \ldots, x_n)$. We say that $f_i$'s are {\em strongly $\mscrG$-non-degenerate} iff
\begin{defnlist}
\item \label{admissible}  $f_i$'s are $\mscrG$-admissible.
\item \label{strongly-non-degenerate} For all non-empty $J \subseteq [n]$, $\{f_j: j \in J\}$ are {\em completely BKK non-degenerate} (\cref{bkk-defn}). 
\item \label{intersectionally-non-degenerate} For all non-empty $I, J, J' \subseteq [n]$ such that $|J| = n - |I|$, $\Ki$ is an irreducible component of $V(f_j: j \in J)$, $J' \subseteq [n]\setminus J$ and $|J'| = k-1$, the following holds: 
\begin{align}
\parbox{.57\textwidth}{If $V$ is any irreducible component of $V(f_j: j \in J)$ distinct from $\Ki$, then $V \cap \Ki \cap V(f_{j'}: j' \in J')$ is finite.} \label{I-J-J'}
\end{align}

\item \label{recursively-non-degenerate} For all $I\subsetneqq [n]$ and $J \subseteq [m]$ such that $|I| = |J|$, define 
\begin{align*}
& f_{I,j} := f_j|_{x_{i'_1} = 1, \ldots, x_{i'_{k'}} =1},\ \text{for all $j \in J$,}
\end{align*}
where $i'_1, \ldots, i'_{k'}$ are elements of $[n]\setminus I$. Then $(f_{I,j}: j \in J)$ are completely BKK non-degenerate.
\end{defnlist}
\end{defn}

We now show the existence of strongly non-degenerate polynomials. Let $\mscrG:= (\Gamma_1, \ldots, \Gamma_n)$ be a collection of $n$ diagrams in $\rr^n$. Fix polytopes $\scrP_j$ such that $\Gamma_j$ is the Newton diagram of a polynomial with Newton polytope $\scrP_j$, $1 \leq j \leq n$. Let $\scrA_j$ be the space of polynomials with support in $\scrP_j$, $1 \leq j \leq n$, and let $\scrA := \prod_{j=1}^n \scrA_j$. 

\begin{lemma} \label{strong-existence}
Assume $\multzeroGamma < \infty$. Then the set of strongly $\mscrG$-non-degenerate systems of polynomials contains a non-empty Zariski open subset of $\scrA$. 
\end{lemma}

\begin{proof}
\Cref{bkk-existence} implies that the set of systems of polynomials which satisfy properties \ref{admissible}, \ref{strongly-non-degenerate} and \ref{recursively-non-degenerate} contains a non-empty Zariski open subset of $\scrA$. We now show that the same is true for polynomials satisfying property \ref{intersectionally-non-degenerate}. \\

Fix $I, J, J' \subseteq [n]$ as in condition \ref{intersectionally-non-degenerate} of \cref{strongly-generic-defn}. It suffices to show that there is a non-empty Zariski open set $U_{I,J,J'}$ of $\scrA$ such that every $(f_1, \ldots, f_n) \in U_{I,J,J'}$ satisfies \eqref{I-J-J'}. We may assume that $I = \{1, \ldots, k\}$. 
Let $A_J$ be the coordinate ring of $\scrA_J := \prod_{j \in J} \scrA_j$ and $K_J$ be the field of fractions of $A_J $. Let $R := K_J(x_1, \ldots, x_k)[x_{k+1}, \ldots, x_n]$ and $\mmm$ be the maximal ideal of $R$ generated by $x_{k+1}, \ldots, x_n$. Since $V(f_j: j \in J)$ is a complete intersection near generic points of $\Ki$, it follows that the ideal $\qqq_J$ generated by $f_j$, $j \in J$ has finite codimension in $R_{\mmm}$. As in the proof of \cref{generic-length}, compute a standard basis of $\qqq_J$ in $R_{\mmm}$ and consider the product $h \in A_J[x_1, \ldots, x_k]$ of all the coefficients of all the monomials in $(x_{k+1}, \ldots, x_n)$ that appear in the basis elements. Let $W$ be the zero set of $h$ in $\Ki \times \scrA_J$. If $(x, (f_j: j \in J)) \in (\Ki \times \scrA_J) \setminus W$, then $\Ki$ is the only irreducible component of $V(f_j: j \in J)$ containing $x$. \\

Let $\pi_J: W \to \scrA_J$ be the natural projection. If $\pi_J$ is not dominant, then we are done. So assume there is an irreducible component $W'$ of $W$ which projects dominantly to $\scrA_J$. Let
\begin{align*}
Y := \{(x, (f_j: j \in J), (f_{j'}: j' \in J')): x \in V(f_{j'}: j' \in J')\} \cap (W' \times \scrA_{J'})
\end{align*}
where $\scrA_{J'} := \prod_{j' \in J'} \scrA_{j'}$. Let $\pi: Y \to \scrA_J \times \scrA_{J'}$ be the natural projection. It suffices to show that $\pi$ is generically finite-to-one. Fix a generic $f_J := (f_j :j \in J) \in \scrA_J$ such that $h|_{\Ki \times \{f_J\}} \not\equiv 0$. Then $Z := \dim(\pi_J^{-1}(f_J)) = k-1$. Since $|J'| = k-1$, it is straightforward to see that $Z \cap V(f_{j'} : j' \in J')$ is finite for generic $(f_{j'}: j' \in J')$, as required. 
\end{proof}

\subsection{Proof of \cref{multiplicity-thm}} \label{multiplicity-proof-section}
Due to \cref{generic-thm,strong-existence,bkk-lemma}, it suffices to prove the following:

\begin{thm}\label{multiplicity-thm'}
Let $\mscrG := (\Gamma_1, \ldots, \Gamma_n)$ be a collection of $n$ diagrams in $\rr^n$. Let $f_1, \ldots, f_n$ be strongly $\mscrG$-non-degenerate. Then  
\begin{align}
\multzerof
	&:= 	\sum_{I \in \scrI_{\mscrG,1}}
				\multzero{\pi_{I'}(\Gamma_{j_1})}{\pi_{I'}(\Gamma_{j_{n-k}})} \times
				\multstar{\Gamma^{I}_1, \Gamma^{I}_{j'_1}}{\Gamma^{I}_{j'_{k-1}}}  
	\label{mult-formula'}
\end{align}
where the right hand side of \eqref{mult-formula'} is as in \cref{multiplicity-thm}.  
\end{thm} 

\begin{proof}
We prove \cref{multiplicity-thm'} by induction on $n$. It suffices to treat the case that $\multzeroGamma < \infty$. The theorem is true for $n = 1$ (see \cref{formula-convention}), so assume it is true for all dimensions smaller than $n$. It is straightforward to check that if $0 \in \Gamma_j$ for some  $j$, $1 \leq j \leq n$, then both sides of \eqref{mult-formula'} are zero. So assume $0 \not\in \Gamma_j$ for all $j$, $1 \leq j \leq n$.\\

Let $\scrZ_0$ be the union of irreducible components of $V(f_2, \ldots, f_n)$ that contain the origin. Since $0 < \multzerof < \infty$ (combine \cref{generic-thm,bkk-lemma}) it follows that $\dim(\scrZ_0) = 1$. For each $I \subseteq [n]$, let $\Zzeroi := \overline{\scrZ_0 \cap \Kstari}$, and let $\Bzeroi$ be the (possibly empty) set of branches of $\Zzeroi$. For each $\nu \in \Vzeroi$, let $\Bzeroinu$ be the set of all branches $B \in \Bzeroi$ such that $\nuBi = \nu$. Then
\begin{align}
\multzerof
	&= \sum_{I \subseteq [n]} \sum_{\nu \in \Vzeroi} \sum_{B \in \Bzeroinu} \ord_B(f_2, \ldots, f_n)\nu_B(f_1) \label{all-Bzeroinu-sum}
\end{align}
where $\ord_B(f_2, \ldots, f_n)$ is the intersection multiplicity of $f_2, \ldots, f_n$ along $B$. We now compute the right hand side of \eqref{all-Bzeroinu-sum}.

\begin{lemma}
Let $I$ be a non-empty subset of $[n]$ and  $\nu \in \Vzeroi$. 
\begin{enumerate}
\item \label{Bzeroinu-non-empty} If $\Bzeroinu \neq \emptyset$ then $I \in \scrI_{\mscrG,1}$. 
\item \label{I_1} Assume $I \in \scrI_{\mscrG,1}$. Then
\begin{align}
&\nu_B(f_1) = \nu(f_1|_{\Ki})\ \text{for each $B \in \Bzeroinu$, and} \label{nu_B-f_1} \\
&\sum_{B \in \Vzeroi} \ord_B(f_{j'_1}|_{\Ki}, \ldots, f_{j'_{k-1}}|_{\Ki}) 
	=  \mv(\In_\nu(\Gamma^I_{j'_1}), \ldots, \In_\nu(\Gamma^I_{j'_{k-1}}) ) \label{Bzeroinu-sum}
\end{align} 
where $j'_1, \ldots, j'_{k-1}$ are defined as in \cref{multiplicity-thm}. 
\end{enumerate}
\end{lemma}

\begin{proof}
At first we prove assertion \ref{Bzeroinu-non-empty}. Let 
\begin{align*}
J  &:= \{j: 2 \leq j \leq n,\  f_j|_{\Ki} \equiv 0\} = \{j: 2 \leq j \leq n,\  \Gamma^I_j = \emptyset \} \\
J' &:= \{2, \ldots, n\} \setminus J
\end{align*}
Pick $B \in \Bzeroinu$. Since $\In(B) \in \Kstari \cap V(\In_\nu(f_2|_{\Ki}), \ldots, \In_\nu(f_n|_{\Ki}))$ (\cref{branch-lemma}), property \ref{strongly-non-degenerate} of \cref{strongly-generic-defn} implies that
\begin{align*}
|I| - 1 \geq \dim(\In_\nu(\sum_{j \in J'} \np(f_j|_{\Ki})))  \geq |J'| = n - 1 - |J|,
\end{align*}
so that $|J| \geq n-|I|$. On the other hand since $\mscrG$ is $I$-isolated at the origin (\cref{mult-support-solution}), it follows that $|J|\leq n- |I|$. Therefor $|J| = n - |I|$. The $I$-isolation at the origin then also implies that $f_1|_{\Ki} \not\equiv 0$, so that $I \in \scrI_{\mscrG,1}$, as required.\\

Now we prove assertion \ref{I_1}. We may assume $J' = \{2, \ldots, k\}$. Let $\Xnu$ be the $\nu$-weighted blow up $\sigma_\nu: \Xnu \to \Ki$ at the origin and let $U := \Kstari \cup \Enustar \subseteq \Xnu$ (\cref{blow-up-defn}). Then $U$ is open in $\Xnu$. For each $j$,  $1 \leq j \leq k$, let $D_j$ be the divisor on $\Ki$ determined by $f_j$, and $D'_j$ be the strict transform of $D_j$ on $\Xnu$. The complete BKK non-degeneracy of $f_2, \ldots, f_n$ implies that there are only finitely many points in $\Enu$ common to $D'_2, \ldots, D'_k$, and all these points are in  $\Enustar$; moreover, these are precisely the points on the strict transforms of the branches in $\Bzeroinu$. Finally, by the BKK non-degeneracy at the origin, none of the points in $\Enustar \cap D'_2 \cdots \cap D'_k$ belong to $D'_1$; in particular this implies identity \eqref{nu_B-f_1}. It follows that
\begin{align*}
&\sum_{B \in \Vzeroi} \nu_B(f_1)\ord_B(f_{j'_1}|_{\Ki}, \ldots, f_{j'_{k-1}}|_{\Ki}) \\
&\qquad\qquad	
	= (\sigma_{\nu}^*(D_1)|_U, D'_2|_U, \ldots, D'_n|_U) 
	= \nu(f_1|_{\Ki})(\Enustar + D'_1|_U, D'_2|_U, \ldots, D'_n|_U) \\
&\qquad\qquad	
	= \nu(f_1|_{\Ki})(\Enustar, D'_2|_U, \ldots, D'_n|_U) 
	= \nu(f_1|_{\Ki})( D'_2|_{\Enustar}, \ldots, D'_n|_{\Enustar}) \\
&\qquad\qquad	
	= \nu(f_1|_{\Ki}) \mult{\In_\nu(f_2|_{\Ki})}{\In_\nu(f_k|_{\Ki})}_{\Kstari}
	=  \nu(f_1|_{\Ki})\mv(\In_\nu(\Gamma^I_2, \ldots, \In_\nu(\Gamma^I_k) ),
\end{align*}
where the last equality follows from Bernstein's theorem. This proves identity \eqref{Bzeroinu-sum}.
\end{proof}

Now let $B \in \Bzeroinu$, where $I \in \scrI_{\mscrG,1}$ and $\nu \in \Vzeroi$. Let $j_1, \ldots, j_{n-k}$ (resp.\ $j'_1, \ldots, j'_{k-1}$) be the elements of $J_I$ (resp.\ $\{2, \ldots, n\}\setminus J_I$). Let $D_j$ be the Cartier divisor on $\kk^n$ defined by $f_j$, $2 \leq j \leq n$. Since $D_j$'s intersect properly at $B$, $\ord_B(f_2, \ldots, f_n)$ is the coefficient of (the closure of) $B$ in the (proper) intersection product $D_2 \cdot D_3 \cdots D_n$ = $(D_{j_1} \cdots D_{j_{n-k}})(D_{j'_1} \cdots D_{j'_{k-1}})$. Note that 
\begin{itemize}
\item $\Ki$ is an irreducible component of $D_{j_1} \cdots D_{j_{n-k}}$, and 
\item property \ref{intersectionally-non-degenerate} of strong non-degeneracy ensures that $B$ is not contained in any other irreducible component of $D_{j_1} \cdots D_{j_{n-k}}$.
\end{itemize}
It follows that 
\begin{align}
\ord_B(f_2, \ldots, f_n) = 
	\ord_{\Ki}(f_{j_1}, \ldots, f_{j_{n-k}}) \ord_B(f_{j'_1}|_{\Ki}, \ldots, f_{j'_{k-1}}|_{\Ki})
\end{align}
Let $i'_1, \ldots, i'_{n-k}$ be the elements of $I' = [n]\setminus I$. As in property \ref{recursively-non-degenerate} of strong non-degeneracy, let $f_{I',j_1}, \ldots, f_{I',j_{n-k}}$ be the polynomials in $(x_{i'}: i' \in I')$ formed by specializing $(x_i: i \in I)$ to $(1, \ldots, 1)$. Then $f_{I',j_1}, \ldots, f_{I',j_{n-k}}$ are completely BKK non-degenerate. \Cref{bkk-deformation} implies that complete BKK non-degeneracy is preserved by specialization of $(x_i: i \in I)$ in the expressions for $f_{j_1}, \ldots, f_{j_{n-k}}$ to generic $\xi \in (\kk^*)^k$. \Cref{generic-length} and the inductive hypothesis then imply that 
\begin{align}
\ord_{\Ki}(f_{j_1}, \ldots, f_{j_{n-k}}) = \multzero{\pi_{I'}(\Gamma_{j_1})}{\pi_{I'}(\Gamma_{j_{n-k}})} \label{ord-Ki}
\end{align}
\Cref{multiplicity-thm'} follows from combining identities \eqref{all-Bzeroinu-sum}--\eqref{ord-Ki}. 
\end{proof}

\section{Proof of the extended BKK bound} \label{bkk-proof-section}
In this section we prove \cref{bkk-thm} following the same approach of the proof of \cref{multiplicity-thm} in \cref{multiplicity-section}. Throughout this section $\mscrP:= (\scrP_1, \ldots, \scrP_n)$ is a collection of $n$ convex integral polytopes in $\rr^n$.

\begin{defn} \label{strongly-bkk-non-degenerate}
 Given $f_1, \ldots, f_n \in \kk[x_1, \ldots, x_n]$, we say that they are {\em strongly $\mscrP$-non-degenerate} iff they are $\mscrP$-admissible and satisfy properties \ref{strongly-non-degenerate}, \ref{intersectionally-non-degenerate} and \ref{recursively-non-degenerate} of \cref{strongly-generic-defn}. 
\end{defn}

\Cref{strong-existence} shows that generic systems of polynomials with support $\mscrP$ are strongly $\mscrP$-non-degenerate. We prove \cref{bkk-thm} via the same arguments as in the proof of \cref{multiplicity-thm}. In particular, we show the following: if $f_1, \ldots, f_n$ are strongly $\mscrP$-non-degenerate, then 
\begin{align}
\multisof_{\KnbarS} 
		&= 	\sum_{I \in \ISPone}
					\multzero{\pi_{I'}(\Gamma_{j_1})}{\pi_{I'}(\Gamma_{j_{n-k}})} \times
					\multstarinftySS{\scrP^{I}_1, \scrP^{I}_{j'_1}}{\scrP^{I}_{j'_{k-1}}}{(\overbar \mscrS \cup \SP)^I}
		\label{bkk-formula'}
\end{align}
where the right hand side of \eqref{bkk-formula'} is as in \cref{bkk-thm}.  \\

Let $\mscrS' := \overbar \mscrS \cup \SP$. Define
\begin{align*}
\ISP &:= \{I \subseteq [n]:  I \not\in \overbar \mscrS \cup \SP \cup \{\emptyset\},\ |\NiP| = |I|,\ \text{$\mscrP$ is $\Kstari$-non-trivial} \} 
\end{align*}
where $\NiP$ is as in \eqref{NiP}, i.e.\ $\ISPone$ is simply the subset of $\ISP$ consisting of subsets of $[n]$ containing $1$. \Cref{herrero}, \cref{Pisolated-cor} and assertion \ref{bkk-trivial} of \cref{bkk-lemma} imply that
\begin{prooflist}
\item \label{isolated-observation-1}  all roots of $f_1, \ldots, f_n$ in $X^{\ISP} := \bigcup_{I \in \ISP}\Kstari$ are isolated, 
\item \label{isolated-observation-2} and in turn each isolated root of $f_1, \ldots, f_n$ is contained in $X^{\ISP}$. 
\end{prooflist}
Let $\scrZ$ be the union of irreducible components of $V(f_2, \ldots, f_n)$ on $X^{\ISP}$. Pick an irreducible component $Z$ of $\scrZ$ and the the smallest $I \in \ISP$ such that $\Kstari$ contains $Z$.

\begin{claim}\label{dimension-claim}
$\dim(Z) = 1$ and $I \in \ISPone $.  
\end{claim}

\begin{proof}
It follows from definition of $\scrI$ that $|\NiP| = |I|$. If $\NiP \subseteq \{2, \ldots, n\}$, then \cref{bkk-lemma} implies that $V(f_2, \ldots, f_n) \cap \Kstari$ is finite. Since $\dim(Z) \geq 1$, this contradicts the construction of $I$. It follows that $1 \in \NiP$, so that $I \in \ISPone $. This implies that $|\NiP\setminus \{1\}| = |I| - 1$. The same arguments as in the proof of \cref{bkk-dimension} then show that $\dim(V(f_j: 2 \leq j \leq n) \cap \Kstari) = 1$. 
\end{proof}
For each $I \in \ISPone$, let $\Zi := \overline{\scrZ \cap \Kstari}$, and let $\Bi$ be the set of branches of $\Zi$. For each $\nu \in \VSprimei$ (resp.\ $\omega \in \wtsi$), let $\Binu$ (resp.\ $\Biomega$) be the set of all branches $B \in \Bi$ such that $\nuBi = \nu$ (resp.\ $\nuBi = -\omega$). Observations \ref{isolated-observation-1}, \ref{isolated-observation-2} and \cref{dimension-claim} imply that
\begin{align*}
\multisof_{\KnbarS}
	&= -\sum_{I \in \ISPone} \left(
	\sum_{\omega \in \wtsi} \sum_{B \in \Biomega} \ord_B(f_2, \ldots, f_n)
	+\sum_{\nu \in \VSprimei} \sum_{B \in \Binu} \ord_B(f_2, \ldots, f_n)
	\right) \nu_B(f_1)
\end{align*}
Identity \eqref{bkk-formula'} now follows exactly in the same way as the proof of \eqref{mult-formula'} from \eqref{all-Bzeroinu-sum}. \qed

\appendix
 
\section{Existence and deformations of non-degenerate systems} \label{bkk-section}
In this section we establish existence of non-degenerate (with respect to a given collection of diagrams or polytopes) systems of polynomials and show that generic systems are non-degenerate. Throughout this section $f_1, \ldots, f_m$, $m \geq 1$, denote polynomials in $(x_1, \ldots, x_n)$ and $\Gamma_i$ (resp.\ $\scrP_i$) denote the Newton diagram (resp.\ Newton polytope) of $f_i$, $i = 1, \ldots, n$.

\begin{defn} \label{bkk-defn}
 \mbox{}
\begin{itemize}
\item We say that $f_1, \ldots, f_m$ are {\em completely BKK non-degenerate} iff for all $J \subseteq [m]$, $\In_\nu(f_j)$, $j \in J$, have no common zero in $\nktorus$ for all $\nu \in \V$ such that $\dim(\In_\nu(\sum_{j \in J} \np(f_j))) < |J|$. 
\item Let $\mscrS$ be a collection of subsets of $[n]$. We say that $f_1, \ldots, f_m$ are {\em completely BKK-non-degenerate on $\kk^n$} iff $\{f_j|_{\Ki}: f_j|_{\Ki} \not\equiv 0\}$ are completely BKK non-degenerate for every non-empty subset $I$ of $[n]$. 
\end{itemize}
\end{defn}


\begin{lemma} \label{bkk-lemma}
Let $\mscrG := (\Gamma_1, \ldots, \Gamma_n)$ and $\mscrP := (\scrP_1 \ldots, \scrP_n)$. 
\begin{enumerate}
\item \label{bkk-to-G} Assume $\multzero{\Gamma_1}{\Gamma_n} < \infty$. If $f_1, \ldots, f_n$ are completely BKK non-degenerate on $\kk^n$, then they are $\mscrG$-non-degenerate.
\item 
\begin{enumerate}
\item \label{bkk-finite} If $f_1, \ldots, f_n$ are completely BKK non-degenerate, then $V(f_1, \ldots, f_n) \cap \nktorus$ is finite. 
\item \label{bkk-to-P} Let $\mscrS$ be a collection of subsets of $[n]$. If $f_1, \ldots, f_n$ are completely BKK non-degenerate on $\kk^n$, then they are $\mscrP$-non-degenerate on $\KnbarS$. 
\end{enumerate}
\item \label{bkk-trivial} Assume $f_1, \ldots, f_n$ are completely BKK non-degenerate on $\kk^n$. Let $I \subseteq [n]$ be such that $\mscrP$ is $\Kstari$-trivial. Then $V(\In_{\nu}(f_1|_{\Ki}), \ldots, \In_{\nu}(f_n|_{\Ki})) \cap \nktorus = \emptyset$ for all $\nu \in \Vi$. 
\end{enumerate}
\end{lemma}

\begin{proof}
Assertions \eqref{bkk-to-G} and \eqref{bkk-to-P} follow immediately from definitions. Assertion \eqref{bkk-finite} is a consequence of the following observation: if $V(f_1, \ldots, f_n) \cap \nktorus $ contains a curve $C$ and $\nu$ is the monomial valuation corresponding to a branch $B$ of $C$ at infinity, then $V(\In_\nu(f_1), \ldots, \In_\nu(f_n)) \cap \nktorus \neq \emptyset $ (\cref{branch-lemma}). Assertion \eqref{bkk-trivial} follows from \cref{bkk-defn,herrero}. 
\end{proof}

Let $A_1, \ldots, A_m$ be finite subsets of $\zz^n$. For each $j$, $1 \leq j \leq m$, we denote by $\scrA_j$ the space of polynomials $f$ such that $\supp(f) \subseteq A_j$. Similarly, given $J \subseteq [m]$, we denote by $\scrA_J$ the space of $|J|$-tuples $(f_j: j \in J)$ of polynomials such that $\supp(f_j) \subseteq A_j$ for each $j \in J$. Note that $\scrA_J \cong \prod_{j \in J} \scrA_j$. Define $\scrP_j := \conv(A_j)$, $j \in J$. Let $I \subseteq [n]$ be such that $\scrP^I_j \neq \emptyset$ for all $j \in J$. Define $\scrP^I_J := \sum_{j \in J} \scrP^I_j$. 

\begin{lemma} \label{non-degenerate-lemma}
Let $\nu \in \Vi$ be such that $d := \dim(\In_\nu(\scrP^I_J)) < |J|$. Then there is a non-empty Zariski open subset $U$ of $\scrA_J$ such that $V(\In_\nu(f_j|_{\Ki}): j \in J) \cap \nktorus = \emptyset$ for all $(f_j: j \in J) \in U$. 
\end{lemma}

\begin{proof}
We proceed by induction on $|J|$. It is clear for $|J| = 1$. Now assume $|J| \geq 2$. Let  Pick $\tilde J \subset J$ such that 
\begin{itemize}
\item $|\tilde J| = |J|-1$,
\item $\dim(\In_\nu(\scrP^I_{\tilde J})) = d$.
\end{itemize}
Pick $\eta \in \Vi$ such that $\In_\eta(\scrP^I_{\tilde J})$ is a proper face of $\In_\nu(\scrP^I_{\tilde J})$. Let $\tilde d := \dim(\In_\eta(\scrP^I_{\tilde J}))$. Since $\tilde d <  |\tilde J|$, it follows by induction that there is an open set $U_\eta$ of $\scrA_{\tilde J}$ such that $V(\In_\eta(f_j|_{\Ki}): j \in \tilde J) \cap \nktorus = \emptyset$ for all $(f_j: j \in \tilde J) \in U_\eta$. Let $\tilde U$ be the intersection of all such $U_\eta$. 

\begin{proclaim} \label{bkk-dimension}
$V(\In_\nu(f_j|_{\Ki}): j \in \tilde J) \cap \nktorus$ is a finite set. 
\end{proclaim} 

\begin{proof}
Indeed, it is straightforward to see that the $V(\In_\nu(f_j|_{\Ki}): j \in \tilde J) \cong \tilde V \times (\kk^*)^s$ for some $s \geq 0$, where $\tilde V$ is the zero set of $(f_j: j \in \tilde J)$ on the torus $T_\nu$ of the toric variety $X_\nu$ determined by $\In_\nu(\scrP^I_{\tilde J})$. Since each proper face of $\In_\nu(\scrP^I_{\tilde J})$ corresponds to an irreducible component of $X_\nu \setminus T_\nu$, it follows that the closure of $\tilde V$ in $X_\nu$ is contained in $X_\nu \setminus T_\nu$, and therefore must be finite. 
\end{proof}
Let $W:= \{(f_j: j \in J) \in \scrA_J: (\In_\nu(f_j|_{\Ki}): j \in J) \cap \nktorus = \emptyset\}$. Let $j$ be the unique element of $J \setminus \tilde J$. It follows from the claim that for each $(f_j: j \in \tilde J) \in \tilde U$, there is an open subset $U_j$ of $\scrA_j$ such that $\tilde U \times U_j \subseteq W$. Since $W$ is constructible, it follows that $W$ must contain an open subset of $\scrA_J$, as required. 
\end{proof}


\begin{cor} \label{bkk-existence}
There is a non-empty open subset $U$ of $\scrA_{[m]}$ such that all $(f_1, \ldots, f_m) \in U$ are completely BKK non-degenerate on $\kk^n$. In particular, if $m = n$, and $\mscrG$ and $\mscrP$ are as in \cref{bkk-lemma}, then
\begin{enumerate}
\item all $(f_1, \ldots, f_n) \in U$ are $\mscrP$-non-degenerate on $\kk^n$,
\item if $0 < \multzeroGamma < \infty$, then all $(f_1, \ldots, f_n) \in U$ are $\mscrG$-non-degenerate.  \qed
\end{enumerate}
\end{cor}

The following lemma studies deformations of completely BKK non-degenerate systems. Pick polynomials $g_j$ such that $\np(g_j) = \np(f_j)$, $1 \leq j \leq m$. Let $h_j := f_j + \phi_j(t_j)g_j \in \kk[x_1, \ldots, x_n, t_j]$, $1 \leq j \leq m$, where $t_j$'s are indeterminates and each $\phi_j$ is a rational function in $t_j$ such that $\phi_j(t_j) = 0$. Write $h_{\epsilon,j}$ for $h_j|_{t_j = \epsilon}$ for each $j \in [m]$ and $\epsilon \in \kk$. 

\begin{lemma} \label{bkk-deformation}
Assume $f_1, \ldots, f_m$ are completely BKK non-degenerate on $\kk^n$. Then there is a non-empty Zariski open subset $U$ of $\kk^m$ such that $h_{\epsilon_1,1}, \ldots, h_{ \epsilon_m, m}$ are completely BKK non-degenerate on $\kk^n$ for all $(\epsilon_1, \ldots, \epsilon_m) \in U$.    
\end{lemma}

\begin{proof}
Pick $J \subseteq [m]$, $I \subseteq [n]$ and $\nu \in \Vi$ such that $\scrP^I_j \neq \emptyset$ for all $j \in J$, and $d := \dim(\In_\nu(\scrP^I_J)) < |J|$. It suffices to show that there is a non-empty open subset $U$ of $\kk^m$ such that $V(\In_\nu(h_{\epsilon_j,j}|_{\Ki}): j \in J) \cap \nktorus = \emptyset$ for all $(\epsilon_1, \ldots, \epsilon_m) \in U$. This follows by induction on $|J|$ exactly as in the proof of \cref{non-degenerate-lemma}. 
\end{proof}

\begin{defn} \label{essentially-complete-defn}
Let $f_1, \ldots, f_m$ be rational functions on a variety $X$ and $U$ be a non-empty open subset of $\kk^n$. We say that $f_1, \ldots, f_m$ satisfy the {\em essentially complete intersection property on $U$} iff each $f_j$ is regular on $U$, and for each $J \subseteq [m]$, $V(f_j: j \in J) \cap U$ has codimension $|J|$ in $U$ (see \cref{coconvention}).
\end{defn}

\begin{lemma}
Let $f_1, \ldots, f_m \in \kk[x_1, \ldots, x_n]$ and $U$ be an open subset of $\kk^n$ such that $f_j$'s satisfy the essentially complete intersection property on $U$. Let $N$ be a positive integer, $t$ be an indeterminate, and for each $i,j$, $1 \leq i \leq N$, $1 \leq j \leq m$,  let $f_{i,j} \in \kk[x_1, \ldots, x_n]$ and $\phi_{i,j} \in \kk(t)$ be such that $\phi_{i,j}$ is regular at $0$ with $\phi_{i,j}(0) = 0$. Define $g_j(x,t) := f_j + \sum_{i=1}^N \phi_{i,j}(t)f_{i,j}$, $1 \leq j \leq N$. Then there exists an open neighborhood $\tilde U$ of $U \times \{0\}$ in $U \times \kk$ such that $g_1, \ldots, g_n, t-\epsilon$ satisfy the essentially complete intersection property on $\tilde U$ for all $\epsilon \in \kk$. 
\end{lemma}

\begin{proof}
We prove by induction on $k$ that for each $k$, $1 \leq k \leq m$, there is an open neighborhood $\tilde U_k$ of $(P,0)$ in $U \times \kk$ such that $g_1, \ldots, g_k, f_{k+1}, \ldots, f_m, t-\epsilon$ satisfy the essentially complete intersection property on $\tilde U_k$ for generic $\epsilon \in \kk$. \\

Pick $J \subseteq \{2, \ldots, n\}$. Let $V_J := V(f_j: j \in J) \subset U \times T$, where $T$ is an open neighborhood of $0$ in $\kk$ on which each $\phi_{i,j}$ is regular. At first assume $V_J \neq \emptyset$. The ideal of each irreducible component of $V_J$ is generated by polynomials in $(x_1, \ldots, x_n)$. Since $f_1$ is not in any of these ideals, it follows that $g_1$ is not in any of these ideals as well (here we need to use that $t$ divides each $\phi_{i,1}(t)$). Therefore $V(g_1) \cap V_J$ is either empty or has pure dimension $n-|J|$ in $U \times T$. Since $\dim(V_J \cap \{t = \epsilon\}) = n-|J|$ for each $\epsilon$, it follows that there is an open neighborhood $T_J$ of $0$ in $T$ such that $V(g_1) \cap V_J$ does not contain any component of $V_J \cap \{t_1 = \epsilon\}$ for any $\epsilon \in T_J$. If $V_J = \emptyset$, then set $T_J := \kk$. Then the claim holds for $k = 1$ with $\tilde U_1 := U \times \bigcap_J T_J$. \\

Now assume the claim is true for $k$. It suffices to show that for every $l$, $k+1 \leq l \leq m$, there is a neighborhood $U^*_l$ of $U \times \{0\}$ such that $V(g_1, \ldots, g_{k+1},  f_{k+2}, \ldots, f_l, t- \epsilon) \cap U^*_l$ is either empty or has codimension $l$ in $\kk^n$ for every $\epsilon \in \kk$. By inductive hypothesis, there is a neighborhood $\tilde U_k$ of $U \times \{0\}$ such that $\tilde V_l := V(g_1, \ldots,  g_k, f_{k+2}, \ldots, f_l) \cap \tilde U_k$ is empty or has dimension $n-l+2$. We may assume $\tilde V_l \neq \emptyset$, since otherwise the claim is trivially true. Let $V'_l$ be the union of all components of $\tilde V_l$ that intersect $U \times \{0\}$. This means that there is an open neighborhood $U'_l$ of $U \times \{0\}$ in $\tilde U_k$ such that $U'_l\cap\tilde V_l = V'_l$. For each irreducible component $Y$ of $V'_l$, $\dim(Y \cap U \times \{0\}) \geq \dim(Y) - 1 = n -l + 1$. Since dimension of every component of $V_l := V(f_1, \ldots, f_k, f_{k+2}, \ldots, f_l) \subseteq \kk^n$ is $n-l+1$, it follows that every component of $V'_l$ contains a component of $V_l \times \{0\}$. Since $V(g_{k+1})$ does not contain any component of $V_l \times \{0\}$, it follows that $V(g_{k+1}) \cap V'_l$ is either empty or has pure dimension $n-l+1$. The claim then follows as in the proof of $k = 1$ case. 
\end{proof}

\begin{cor} \label{isolated-family}
Let $f_1, \ldots, f_n \in \kk[x_1, \ldots, x_n]$ and $P \in \kk^n$ be such that $P$ is an isolated zero of $V(f_1, \ldots, f_n)$. Let $N$ be a positive integer, $t$ be an indeterminate, and for each $i,j$, $1 \leq i \leq N$, $1 \leq j \leq m$,  let $f_{i,j} \in \kk[x_1, \ldots, x_n]$ and $\phi_{i,j} \in \kk(t)$ be such that $\phi_{i,j}$ is regular at $0$ with $\phi_{i,j}(0) = 0$. Define $g_j(x,t) := f_j + \sum_{i=1}^N \phi_{i,j}(t)f_{i,j}$, $1 \leq j \leq N$. Then there is an open neighborhood $\tilde U$ of $(P,0)$ on $\kk^n  \times \kk$ such that 
\begin{enumerate}
\item  $\tilde Z:= V(g_1, \ldots, g_n) \cap\tilde  U$ is a curve. 
\item  For generic $\epsilon \in \kk$, the points on $\tilde Z \cap \{t = \epsilon\}$ are isolated zeroes of $V(g_1|_{t= \epsilon}, \ldots, g_n|_{t= \epsilon}) $. \qed
\end{enumerate}
\end{cor}

\section{A technical lemma on intersection multiplicity} \label{technical-section}
In this section we prove \cref{generic-length} which is used in the proofs of \cref{multiplicity-thm,bkk-thm}. For a commutative ring $R$ and a module $M$ over $R$, we write $l_R(M)$ for the {\em length} of $M$ as a module over $M$. Also, for $f_1, \ldots, f_k \in R$, we write $\langle f_1, \ldots, f_k \rangle$ to denote the ideal of $R$ generated by $f_1, \ldots, f_k$. 

\begin{lemma} \label{generic-length}
Let $l \in \zz$, $1 \leq l \leq n$, $m := n-l$, and $f_1, \ldots, f_l \in \kk[x_1, \ldots, x_m, y_1, \ldots, y_l]$ be such that $l_{A_\mmm}(A_\mmm/\langle f_1, \ldots, f_l \rangle)$ is finite, where $A := \kk(x_1, \ldots, x_m)[y_1, \ldots, y_l]$ and $\mmm$ be the maximal ideal of $A$ generated by $y_1, \ldots, y_l$. For $\xi := (\xi_1, \ldots, \xi_m) \in \kk^m$, we write 
\begin{align*}
f_{\xi,j} &:= f_j|_{(x_1, \ldots, x_m) = (\xi_1, \ldots, \xi_m)},\ 1 \leq j \leq l.
\end{align*} 
Then for generic $\xi \in \kk^m$, 
\begin{align*}
l_{B_\nnn}(B_\nnn/\langle f_{\xi,1}, \ldots, f_{\xi,l} \rangle) = l_{A_\mmm}(A_\mmm/\langle f_1 , \ldots, f_l \rangle) 
\end{align*} 
where $B := \kk[y_1, \ldots, y_l]$ and $\nnn$ is the maximal ideal of $B$ generated by $y_1, \ldots, y_l$. 
\end{lemma}

\begin{proof}
We follow the theory of {\em standard bases} from \cite[Chapter 4]{littlesheacox-uag}. Fix a {\em local order} on monomials in $(y_1, \ldots, y_l)$. Then there are precisely $m := l_{A/\mmm}(A_\mmm/\langle f_1, \ldots, f_l \rangle)$ {\em standard monomials} in $(y_1, \ldots, y_l)$ for the ideal $\qqq := \langle f_1, \ldots, f_l \rangle$. We now study how these monomials are computed.\\

Starting from $f_1, \ldots, f_l$, a standard basis $f_1, \ldots, f_l, h_1, \ldots, h_s$ is computed. After computation of $h_1, \ldots, h_j$, the next step is as follows: take every pair of elements in the current basis, compute $S$-polynomial of the pair, and compute the remainder using {\em Mora normal form algorithm} \cite[Section 4.3]{littlesheacox-uag}. Note that computations of $S$-polynomials and remainders consists of (a succession of) {\em reductions}, i.e.\ multiplication by monomials in $(y_1, \ldots, y_l)$ followed by a subtraction. It follows that, once you choose $\xi$ such that {\em none} of the coefficients (which are polynomials in $(x_1, \ldots, x_m)$) of the monomial terms (in $(y_1, \ldots, y_l)$) appearing in any of the $f_i$'s or $h_j$'s vanish, then the standard basis produced by the Mora normal form algorithm for the ideal $\langle f_{\xi,1}, \ldots, f_{\xi,l}\rangle$ in $B_\nnn$ will be the specialization of $f_1, \ldots, f_l, h_1, \ldots, h_s$.  It follows that $\langle f_1, \ldots, f_l \rangle)$ and $\langle f_{\xi,1}, \ldots, f_{\xi,l}\rangle$ have the same standard monomials, and therefore the same length. 
\end{proof}

\section{Counter-examples to \cite[Theorem 1]{rojas-wang}  and \cite[Affine Point Theorem II]{rojas-toric}} %
\label{counter-section}
Let $\scrP_1, \ldots, \scrP_n$ be convex integral polytopes in $(\rr_{\geq 0})^n$. We denote by $\multP_{\kk^n}$ the number (counted with multiplicity) of zeroes on $\kk^n$ of generic $f_1, \ldots, f_n \in \kk[x_1, \ldots, x_n]$ such that the support of each $f_j$ is contained in $\scrP_j$. \cite[Theorem 1]{rojas-wang} states that if the intersection of each $\scrP_j$ with each of the $n$ coordinate hyperplanes is non-empty, then $\multP_{\kk^n}$ is either infinite or equals the mixed volume $\mv(\tilde P_1, \ldots, \tilde P_j)$ of $\tilde P_1, \ldots, \tilde P_n$, where each $\tilde \scrP_j$ is the convex hull of $\scrP_j \cup \{0\}$. Below we present several examples which show that this theorem is not true in general for dimenion $\geq 3$. This theorem is a special case of \cite[Affine Point Theorem II]{rojas-toric}, therefore the examples below also serve as counterexamples to the latter. It seems that the error in these theorems is structural: in \cite[proof of Theorem 7, Page 128]{rojas-wang}, a toric compactification of $\kk^n$ is constructed and the formula given by \cite[Theorem 1]{rojas-wang} is precisely the intersection number of the divisors of zeroes of the corresponding (generic) polynomials. However, in general these divisors intersect non-trivially at infinity (this is precisely what happens in the case of the examples below), so that their intersection number may be greater than the number of roots. \\

In each of the examples below, we define polynomials $f_1, f_2, f_3 \in \kk[x,y,z]$, and for each $j$, write $\scrP_j$ (resp.\ $\tilde \scrP_j$) for the Newton polytope of $f_j$ (resp.\ the convex hull of $\scrP_j \cup \{0\}$). 

\begin{example}
Let $f_1 := ax +by$, $f_2 := a'x + b'y$, $f_3 := pz^kx + q$, where $a,b,a',b',p,q$ are generic elements in $\kk$ and $k \geq 1$. Then $[\scrP_1, \scrP_2, \scrP_3]_{\kk^3} =  [f_1, f_2, f_3]_{\kk^3} = 0$. However, it is straightforward to compute directly that if $\tilde f_1, \tilde f_2, \tilde f_3$ are generic polynomials such that the Newton polytope of each $\tilde f_i$ is $\tilde \scrP_i$, then the number of solutions in $(\kk^*)^3$ of $\tilde f_1, \tilde f_2, \tilde f_3$  is $k$. It then follows from Bernstein's theorem \cite{bern} that $\mv(\tilde \scrP_1, \tilde \scrP_2, \tilde \scrP_3) = k$. 
\end{example}

\begin{example}
Let $f_1 = ax +by + cx^2$, $f_2 = a'x + b'y +c'x^2$, $f_3 = pz^kx + q$, where $a,b,c,a',b',c',p,q$ are generic elements in $\kk$ and $k \geq 1$. Then $[\scrP_1, \scrP_2, \scrP_3]_{\kk^3} =  k < \mv(\tilde \scrP_1, \tilde \scrP_2, \tilde \scrP_3) = 2k$. 
\end{example}

%


\bibliographystyle{alpha}
\bibliography{../../utilities/bibi}

\end{document}